\documentclass[11pt]{amsart}
\usepackage{geometry}                % See geometry.pdf to learn the layout options. There are lots.
\usepackage{graphicx}
\usepackage{amssymb}
\usepackage{epstopdf}
 
\usepackage{amsmath}
\usepackage{color}
\usepackage{hyperref}
\usepackage{pdfsync}
 \usepackage{tikz}
\usepackage[all]{xy}
\xyoption{matrix}
\xyoption{arrow}

\newtheorem{thm}{Theorem}[subsection]
\newtheorem{lem}[thm]{Lemma}
\newtheorem{cor}[thm]{Corollary}
\newtheorem{prop}[thm]{Proposition}

\theoremstyle{definition}
\newtheorem{defn}[thm]{Definition}
\newtheorem{eg}[thm]{Example}

\newtheorem{notation}[thm]{Notation}

\newtheorem{rem}[thm]{Remark}
 
\numberwithin{equation}{subsection}

\newcommand{\op}{^{op}}

\newcommand{\mat}[1]{\ensuremath{
\left[\begin{matrix}#1
\end{matrix}\right]
}}
 
\newcommand{\vs}[1]{\vskip .#1 cm} %enter amount of skip wanted at #1
\newcommand{\noi}{\noindent}

\newcommand{\xrarrow}{\xrightarrow} %right arrow {label on top}
\newcommand{\xlarrow}{\xleftarrow}
\newcommand{\ot}{\leftarrow}
\newcommand{\then}{\Rightarrow}

\newcommand{\into}{\hookrightarrow}
 \newcommand{\onto}{\twoheadrightarrow}
 \newcommand{\cof}{\rightarrowtail}
\newcommand{\toto}{\rightrightarrows}

\newcommand{\smallprod}{\,{\textstyle{\prod}}\,}

 \newcommand{\Vrep}{V\!rep}

\DeclareMathOperator{\coker}{coker}%\newcommand{\coker}{\text{coker}}
\DeclareMathOperator{\Hom}{Hom}%
\DeclareMathOperator{\Ext}{Ext}%
\DeclareMathOperator{\sgn}{sgn}%
\DeclareMathOperator{\End}{End}%
\DeclareMathOperator{\Aut}{Aut}

\DeclareMathOperator{\undim}{\underline{dim}}
 
\newcommand{\field}[1]{\mathbb{#1}}
\newcommand{\ZZ}{\ensuremath{{\field{Z}}}}
\newcommand{\FF}{\ensuremath{{\field{F}}}}

\newcommand{\CC}{\ensuremath{{\field{C}}}}
\newcommand{\RR}{\ensuremath{{\field{R}}}}
\newcommand{\QQ}{\ensuremath{{\field{Q}}}}
\newcommand{\NN}{\ensuremath{{\field{N}}}}
\newcommand{\HH}{\ensuremath{{\field{H}}}}
\newcommand{\kk}{\ensuremath{K}}

\newcommand{\commentout}[1]{}

\newcommand{\cC}{\ensuremath{{\mathcal{C}}}}
\newcommand{\cD}{\ensuremath{{\mathcal{D}}}}

\newcommand{\cM}{\ensuremath{{\mathcal{M}}}}

\newcommand{\cS}{\ensuremath{{\mathcal{S}}}}

\newcommand{\cW}{\ensuremath{{\mathcal{W}}}}

\newcommand{\cX}{\ensuremath{{\mathcal{X}}}}

 \newcommand\brk[1]{\left<#1\right>}

\tikzset{help lines/.style={step=#1cm,very thin, color=gray},
help lines/.default=.5} % draws a grid spaced .#1 cm
\tikzset{thick grid/.style={step=#1cm,thick, color=red},
thick grid/.default=1} % draws a grid spaced .#1 cm

\title{Modulated semi-invariants}
\author{Kiyoshi Igusa}
\address{Department of Mathematics, Brandeis University, Waltham, MA 02454}\email{igusa@brandeis.edu}
 \thanks{The first author is supported by NSA Grant \#H98230-13-1-0247}

\author{Kent Orr}
\address{Department of Mathematics, Indiana University, Bloomington, IN 47405}
\email{korr@indiana.edu}
\thanks{The second author is supported by Simons Foundation Grant \#209082}
 
\author{Gordana Todorov}
\address{Department of Mathematics, Northeastern University, Boston, MA 02115}
\email{g.todorov@neu.edu}
\thanks{The third author is supported by NSF Grant \#DMS-1103813 and \#DMS-0901185}

\author{Jerzy Weyman}
\address{Department of Mathematics, University of Connecticut, Storrs, CT 06269}
\email{jerzy.weyman@uconn.edu}
\thanks{The fourth author is supported by NSF Grant \#DMS-1400740}

\subjclass[2010]{
16G20; 20F55}

\begin{document}

\color{black}

\begin{abstract} We prove the basic properties of determinantal semi-invariants for presentation spaces over any finite dimensional hereditary algebra over any field. The results include the virtual generic decomposition theorem, stability theorem and the $c$-vector theorem, the last says that the $c$-vectors of a cluster tilting object are, up to sign, the determinantal weights of the determinantal semi-invariants defined on the cluster tilting objects. Applications of these theorems are given in several concurrently written papers.
\end{abstract}

\maketitle

%\tableofcontents

%\listoffigures

 %\newpage
 
 \section*{Introduction}

There is a rich theory of semi-invariants for representations of quivers \cite{S91}, \cite{King}, \cite{DW}, \cite{SW}, \cite{SvdB} and its relation to cluster categories and cluster algebras \cite{IOTW1}, \cite{Chindris}, \cite{BHIT}. In this paper, we show how this theory and its relation to cluster algebras can be extended {to finite dimensional hereditary algebras over a field, which in particular include} all modulated acyclic quivers over any field. Furthermore, we prove the relationship between $c$-vectors and semi-invariants.

Over a fixed field $\kk$, a \emph{$\kk$-modulated quiver} is a triple $(Q,\{F_i\}_{i\in Q_0},\{M_{ij}\}_{i\to j\in Q_1})$ where $Q$ is a quiver (directed graph) without oriented cycles, $F_i$ is a finite dimensional division algebra for each vertex $i\in Q_0$ and $M_{ij}$ is an $F_i$-$F_j$ bimodule for every arrow $i\to j$ in $Q_1$. The standard modulation of a simply laced quiver $Q$ is given by taking each $F_i=K$ and each $M_{ij}=K$. A representation $V$ of a modulated quiver with \emph{dimension vector} $\alpha=(\alpha_1,\cdots,\alpha_n)$ consists of {$F_i$-modules $V_i$ of dimension $\alpha_i$ at each vertex $i\in Q_0$ and $F_j$-linear maps $V_i\otimes_{F_i}M_{ij}\to V_j$} for each arrow $i\to j$ in $Q_1$.

We study representation and presentation spaces of modulated quivers. When $Q$ is a simply laced quiver, the standard definition of the {representation space} of $Q$ with dimension vector $\alpha\in \NN^n$ is
\[
	Rep(Q,d)=\prod_{i\to j\in Q_1} \Hom_\kk(K^{\alpha_i},K^{\alpha_j}).
\]
When $\kk$ is algebraically closed, any finite dimensional hereditary algebra is Morita equivalent to the path algebra $KQ$ of a quiver $Q$. Choosing an element of $Rep(Q,\alpha)$ is equivalent to choosing a $KQ$-module structure on the $KQ_0$-module $\bigoplus_i K^{\alpha_i}$.

Over the modulated quiver $(Q,\{F_i\}_{i\in Q_0},\{M_{ij}\}_{i\to j\in Q_1})$ the representation space for dimension vector $\alpha\in \NN^n$ is
\[
	Rep(Q,\{F_i\}_{i\in Q_0},\{M_{ij}\}_{i\to j\in Q_1},\alpha)=\prod_{i\to j\in Q_1} \Hom_{F_j}(M_{ij}^{\alpha_i},F_j^{\alpha_j}).
\]
Each element of $Rep(Q,\{F_i\}_{i\in Q_0},\{M_{ij}\}_{i\to j\in Q_1},\alpha)$ gives the right $\prod_{i=1}^n F_i$-module $\bigoplus_{i=1}^n F_i^{\alpha_i}$ the structure of a right module over the tensor algebra of $(Q,\{F_i\}_{i\in Q_0},\{M_{ij}\}_{i\to j\in Q_1})$. In each case, the representation space is an affine space over $\kk$.

In this paper we deal with arbitrary finite dimensional hereditary algebras over any field, not necessarily tensor algebras of modulated quivers. Notice that, if $\kk$ is not perfect, there may be finite dimensional hereditary algebras over $\kk$ which are not Morita equivalent to the tensor algebras of modulated quivers. (Appendix A, Sec \ref{ss: Appendix}.) 

For an arbitrary finite dimension hereditary algebra $\Lambda$ we define (in \ref{def: representation space}) the representation space $Rep(\Lambda,\alpha)$ to be a certain subspace of the space $\Hom_\Lambda(rad\,P(\alpha),P(\alpha))$, where $P(\alpha)$ denotes $\bigoplus_i P_i^{\alpha_i}$. This is isomorphic to $Rep(Q,\{F_i\}_{i\in Q_0},\{M_{ij}\}_{i\to j\in Q_1})$ in the modulated case. We identify each element $f\in Rep(\Lambda,\alpha)$ with the $\Lambda$-module which is the cokernel of the homomorphism $f:rad\,P(\alpha)\to P(\alpha)$.

We consider $Rep(\Lambda,\alpha)$ as an affine space over $\kk$. {At the beginning we assume that $\kk$ is infinite so that nonempty open subsets of this space are dense. (We extend to arbitrary fields later, in Section \ref{sec: extend to finite K}) }The first theorem of this paper is Theorem \ref{thm: generic decomposition thm 1}: If there exists a $\Lambda$-module $M$ which is \emph{rigid} meaning $\Ext^1_\Lambda(M,M)=0$ with $\undim M=\alpha$, then the elements of $Rep(\Lambda,\alpha)$ which are isomorphic to $M$ form an open dense subset. We call this the \emph{generic representation} of dimension $\alpha$ and denote it by $M_\alpha$. If $M_\alpha$ is indecomposable then $\alpha$ is a \emph{real Schur root} of $\Lambda$. As a consequence we have: 

{
\begin{thm} [Generic Decomposition Theorem \ref{cor: generic decomposition thm 2}] Let $\Lambda$ be a finite dimensional hereditary algebra over an infinite field, and let $\beta_1,\cdots,\beta_k$ be real Schur roots so that $M_{\beta_i}$ do not extend each other. Then for any nonnegative integer linear combination $\alpha=\sum_{i=1}^k n_i\beta_i$, the generic representation in $Rep(\Lambda,\alpha)$ is isomorphic to $\bigoplus_{i=1}^k M_{\beta_i}^{n_i}$.
\end{thm}

Representation spaces are defined for $\alpha\in\NN^n$.}
Next, we generalize the construction to arbitrary integer vectors $\alpha\in\ZZ^n$ by constructing presentation spaces and considering their direct limit which we call \emph{virtual representation space}. We choose vectors $\gamma_0,\gamma_1\in\NN^n$ so that $\undim P(\gamma_0)-\undim P(\gamma_1)=\alpha$. We call $\Hom_\Lambda(P(\gamma_1),P(\gamma_0))$ a \emph{presentation space} and denote it by $Pres_\Lambda(\gamma_1,\gamma_0)$ and view it as a generalization of $Rep(\Lambda,\alpha)$. However, there are an infinite number of choices for $\gamma_0,\gamma_1$ for each $\alpha\in \ZZ^n$. To define a single space for each $\alpha\in\ZZ^n$ which contains all of these presentation spaces, we take their direct limit (colimit):
\[
	\Vrep(\Lambda,\alpha):=colim\, Pres_\Lambda(\gamma_1,\gamma_0)
\]
where the colimit is over all pairs $\gamma_0,\gamma_1$ so that $\undim P(\gamma_0)-\undim P(\gamma_1)=\alpha$. Representatives of $\Vrep(\Lambda,\alpha)$ are presentations $p:P(\gamma_1)\to P(\gamma_0)$ which we denote $P(\gamma_\ast)$. The next theorem in this paper is:

{
\begin{thm}[Virtual Generic Decomposition Theorem \ref{thm 2.3.11: virtual generic decomposition theorem}] Let $\{\beta_i\}$ be a {partial cluster tilting set} (Definition \ref{def: partial cluster tilting set}). Let $\alpha=\sum r_i\beta_i\in\ZZ^n$ where $r_i\in \QQ$. Then $r_i\in \ZZ$ and the general virtual representation in $\Vrep(\Lambda,\alpha)$ is isomorphic to $\bigoplus_i P(\gamma^i_{\ast})^{r_i}$ where $P(\gamma^i_{\ast})$ are rigid objects in $\Vrep(\Lambda,\beta_i)$. In other words, the set of all elements of $\Vrep(\Lambda,\alpha)$ isomorphic to $\bigoplus_i P(\gamma^i_{\ast})^{r_i}$ is open and dense.
\end{thm}
}

The groups $\Aut_\Lambda(P(\gamma_0))$, $\Aut_\Lambda(P(\gamma_1))$ act on presentation space $Pres_\Lambda(\gamma_1,\gamma_0)$. A \emph{semi-invariant} on $Pres_\Lambda(\gamma_1,\gamma_0)$ is a polynomial function $\sigma: Pres_\Lambda(\gamma_1,\gamma_0)\to \kk$ so that, for any $(g_0,g_1)\in\Aut_\Lambda(P(\gamma_0))\times\Aut_\Lambda(P(\gamma_1))^{op}$ and $f:P(\gamma_1)\to P(\gamma_0)$ we have $\sigma (g_0fg_1)=\chi_0(g_0)\sigma(f)\chi_1(g_1)$ where $\chi_0,\chi_1$ are characters $\Aut_\Lambda(P(\gamma_s))\to\kk^\ast$ for $s=0,1$ where by \emph{character} we mean a regular (polynomial) function which is a homomorphism of groups. Every group homomorphism $\Aut_\Lambda(P(\alpha))\to\kk^\ast$ factors through the group $\prod_{i=1}^n \Aut_\Lambda(P_i^{\alpha_i})=\prod_{i=1}^n GL(\alpha_i,F_i)$. Since $\sigma$ is defined on the affine space $Pres_\Lambda(\gamma_1,\gamma_0)$, these characters extend to the endomorphism rings of $P(\gamma_0)$, $P(\gamma_1)$ (by $g\mapsto \sigma(gf)/\sigma(f)$ for a fixed $f\in Pres_\Lambda(\gamma_1,\gamma_0)$ on which $\sigma(f)\neq0$.) In Appendix B Theorem \ref{thm relating determinant and reduced norm} we show that every character $\End_F(F^m)\to K$ is a power of the ``reduced norm'' {(and thus a fractional power of the \emph{determinantal character} given by taking the determinant of an $F$-endomorphism of $F^m$ considered as a linear map over $\kk$.} {See Definition \ref{def: reduced norm}}). Therefore, the characters associated to any semi-invariant on the presentation space $Pres_\Lambda(\gamma_1,\gamma_0)$ are nonnegative integer powers of the reduced norm for each division algebra $F_i$. This gives a vector weight in $\NN^n$. The weights coming from $P(\gamma_0)$ and $P(\gamma_1)$ are equal when defined.

In this paper we do not use the reduced norm weights. We use \emph{determinantal (det-) weights}. The coefficients of the det-weights are in general fractions. They are integers if and only if the characters are powers of the determinantal character. We also consider only certain semi-invariants: the \emph{determinantal semi-invariants} $\sigma_\beta$ which have \emph{determinantal weight} $\beta\in \NN^n$. These semi-invariants are given on any presentation $f:P(\gamma_1)\to P(\gamma_0)$ by $\sigma_\beta(f)=\det_\kk \Hom_\Lambda(f,M_\beta)$. These semi-invariants are clearly compatible with stabilization and therefore define semi-invariants on $\Vrep(\alpha)$ in the case when $\Hom_\Lambda(f,M_\beta)$ is an isomorphism.  We call the set of such $\alpha\in\ZZ^n$ the (integral) \emph{domain of the semi-invariant of det-weight $\beta$} and denote it by $D_\ZZ(\beta)$.

In Section \ref{ss: virtual stability theorem} %{Virtual stability theorem}
we prove the virtual stability theorem which states that these domains of semi-invariants are given by ``stability conditions''.

{
\begin{thm}[Virtual Stability Theorem \ref{thm 3.1.1: virtual stability theorem}]
Let $\Lambda$ be a finite dimensional hereditary algebra over a field with $n$ simple modules. Let $\alpha\in \mathbb Z^n$ and $\beta$ a real Schur root. Then, the following are equivalent:
\begin{enumerate}
\item  There exists a morphism of projective modules $f\!:\!P\to Q$ so that $\  \undim Q-\undim P=\alpha$ and $f$ induces an isomorphism
\[
f^\ast:\Hom_\Lambda(Q,M_\beta)\xrightarrow{\approx} \Hom_\Lambda(P,M_\beta).
\]
\item Stability conditions for $\alpha$ and $\beta$ hold:
 $\brk{\alpha,\beta}=0$ and
 $\brk{\alpha,\beta'}\le 0$ for all real Schur subroots $\beta'\subseteq \beta$ where $\brk{\cdot,\cdot}$ is the Euler-Ringel form defined in Proposition \ref{Euler-Ringel form}.
\item There is a semi-invariant of {det-weight} $\beta$ on the presentation space $\Hom_\Lambda(P,Q)$.\end{enumerate}
\end{thm}
}

We prove this first in the case when $\beta$ is sincere  and $K$ is infinite (subsection \ref{ss: sincere case}), then for any $\beta$ (subsection \ref{ss: non-sincere case}), then for any field $\kk$ (subsection \ref{sec: extend to finite K}).

In Section \ref{sec: c-vector theorem} we prove the $c$-vector theorem below which states that, up to a precisely given sign, the {det-weight}s $\beta_i$ associated to a cluster tilting object are equal to the $c$-vectors associated to the cluster tilting object.

\begin{thm}[c-vector Theorem \ref{c-vector thm}]
Let $T=\bigoplus_{i=1}^n T_i$ be a cluster tilting object for $\Lambda$ and let $f_i=\dim_\kk\End_\Lambda(T_i)$.
\begin{enumerate}
\item There exist unique real Schur roots $\beta_1,\cdots,\beta_n$ so that $\undim T_i\in D(\beta_j)$ for $i\neq j$.
\item The $c$-vectors associated to the cluster tilting object are equal to $\beta_i$ up to sign: $c_i=\pm\beta_i$. More precisely,
$
	c_i=(-f_i/{\brk{\undim T_i,\beta_i}})\beta_i
$.
\item $\brk{\undim T_i,c_i}=-f_i$ for each $i=1,\dots,n$. 
\end{enumerate}
\end{thm}
 
The $c$-vector theorem implies the sign coherence of $c$-vectors since weight vectors always lie in $\NN^n$. Sign coherence of $c$-vectors has been shown in many cases \cite{DWZ1}, \cite{Pl} and in general in \cite{GHKK}. We end with an example of a semi-invariant picture (Figure \ref{Fig C3}) illustrating some of our theorems and the important properties of the picture used in other papers \cite{BHIT}, \cite{IOTW4}, \cite{IT13}, \cite{IT14}. 

There are also two appendices. {Appendix A (Sec \ref{ss: Appendix})} discusses when an hereditary algebra is Morita equivalent to the tensor algebra of a modulated quiver and gives an example when this is not true. {Appendix B (Sec \ref{ss: Appendix B})} reviews the basic properties of reduced norm and shows that every character $M_k(D)\to\kk$ is a power of the reduced norm. Thus every semi-invariant on presentation space has a weight vector $w\in\NN^n$ so that, under automorphisms of $P(\gamma_0),P(\gamma_1)$, the semi-invariant changes by the product of $\overline n_i^{w_i}$ where $\overline n_i$ are the reduced norms of the $GL(F_i)$ blocks of the automorphisms. We call $w$ the \emph{reduced weight}. We define the \emph{reduced norm semi-invariants} $\overline\sigma_\beta$ and show that their reduced weights $\overline\beta$ are the $c$-vectors associated to a \emph{reduced exchange matrix} $\overline B_\Lambda=ZB_\Lambda Z^{-1}$.

%\newpage
%%%%%%%%%%%%%%%%%%%%%%%%%%
%
%                Section  {Basic notation}
%
%%%%%%%%%%%%%%%%%%%%%%%%%%

{
\section{Basic definitions}

In this paper $\Lambda$ will be a basic finite dimensional hereditary algebra over any field $\kk$. Basic means that, as a right module over itself, the summands of $\Lambda$ are pairwise nonisomorphic. {Finite dimensional hereditary algebras share many important properties with the tensor algebra of their associated modulated quiver. For example they have the same Euler matrix, the same real Schur roots, the same semi-invariant domains and the same $c$-vectors, which are the topics we study in this paper. So, we begin with modulated quivers which are slightly easier to understand than the general case. Then we extend the definitions of presentation spaces and semi-invariants on presentation spaces to general finite dimensional hereditary algebras. }

\subsection{Modulated quivers}\label{sec1.1}

By a \emph{modulated quiver} $(Q,\cM)$ over $\kk$ we mean a finite quiver $Q$ without oriented cycles together with
\begin{enumerate}
\item a finite dimensional division algebra $F_i$ over $\kk$ at each vertex $i$ of $Q$ and
\item a finite dimensional $F_i$-$F_j$ bimodule $M_{ij}$ for every arrow $i\to j$ in $Q$.
\end{enumerate}
The absence of multiple arrows is not a restriction. If we have a quiver with more than one arrow $i\to j$ then these are combined into one arrow with the associated bimodule being the direct sum of the bimodules on the original arrows. For example, the quiver $1\toto 2$ is equivalent to $1\to 2$ with bimodule $\kk^2$ on the arrow.
}

\begin{defn}\label{def: valuation of a modulated quiver} The \emph{valuation} on $Q=(Q_0,Q_1)$ given by the modulation $\cM$ is defined to be the sequence of positive integers $f_i, i\in Q_0$ and $d_{ij},d_{ji}$ for $i\to j$ in $Q_1$ given as follows.% set of numbers.
\begin{enumerate}
\item $f_i=\dim_\kk F_i$ for each $i\in Q_0$.
\item $d_{ij}=\dim_{F_j}M_{ij}$, $d_{ji}=\dim_{F_i}M_{ij}$ for each $i\to j$ in $Q_1$.
\end{enumerate}
\end{defn}

\begin{prop}\cite{DR}
For any sequence of positive integers $f_i, i\in Q_0$ and pairs of positive integers $(d_{ij},d_{ji})$ for every arrow $i\to j$ in $Q_1$ there exists a modulation of $Q$ having these numbers as valuation if and only if $d_{ij}f_j=f_id_{ji}$ for all $i,j$.
\end{prop}

\begin{proof} Let $\kk$ be any finite field, $\kk=\FF_q$. For each $i$ let $F_i$ be the field with $q^{f_i}$ elements. For each arrow $i\to j$, let $M_{ij}$ be the field with $q^{d_{ij}f_j}$ elements.\end{proof}

{
A (finite dimensional) \emph{representation} $V$ of a modulated quiver $Q$ is given by
\begin{enumerate}
\item a finite dimensional $F_i$-vector space $V_i$ at each vertex $i$ in $Q_0$ and
\item an $F_j$ linear map $V_i\otimes_{F_i}M_{ij}\to V_j$ for every arrow $i\to j$ in $Q_1$.
\end{enumerate}}
A (finite dimensional) representation of a modulated quiver is the same as a finite dimensional module over the \emph{tensor algebra} $T(Q,\cM)$ of $(Q,\cM)$ which is defined to be the direct sum of all \emph{tensor paths}:
\[
	T(Q,\cM):=\bigoplus M_{j_0,j_1}\otimes_{F_{j_1}} M_{j_1,j_2}\otimes_{F_{j_2}} \cdots \otimes_{F_{j_{r-1}}} M_{j_{r-1},j_r}\ ,
\]
including paths of length zero ($F_j$), with multiplication given by concatenation of paths. Since the quiver $Q$ has no oriented cycles this algebra is a finite dimensional hereditary algebra over $\kk$. 

\begin{defn}\label{def: associated modulated quiver}
Given a finite dimensional hereditary algebra $\Lambda$ over a field $\kk$, the \emph{associated modulated quiver} $(Q,\cM)$ is given as follows. Fix an ordering of the simple $\Lambda$-modules $S_1,\cdots,S_n$. Let $P_i$ be the projective cover of $S_i$.
\begin{enumerate}
\item Let $Q$ be the quiver with $Q_0=\{1,\cdots,n\}$ and arrows $i\to j$ when $\Ext^1_\Lambda(S_i,S_j)\neq 0$.
\item Let $F_i=\End_\Lambda(S_i)$ for each $i\in Q_0$.
\item For each $i\to j$ in $Q_1$ let $M_{ij}=\Hom_\Lambda(P_j,rP_i/r^2P_i)$.
\end{enumerate}
\end{defn}

{
There are examples of hereditary algebras which are not equivalent to their associated modulated quiver. We discuss these pathologies in {Appendix A (Sec \ref{ss: Appendix})}. Our results are general and hence include these pathological cases as well.%Results in the paper are done in general, hence including these pathological cases.

\subsection{Euler matrix}\label{ss: Euler matrix}
The underlying valued quiver of an hereditary algebra $\Lambda$ is the valued quiver of its associated modulated quiver. However, it is useful to go directly from $\Lambda$ to its underlying valued quiver, i.e., $f_i=\dim_\kk F_i$ where $F_i=\End_\Lambda(S_i)$ for each $i\in Q_0$, $d_{ij}=\dim_{F_j}\Hom_\Lambda(P_j,rP_i/r^2P_i)$, $d_{ji}=\dim_{F_i}\Hom_\Lambda(P_j,rP_i/r^2P_i)$ for each $i\to j$ in $Q_1$.}

The \emph{dimension vector} $\undim V$ of $V$ is defined to be the vector whose $i$-th coordinate is $\dim_{F_i}V_i$. We also have the dimension vector over $\kk$ given by
\[
	\undim_\kk V=D\undim V
\]
where $D$ is the diagonal matrix with diagonal entries $f_i=\dim_\kk F_i$. Let $E,L,R$ be the $n\times n$ matrices with $ij$ entries
\[
	E_{ij}=\dim_{\kk}\Hom_\Lambda(S_i,S_j)-\dim_{\kk}\Ext_\Lambda^1(S_i,S_j)
\]
\[
	L_{ij}=\dim_{F_j}\Hom_\Lambda(S_i,S_j)-\dim_{F_j}\Ext_\Lambda^1(S_i,S_j)
\]
\[
	R_{ij}=\dim_{F_i}\Hom_\Lambda(S_i,S_j)-\dim_{F_i}\Ext^1_\Lambda(S_i,S_j)
\]
Then $E_{ij}=L_{ij}f_j=f_iR_{ij}$ or, equivalently,
\[	E=LD=DR\, 	.	\]
We call $E$ the \emph{Euler matrix} of $\Lambda$, $L$ the \emph{left Euler matrix} of $\Lambda$ and $R$ the \emph{right Euler matrix} of $\Lambda$. The underlying valued quiver of $\Lambda$ has vertices $1\le i\le n$ corresponding to the simple modules $S_i$ and an arrow $i\to j$ if $E_{ij}<0$ with valuations $f_i$ on vertex $i$ and $(d_{ij},d_{ji})=(-L_{ij},-R_{ij})$ for every arrow $i\to j$.

{
\begin{eg} For example, for the valued quiver  $
\xymatrixrowsep{10pt}\xymatrixcolsep{60pt}
\xymatrix{%begin xy matrix
\,_{f_2=3}\bullet \ar[r]^{(d_{21},d_{12})=(3,2)}&\bullet_{f_1=2}
	}%end xy matrix
$  we have:
\[
	LD=\mat{1 & 0 \\ -3 & 1}\mat{2 & 0 \\ 0 & 3}=E=\mat{2 & 0 \\ -6 & 3}=\mat{2 & 0 \\ 0 & 3}\mat{1 & 0 \\ -2 & 1}=DR.
\]
\end{eg}}

The matrices $L,R$ are always unimodular and $\det\,E=\det\,D$ is always the product of the dimensions $f_i$ of $F_i=\End_\Lambda(P_i)=\End_\Lambda(S_i)$.

{We also use, in the subsection on $c$-vectors (sec \ref{ss: def of c-vectors})}, the \emph{exchange matrix} $B=L^t-R$. Since $DB=DL^t-DR=E^t-E$, $DB$ is always skew symmetric. In the example this is:
\[
	B=L^t-R=\mat{0 & -3 \\ 2 & 0},\ DB=E^t-E=\mat{0 & -6 \\ 6 & 0}.
\]

{
\begin{prop}\label{Euler-Ringel form} Let $\left<\cdot,\cdot\right>:\RR^n\times\RR^n\to\RR$ be the Euler-Ringel pairing given by $\left<x,y\right>=x^t Ey$. Then, for any two $\Lambda$-modules $M,N$ we have:
\[
	\left<\undim M,\undim N\right>=\dim_K \Hom_\Lambda(M,N)-\dim_K\Ext^1_\Lambda(M,N)\,.
\]
\end{prop}}

For example, pairing $(P_1,\cdots,P_n)$ with $(S_1,\cdots,S_n)$ gives
\[
	\left<\undim P_i,\undim S_j\right>=\dim_K \Hom_\Lambda(P_i,S_j)=f_i\delta_{ij}.
\]
This equation can be written as $PEI_n=D$ where the $i$-th row of the matrix $P$ is $\undim P_i$. Furthermore $PE=D$ and $E=LD$ imply that $P=L^{-1}$.

{
\begin{prop}\label{prop: pairing is divisible by fM}
Suppose that $\End_\Lambda(M)$ is a division algebra. Then $\left<\undim M,\undim N\right>$ and $\left<\undim N,\undim M\right>$ are divisible by $f_M=\dim_\kk \End_\Lambda(M)$.
\end{prop}

\begin{proof}
$
	\left<\undim M,\undim N\right>=\dim_\kk\Hom_\Lambda(M,N)-\dim_\kk\Ext^1_\Lambda(M,N)
$
which is divisible by $f_M$ since $\Hom_\Lambda(M,N)$ and $\Ext^1_\Lambda(M,N)$ are vector spaces over $\End_\Lambda(M)$.
\end{proof}}

 \begin{notation}
 For each $\alpha\in\NN^n$ we use the notation $P(\alpha)=\bigoplus_iP_i^{\alpha_i}$. For example, $\Lambda=P(1,1,\cdots,1)$ if $\Lambda$ is basic.
 \end{notation}

%For any $\gamma=(\gamma_i)\in\NN^n$ we will use the notation $P(\gamma)=\bigoplus_{i=1}^n P_i^{\gamma_i}$.
With this notation, we have the following. Suppose $\undim M=\beta$ and $\gamma\in\NN^n$. Then
\[
	\brk{\undim P(\gamma),\undim M}=\sum_{i=1}^n  \gamma_i\dim_\kk \Hom_\Lambda(P_i,M)=\sum_{i=1}^n  \gamma_if_i\beta_i\,.
\]

\subsection{Exceptional sequences}\label{sec:exc seq}

We review the definition and basic properties of exceptional sequences. See \cite{C-B93}, \cite{Ri} for details.

\begin{defn}
Let $\Lambda$ be a finite dimensional hereditary algebra over any field $\kk$. Then a $\Lambda$-module $M$ is called \emph{exceptional} if $\Ext^1_\Lambda(M,M)=0$ and $\End_\Lambda(M)$ is a division algebra. In particular $M$ is indecomposable. A sequence of modules $(X_1,\cdots,X_k)$ is called an \emph{exceptional sequence} if all objects are exceptional and
\[
	\Hom_\Lambda(X_j,X_i)=\Ext^1_\Lambda(X_j,X_i)=0\ \text{for all }j>i.
\]
An exceptional sequence is called \emph{complete} if it is of maximal length. {By \ref{prop: properties of exceptional sequences}(1) below, the maximal length is equal to the number of nonisomorphic simple modules.}

\end{defn}
The following are standard examples of complete exceptional sequences.

\begin{prop} \label{exceptional sequences} Let $\Lambda$ be a finite dimensional hereditary algebra with admissible order given by $\Hom_{\Lambda}(P_j,P_i) =0$ for all $j>i$.
Then:
\begin{enumerate}
\item The simple modules $(S_n, S_{n-1},\cdots,S_1)$ form an exceptional sequence.
\item The projective modules $(P_1, P_2, \cdots,P_n)$ form an exceptional sequence.
\item The injective modules $(I_1, I_{2},\cdots,I_n)$ form an exceptional sequence where $I_i$ is the injective envelope of the simple module $S_i$ for $i=1,\cdots,n$.
\end{enumerate}
\end{prop}

{Exceptional sequences have many nice properties. We list here only those properties that we use to prove the main theorems in this paper.}

\begin{prop}\label{prop: properties of exceptional sequences}
Let $n$ be the number of simple $\Lambda$ modules.
\begin{enumerate}
\item Exceptional sequences are complete if and only if they have $n$ objects.
\item Every exceptional sequence can be extended to a complete exceptional sequence.
\item If $(X_1,\cdots,X_n)$ is an exceptional sequence, $\{\undim X_i\}$ generates $\ZZ^n$ as a $\ZZ$-module.
\item Given an exceptional sequence $(X_1,\cdots,X_{n-1})$ of length $n-1$ and any $j=1,\cdots,n$, there are modules $Y_j$, unique up to isomorphism, so that
\[
	(X_1,\cdots,X_{j-1},Y_j,X_j,\cdots,X_{n-1})
\]
is a (complete) exceptional sequence.
\item {$\End_\Lambda(Y_j)\cong \End_\Lambda(Y_{j'})$ for all $j,j'$ in (4) above.}
\item Let $(X_1,\cdots,X_n)$ be an exceptional sequence. Then:\\
If $X_n$ is non-projective, then  $(\tau X_{n}, X_1,\cdots,X_{n-1})$ is an exceptional sequence.  \\
If $X_n=P_k$ is projective, then  $(I_{k}, X_1,\cdots,X_{n-1})$ is an exceptional sequence. 
\end{enumerate}
\end{prop}

Condition (4) implies that there is an action of the braid group on $n$ strands on the set of (isomorphism classes of) complete exceptional sequences. {For example, the braid generator $\sigma_i$ which moves the $i$-th strand over the $(i+1)$-st strand acts by:
\begin{equation}\label{eq: braid action}
	\sigma_i(X_1,\cdots,X_n)=(X_1,\cdots,X_{i-1},X_{i+1}^\ast,X_i,X_{i+2},\cdots,X_n)
\end{equation}
where, by property (4), $X_{i+1}^\ast$ is the unique exceptional module which fits in the indicated location in the exceptional sequence given the other objects. One of the very important theorems about exceptional sequences used in this paper is the following result proved in \cite{C-B93} in the algebraically closed case and \cite{Ri} in general.}

\begin{thm}[Crawley-Boevey, Ringel]\label{Thm: braid group acts transitively on exc seqs}
The braid group on $n$ strands acts transitively on the set of all complete exceptional sequences.
\end{thm}

{In the case when $\kk$ is algebraically closed, or, more generally when $\Lambda=\kk Q$ is the path algebra of a simply laced quiver without oriented cycles, it follows that $\End_\Lambda(M)=\kk$ for every exceptional $\Lambda$-module $M$.} In general, the endomorphism rings of the $X_i$ are division algebras which remain the same after any braid move by Proposition \ref{prop: properties of exceptional sequences}(5). So, Theorem \ref{Thm: braid group acts transitively on exc seqs} implies the following.

\begin{cor}\label{cor: End(Xi) is End(S_i)}
For any exceptional sequence $(X_1,\cdots,X_n)$, there is a permutation $\pi$ of $n$ so that $\End_\Lambda(X_i)\cong \End_\Lambda(S_{\pi(i)})$ for all $i$.
\end{cor}

Another important consequence of this theorem is the following.

\begin{prop}
Suppose that $(\beta_1,\cdots,\beta_n)$ is the set of dimension vectors of a complete exceptional sequence. Then each vector $\beta_j$ is uniquely detemined by the other vectors together with the requirements that
\begin{enumerate}
\item $\brk{\beta_k,\beta_i}=0$ for $k>i$.
\item The vectors $\beta_i$ additively generate $\ZZ^n$.
\end{enumerate}
\end{prop}

Note that these conditions depend only on $n$ and the Euler form $\brk{\cdot,\cdot}$. {This implies that there is an action of the braid group on $n$-strands on the set of dimension vectors of exceptional sequences which is given by}
\[
	\sigma_i(\beta_1,\cdots,\beta_n)=(\beta_1,\cdots,\beta_{i-1},\beta_{i+1}^\ast,\beta_i,\beta_{i+2},\cdots,\beta_n)
\]
where $\beta_{i+1}^\ast$ is the unique vector in $\NN^n$ satisfying the conditions of the proposition above.

\begin{cor}\label{cor: characterization of real Schur roots}
An exceptional $\Lambda$-module is uniquely determined up to isomorphism by its dimension vector. Furthermore, a vector $\beta\in\NN^n$ is the dimension vector of an exceptional module if and only if it appears in $\sigma(\alpha_1,\cdots,\alpha_n)$ for some element $\sigma$ of the braid group on $n$ strands where $\alpha_i=\undim S_i$ are the unit vectors in $\ZZ^n$. In particular, the set of such dimension vectors depends only on the underlying valued quiver of $\Lambda$.
\end{cor}

{This is a restatement of another important theorem of Schofield \cite{S91}: The dimension vectors of the exceptional $\Lambda$-modules are the real Schur roots of $\Lambda$. In this paper we will not need the original definition of a real Schur root \cite{K}. Following \cite{S91}, \cite{S92}, we use the characterizing property of real Schur roots given by the above corollary as the definition. }

\begin{defn}\cite{S92}
A \emph{real Schur root} of $\Lambda$ is a vector $\beta\in\NN^n$ with the property that there exists an exceptional module $M_\beta$ with $\undim M_\beta=\beta$.
\end{defn}

As a special case of Corollary \ref{cor: characterization of real Schur roots} above we have the following.

\begin{cor}\label{cor: real schur roots are same for associated modulated quiver}
The real Schur roots of an hereditary algebra are the same as those of the associated modulated quiver.
\end{cor}

\subsection{Extension to arbitrary fields}\label{sec1.4}

The main results of this paper hold for arbitrary fields. The proofs are first done for infinite fields and they are extended to all fields using the following arguments.

Recall that if $\kk$ is a finite field and $F$ is a finite field extension of $\kk$ then
\[
	F\otimes_\kk\kk(t)\cong F(t)
\]
{is a finite field extension of $\kk(t)$. For any $\kk$-algebra $\Lambda$ we will use the notation $\Lambda(t)$ to denote $\Lambda\otimes_\kk\kk(t)$. This is a finite dimensional hereditary algebra over $\kk(t)$. For any $\Lambda$-module $M$, let $M(t)$ denote the $\Lambda(t)$-module $M\otimes_\kk\kk(t)$. Recall that the dimension vector of $M(t)$ as a $\Lambda(t)$-module is the vector whose $i$-th coordinate is $\dim_{F_i(t)}\Hom_{\Lambda(t)}(P_i(t),M(t))$.}

\begin{thm}\label{thm: extension to finite fields}
Let $\Lambda$ be a finite dimensional hereditary algebra over a finite field $\kk$ and {let $M$ be} an exceptional $\Lambda$-module with $\undim M=\beta$. Then, $\Lambda(t)$ is a finite dimensional hereditary algebra over $\kk(t)$ and $M(t)$ is an exceptional $\Lambda(t)$-module with {the same} dimension vector $\beta$. Furthermore, every exceptional $\Lambda(t)$ module is isomorphic to $M(t)$ for a unique exceptional $\Lambda$-module $M$.
\end{thm}

\begin{proof}
Since tensor product {over $\kk$} with $\kk(t)$ is exact we get:
%\begin{equation}\label{eq: Hom otimes K(t)}
\[	\Hom_\Lambda(X,Y)\otimes_\kk\kk(t)\cong \Hom_{\Lambda(t)}(X(t),Y(t))
\]%\end{equation}
\[%\begin{equation}\label{eq: Ext otimes K(t)}
	\Ext^1_\Lambda(X,Y)\otimes_\kk\kk(t)\cong \Ext^1_{\Lambda(t)}(X(t),Y(t))
\]%\end{equation}
for any two $\Lambda$-modules $X,Y$. Therefore, a $\Lambda$-module $M$ is exceptional if and only if $M(t)=M\otimes_\kk\kk(t)$ is an exceptional $\Lambda(t)$-module. The rest follows from Corollary \ref{cor: characterization of real Schur roots}.
\end{proof}

% Section:

%%\newpage
%%%%%%%%%%%%%%%%%%%%%%%%%%
%
%                Section  XXXXX
%
%%%%%%%%%%%%%%%%%%%%%%%%%%

\section{Virtual representations and semi-invariants}

Throughout this section we consider representations of a finite dimensional hereditary algebra $\Lambda$ over an infinite field $\kk$. 

\subsection{Generic decomposition theorem}\label{real schur roots}

We first recall the Happel-Ringel Lemma \cite{HR}.

\begin{lem}[Happel-Ringel]
Suppose that $T_1,T_2$ are indecomposable modules over an hereditary algebra $\Lambda$ so that $\Ext_\Lambda^1(T_1,T_2)=0$. Then any nonzero morphism $T_2\to T_1$ is either a monomorphism or an epimorphism.
\end{lem}

An important consequence of this lemma is the following observation of Schofield.

\begin{lem}[Schofield]\label{lem: Schofield's observation}
Suppose that $\{M_i\}$ is a set of nonisomorphic indecomposable modules so that $\Ext_\Lambda^1(M_i,M_j)=0$ for all $i,j$. Then $\{M_i\}$ can be renumbered so that $\Hom_\Lambda(M_j,M_i)=0$ for $j>i$, i.e., so that it forms an exceptional sequence.
\end{lem}

\begin{proof}
If not, there is an oriented cycle of nonzero morphisms between the $M_i$ which are monomorphisms or epimorphisms. In this oriented cycle there must be an epimorphism followed by a monomorphism, hence the composition is neither a monomorphism nor an epimorphism which contradicts the Happel-Ringel Lemma.
\end{proof}

 We give a definition of the representation space $Rep(\Lambda,\alpha)$ for any finite dimensional hereditary algebra $\Lambda$ and $\alpha\in\NN^n$ which will agree with the classical definition of $Rep(\Lambda,\alpha)$ when $\Lambda=KQ$, the path algebra of a quiver, or when $\Lambda=T(Q,\cM)$, the tensor algebra of any modulated quiver. In the definition we need to choose a decomposition of the radical of each projective module, however, if different decompositions are chosen, there is an isomorphism of the affine varieties which induce isomorphisms on cokernel modules.

\begin{defn}\label{def: representation space} For each $\alpha\in\NN^n$ consider the following space
\begin{equation}\label{eq: Rep0(L,a)}
	H(\Lambda,\alpha):=\smallprod_{j\in Q_0}\Hom_\Lambda(\textstyle\bigoplus_{i} P_i^{\alpha_jd_{ji}},P_i^{\alpha_i})\subseteq \Hom_\Lambda(\textstyle\bigoplus_{i,j} P_i^{\alpha_jd_{ji}},P(\alpha))
\end{equation}
where $d_{ji}=\dim_{F_i}\Ext^1_\Lambda(S_j,S_i)$ for $i\neq j$. For each choice of decompositions
\begin{equation}\label{eq: decomposition of rad Pi}
	\{\varphi_j:\textstyle\bigoplus_{i\in Q_0} P_i^{d_{ji}}\cong rad\,P_j\}_{j\in Q_0}
\end{equation}
we get an isomorphism $\varphi_\alpha=\bigoplus \varphi_j^{\alpha_j}:\bigoplus_{ i,j} P_i^{d_{ji}}\cong rad\,P(\alpha)$ and we define the \emph{representation space} $Rep_\varphi(\Lambda,\alpha)$ to be
\[
	Rep_\varphi(\Lambda,\alpha):=\{f:rad\,P(\alpha)\to P(\alpha)\,|\, f\circ \varphi_\alpha\in H(\Lambda,\alpha)\}
\]
Let $N_f:=coker(f-inc:rad\,P(\alpha)\to P(\alpha))$ for each $f\in Rep_\varphi(\Lambda,\alpha)$.
\end{defn}

\begin{rem}\label{rem: after def of Rep(L,a)}
\begin{enumerate}
\item $\undim N_f=\alpha$ for every $f\in Rep_\varphi(\Lambda,\alpha)$.%, i.e., $f-inc:rad\,P(\alpha)\to P(\alpha)$ is always a monomorphism.
\item Every $\Lambda$-module $M$ with $\undim M=\alpha$ is isomorphic to $N_f$ for some $f\in Rep_\varphi(\Lambda,\alpha)$. %We identify $f\in Rep_\varphi(\Lambda,\alpha)$ with the $\Lambda$-module $N_f$ together with the presentation $(f-inc):rad P(\alpha)\to P(\alpha)$.%We view elements of $Rep(\Lambda,\alpha)$ as $\Lambda$-modules with  presentations $rad P(\alpha)\to P(\alpha)$.
\item If $\psi_\alpha$ is obtained by another choice of decompositions \eqref{eq: decomposition of rad Pi} then there is a linear isomorphism $\lambda:Rep_\varphi(\Lambda,\alpha)\cong Rep_\psi(\Lambda,\alpha)$ with the property that $N_f=N_{\lambda f}$.% for all $f\in Rep_\varphi(\Lambda,\alpha)$, $f-inc$ and $\lambda(f)-inc:rad\,P(\alpha)\to P(\alpha)$ have the same image (and therefore the same cokernel).
\item When $\Lambda=KQ$ is the path algebra of a quiver or $\Lambda=T(Q,\cM)$ the tensor algebra of a modulated quiver, there are canonical choices for the isomorphisms $\varphi_j$ of \eqref{eq: decomposition of rad Pi} and the resulting representation space agrees with the classical definition of $Rep(\Lambda,\alpha)$.
\item We will write $Rep_\varphi(\Lambda,\alpha)=Rep(\Lambda,\alpha)$ suppressing the choice of $\varphi$.
\end{enumerate}
\end{rem}

 The following was originally proved by Kac \cite{K} for quivers over an algebraically closed field $\kk$. However, the  proof that we give below works over any infinite field.
 
 \begin{lem}\label{lem: generic decomposition}
Let $\Lambda$ be an hereditary algebra over an infinite field. 
Suppose that $M$ is a rigid $\Lambda$-module with projective presentation $0\to P\xrarrow p P'\to M\to 0$. Then the set of all $f\in \Hom_\Lambda(P,P')$ with $coker(f)\cong M$ is a nonempty Zariski open set in $\Hom_\Lambda(P,P')$.
 \end{lem}

\begin{proof}%[Proof of Theorem \ref{thm: generic decomposition thm 1}]
Since $\Hom_\Lambda(P,P')$ contains a monomorphism, the set of all monomorphisms $P\to P'$ is open. For any two monomorphisms $f_1,f_2:P\into P'$, $\Ext^1_\Lambda(coker\,f_1,coker\,f_2)$ is the cokernel of the map
\[
	\psi(f_1,f_2):\End_\Lambda(P)\oplus \End_\Lambda(P')\to \Hom_\Lambda(P,P')
\]
which sends $(g,g')$ to $g'\circ f_1+f_2\circ g$. Since $\psi(p,p)$ is surjective, the subset $U\subseteq \Hom_\Lambda(P,P')$ of all monomorphisms $f:P\to P'$ so that $\psi(f,p),\psi(p,f)$ and $\psi(f,f)$ are all surjective is open. This implies that $coker\,f\oplus M$ is rigid for all $f\in U$. We will show that $coker\,f$ is isomorphic to $M$ for any $f\in U$. This will imply that $\{f\in \Hom_\Lambda(P,P')\,:\,coker\,f\cong M\}=U$ is open.

For any $f\in U$, let $\{N_j\}$ be the components of $coker\,f$. Let $\{M_i\}$ be the components of $M$. Then $\{M_i,N_j\}$ form a collection of indecomposable modules which do not extend each other. So, by Schofield's observation, we can arrange them into an exceptional sequence, possibly with repetitions. Take the last object in the sequence. By symmetry, suppose it is $N_j$. Since $\undim M=\undim coker\,f$ we have 
\[
	\dim_K\Hom_\Lambda(N_j,M)=\brk{\undim N_j, \undim M}=\dim_K\Hom_\Lambda(N_j,coker\,f)\neq0.
\]
So, there is a nonzero morphism $N_j\to M_i$ for some $i$. Since $N_j$ is last in the exceptional sequence, this can happen only if $N_j\cong M_i$. Then $coker\,f/N_j,M/M_i$ are rigid modules of the same dimension vector. So, $coker\,f/N_j\cong M/M_i$ by induction on dimension. We conclude that $coker\,f\cong N_j\oplus coker\,f/N_j\cong M_i\oplus M/M_i\cong M$ as claimed.
\end{proof}

\begin{thm}\label{thm: generic decomposition thm 1}
Let $\Lambda$ be an hereditary algebra over an infinite field. Suppose that $\alpha\in \NN^n$ and $M$ is a rigid module with $\undim M=\alpha$. Then the set of all $f\in Rep(\Lambda,\alpha)$ so that $N_f\cong M$ forms an open dense subset of $Rep(\Lambda,\alpha)$.
\end{thm}

\begin{proof}
Let $\psi: Rep(\Lambda,\alpha)\into \Hom_\Lambda(rad\,P(\alpha),P(\alpha))$ be the affine linear embedding given by $\psi(f)=f-inc$. Let $V$ be the set of all $f\in Rep(\Lambda,\alpha)$ so that $N_f\cong M$. Then $V$ is open since it is the inverse image under $\psi$ of the open subset of $\Hom_\Lambda(rad\,P(\alpha),P(\alpha))$ of all maps with cokernel isomorphic to $M$.
\end{proof}

{Since exceptional modules are rigid, we have the following immediate consequence.}

\begin{cor}\label{thm: Mbeta is unique}
Suppose that $M_\alpha$ is an {exceptional $\Lambda$-module} with $\undim M_\alpha=\alpha$. Then the set of all $f\in Rep(\Lambda,\alpha)$ so that $N_f\cong M_\alpha$ forms an open and thus dense subset of $Rep(\Lambda,\alpha)$. In particular $M_\alpha$ is uniquely determined up to isomorphism by $\alpha$.
\end{cor}

Another important consequence of Theorem \ref{thm: generic decomposition thm 1} is the following.

\begin{cor}[Generic Decomposition Theorem for rigid {modules} in modulated case]\label{cor: generic decomposition thm 2}
Suppose that $\alpha_1,\cdots,\alpha_k$ are real Schur roots {so that $\Ext_\Lambda^1(M_{\alpha_i},M_{\alpha_j})=0$ for all $i,j$}. Let $\gamma=\sum_{i=1}^k n_i\alpha_i$ be a nonnegative integer linear combination of these roots. Then the generic representation with dimension vector $\gamma$ is isomorphic to $\bigoplus_{i=1}^k M_{\alpha_i}^{n_i }$.
\end{cor}

\begin{proof}
Apply Theorem \ref{thm: generic decomposition thm 1} to the module $M=\bigoplus_{i=1}^k M_{\alpha_i}^{n_i }$ with $\undim M=\gamma$.
\end{proof}

% Section:

%\newpage

\subsection{Presentation Spaces and Semi-invariants}

Let $\gamma_0,\gamma_1$ be vectors in $\NN^n$. We define the \emph{presentation space} $Pres_\Lambda(\gamma_1,\gamma_0)$ to be
\[
	Pres_\Lambda(\gamma_1,\gamma_0 ):=\Hom_\Lambda(P(\gamma_1),P(\gamma_0))
\]
where we use the notation $P(\alpha)=\bigoplus P_i^{\alpha_i}$. Presentation spaces are affine spaces over $\kk$. They are related to representation spaces as follows. Suppose that $\alpha\in\NN^n$. Then $\varphi:rad P(\alpha)\cong  P(\gamma)$ for $\gamma\in\NN^n$ and we have the $\kk$-linear embedding:
\[
	Rep(\Lambda,\alpha)\hookrightarrow Pres_\Lambda(\gamma,\alpha)
\]
sending $f:rad\,P(\alpha)\to P(\alpha)$ to $(f-inc)\circ\varphi^{-1}$. The elements of $Rep(\Lambda,\alpha)$ and their images in $Pres_\Lambda(\gamma,\alpha)$ represent the same module $N_f$. The algebraic group $Aut(P(\gamma_1))\op\times Aut(P(\gamma_0))$ acts on presentation space by composition: $(a,b)f=bfa$. This is a generalization of what happens in the algebraically closed case.% there is a corresponding group action on $Rep(\Lambda,\alpha)$.

\begin{defn}
A \emph{semi-invariant} on $Pres_\Lambda(\gamma_1,\gamma_0)=\Hom_\Lambda(P(\gamma_1),P(\gamma_0))$ is defined to be a regular function
\[
	\sigma:Pres_\Lambda(\gamma_1,\gamma_0)\to \kk
\]
for which there exist characters $\eta_s:\Aut(P(\gamma_s))\to \kk^\ast$, $s=0,1$, so that, for all $(g_0,g_1)\in \Aut(P(\gamma_0))\times \Aut(P(\gamma_1))$ and $f\in Pres_\Lambda(\gamma_1,\gamma_0)$, we have $\sigma(g_0fg_1)=\sigma(f)\eta_0(g_0)\eta_1(g_1)$. The pair of characters $(\eta_0,\eta_1)$ is called the \emph{{weight}} of $\sigma$. 
\end{defn}

The following lemma shows that such characters are products of character on $GL(\alpha_i,F_i)$.

\begin{lem}\label{lem: characters on Aut(P) factor through components}
Every group homomorphism $\Aut_\Lambda(P(\alpha))\to\kk^\ast$ factors through the group $\prod_i \Aut_\Lambda(P_i^{\alpha_i})$.
\end{lem}

\begin{proof} When $K$ has only two elements, the lemma holds trivially. So, we may assume $K$ has at least three elements. Then every element of $K$ can be written as $a-b$ where $a,b\neq0$. So, the elementary matrix
\[
	\mat{1&0\\a-b & 1}=\mat{a&0\\-ab&b}\mat{a&0\\0&b}^{-1}\mat{a&0\\-ab&b}^{-1}\mat{a&0\\0&b}
	%\mat{a^{-1}&0\\1&b^{-1}}\mat{a^{-1}&0\\0&b^{-1}}
\]
is a commutator. We write automorphisms of $P(\alpha)$ in block matrix form. Since $\Hom_\Lambda(P_j,P_i)=0$ for $i<j$, the matrix is lower triangular with diagonal blocks in $\Aut_\Lambda(P_i^{\alpha_i})$. So, every element in the kernel of the homomorphism $\pi:\Aut_\Lambda(P(\alpha))\onto \prod \Aut_\Lambda(P_i^{\alpha_i})$, when written in matrix form, is lower triangular with 1s on the diagonal. But all such matrices are products of elementary matrices such as the one above. So $\ker\pi$ lies in the commutator subgroup of $\Aut_\Lambda(P(\alpha))$. Since $\kk^\ast$ is abelian, any homomorphism $\varphi:\Aut_\Lambda(P(\alpha))\to \kk^\ast$ is trivial on commutators. Therefore $\ker\pi\subseteq\ker\varphi$ which implies that $\varphi$ factors through $\pi$ proving the lemma.
\end{proof}

\begin{rem}\label{rem: determinantal character}
%\begin{enumerate}
(1) By Lemma \ref{lem: characters on Aut(P) factor through components}, every character $\chi: \Aut_\Lambda(\textstyle\bigoplus P_j^{n_j})\to \kk^\ast$ is a product of component characters $\Aut_\Lambda(P_i^{n_i})\to\kk^\ast$.
%\[	\chi_i:\Aut_\Lambda(\textstyle\bigoplus P_j^{n_j})\onto \Aut_\Lambda(P_i^{n_i})\to\kk^\ast.\]

\noi (2)   Since $\End_\Lambda(P_i)=F_i$ is a division algebra, we have $\Aut_\Lambda(P_i^{n_i})\cong GL(n_i,F_i)$. This has a character given by the determinant of the induced automorphism $S_i^{n_i}\to S_i^{n_i}$ considered as a $\kk$-linear map:
\[
	\chi_i=det_\kk:\Aut_\Lambda(P_i^{n_i})\to\kk^\ast.
\]
Then $\chi_i$ is a polynomial of degree $n_if_i$ where $f_i=\dim_KS_i=\dim_KF_i$.
% of $P_i^{n_i}$. 
%\item When $F_i=\End_\Lambda(P_i)$ is a separable commutative field extension of $\kk$, it is well-known that every polynomial character $\Aut_\Lambda(P_i^{n_i})\to\kk^\ast$ is a power of $det_\kk$. In general, we will consider only those characters which are integer powers of $det_\kk$. 

\noi(3)   We will only consider special characters which we call ``determinantal'' (Definition \ref{def: determinantal character} below). There may be other characters called ``reduced norms'' which are explained in detail in {Appendix B, Sec \ref{ss: Appendix B}}.
%\end{enumerate}
\end{rem}

\begin{defn}\label{def: determinantal character}
We call a character $\chi:\Aut_\Lambda(\textstyle\bigoplus P_j^{n_j})\to \kk^\ast$ \emph{determinantal} if its components characters $\Aut_\Lambda(P_i^{n_i})\to\kk^\ast$ are integer powers of the determinant $\chi_i$, i.e., there exists a vector $\alpha\in\ZZ^n$ so that $\chi=\prod_i \chi_i^{\alpha_i}$. The coordinate $\alpha_i$ is uniquely determined by $\chi$ if and only if $n_i\neq0$ (and $\kk$ is infinite).
\end{defn}

The following proposition is analogous to Proposition 3.3.3 from \cite{IOTW1} {which was proved for simply laced quivers. But the same proof works in general.}

\begin{prop} {Let $\Lambda$ be a finite dimensional hereditary algebra over an infinite field $\kk$.} Suppose that $\sigma$ is a nonzero semi-invariant on $Pres_\Lambda(\gamma_1,\gamma_0)$ with {weight}s $\eta_0,\eta_1$ which are determinantal characters given by $\eta_0=\prod_i \chi_i^{\alpha_i}$, $\eta_1=\prod_i \chi_i^{\beta_i}$ where $\alpha,\beta\in\ZZ^n$. Then $\alpha_i=\beta_i$ whenever the $i$-th coordinates of $\gamma_0,\gamma_1$ are both nonzero.\qed
\end{prop}

\begin{defn}
We say that a semi-invariant $\sigma$ on $Pres_\Lambda(\gamma_1,\gamma_0)$ has \emph{determinantal {weight} vector} ({det-weight}) $\beta\in\ZZ^n$ if both of its {weight}s can be written as $\chi_i^{\beta_i}$. In other words, for any $f:P(\gamma_1)\to P(\gamma_0),h\in \Aut(P(\gamma_1)),g\in \Aut(P(\gamma_0))$ we have:
\begin{equation}\label{eq: definition of det weight of SI}
	\sigma(gfh)=\sigma(f)\textstyle\prod \chi_i(g)^{\beta_i}\chi_i(h)^{\beta_i}
\end{equation}
We also say that $f:P(\gamma_1)\to P(\gamma_0)$ \emph{admits a semi-invariant} of {det-weight} $\beta$ if there exists a semi-invariant $\sigma$ of {det-weight} $\beta$ on $Pres_\Lambda(\gamma_1,\gamma_0)$ so that $\sigma(f)\neq0$.
\end{defn}

\begin{eg}\label{eg: H to C to R}
The following example of a modulated quiver illustrates many of these concepts. Let $K=\RR$ and consider the following $\RR$-modulated quiver.
\[
	F_3=\HH\xrarrow{M_{32}=\HH}F_2=\CC \xrarrow{M_{21}=\CC}F_1=\RR
\]
Then $P_2$ is the representation $0\to \CC\to \RR^2$ with radical $0\to 0\to \RR^2$ which is $S_1^2=P_1^2=rad\,P_2$. The structure map of the module $P_2$ gives an $\RR$-linear isomorphism $\CC\cong \RR^2$. Let $\gamma_0=e_2=(0,1,0)$, $\gamma_1=e_1=(1,0,0)$. Then $P(\gamma_0)=P_2$, $P(\gamma_1)=P_1$ and the presentation space $Pres_\Lambda(\gamma_1,\gamma_0)=\Hom_\Lambda(P_1,P_2)\cong \RR^2$. Thus any homomorphism $f:P_1\to P_2$ is given by two real numbers $(x,y)$. Then $\sigma(f)=x^2+y^2$ is a semi-invariant on $Pres_\Lambda(\gamma_1,\gamma_0)$ of weight $\beta=(2,1,\ast)$ with $\beta_3$ being undefined. To see this consider the group which acts on the presentation space. The group is $G=GL(1,\RR)^{op}\times GL(1,\CC)\times GL(0,\HH)$. If $g=(r,a+bi,1)\in G$ then $gf=(a+bi)fr=(axr-byr,ayr+bxr)$ with 
\[
	\sigma(gf)=(a^2+b^2)(x^2+y^2)r^2=\chi_1(g)^2\chi_2(g)^1\chi_3(g)^m\sigma(f).
\]
Since $\chi_3(g)=1$, this equation is true for any $m\in\ZZ$. Thus, $\sigma$ is a determinantal semi-invariant on $Pres_\Lambda(e_1,e_2)$ with det-weight $(2,1,m)$ for any $m\in\ZZ$. The third coordinate is not well-defined since $\gamma_{0,3}=0=\gamma_{1,3}$.
\end{eg}

We now show that the coordinates of $\beta$ are nonnegative when they are well-defined.

\begin{prop}\label{Weights are nonnegative}
If a semi-invariant on $Pres_\Lambda(\gamma_1,\gamma_0)$ has well-defined det-weight $\beta$ then $\beta\in\NN^n$, i.e., $\beta_i\ge0$ for all $i$.
\end{prop}

\begin{proof} Since $\beta$ is well-defined, for each $i$, either $\gamma_{0,i}\neq0$ or $\gamma_{1,i}\neq0$. By symmetry assume $n_i=\gamma_{0,i}\neq0$. Choose $f\in Pres_\Lambda(\gamma_1,\gamma_0)$ so that $\sigma(f)\neq0$. Let $\psi:\End_\Lambda(P_i^{n_i})\into  \End_\Lambda(\bigoplus P_j^{n_j})$ be the embedding which sends $g\in\End_\Lambda(P_i^{n_i})$ to the endomorphism of $\bigoplus P_j^{n_j}$ which is $g$ on $P_i^{n_i}$ and the inclusion map on $P_j^{n_j}$ for every $j\neq i$. Then $g\mapsto\sigma(\psi(g)f)$ is a regular function $\End_\Lambda( P_i^{n_i})\to K$ which extends the map $g\mapsto \chi_i(g)^{\beta_i}\sigma(f)$ on $\Aut_\Lambda(P_i^{n_i})$ and sends $0$ to $0$. This is impossible for $\beta_i<0$. Therefore, $\beta_i\ge0$ for every $i$.
\end{proof}
%\End_\Lambda(\bigoplus P_j^{n_j})$ and $\sigma(gfh)\neq0$ if $g\in \Aut_\Lambda(\bigoplus P_j^{n_j})$. But $\chi_i(g)^{\beta_i}:\Aut_\Lambda(\bigoplus P_j^{n_j})\to \Aut_\Lambda(P_i^{n_i})\to \kk^\ast$ does not extend to a regular function $

{The following proposition is one of the motivations for the uniform notation $\Vrep(\Lambda,\alpha)$ introduced in the next section in Definition \ref{def: virtual representation space}.}

\begin{prop}\label{prop: semi-invariants are perpendicular to dimension}
Suppose that $Pres_\Lambda(\gamma_1,\gamma_0)$ has a semi-invariant of {det-weight} $\beta$ and $(L^t)^{-1}(\gamma_0-\gamma_1)=\alpha$ {where $L$ is the left Euler matrix and $K$ is infinite.} Then $\brk{\alpha ,\beta}=0$.
\end{prop}

\begin{proof} 
Consider the automorphisms of $P(\gamma_1)=\bigoplus_i P_i^{\gamma_{1,i}}$ and $P(\gamma_0)=\bigoplus_i P_i^{\gamma_{0,i}}$ given by multiplication by $\lambda\in\kk^\ast$. The character of this automorphism of $P_i^{\gamma_{1,i}}$ is $\chi_i(\lambda)=\det(\lambda^{\gamma_{1,i}})=\lambda^{\gamma_{1,i}f_i}$. Since $\lambda f=f\lambda$ {for all $f\in Pres_\Lambda(\gamma_1,\gamma_0)$ we conclude that}
\[
	\lambda^{\sum \gamma_{0,i}f_i\beta_i}=\lambda^{\sum \gamma_{1,i}f_i\beta_i}
\]
 Since this {polynomial equation} holds for all $\lambda\in \kk^\ast$ which is infinite, we conclude that ${\sum \gamma_{0,i}f_i\beta_i}={\sum \gamma_{1,i}f_i\beta_i}$. So,
\[
	0=\sum (\gamma_{0,i}-\gamma_{1,i})f_i\beta_i=\brk{\alpha ,\beta}
\]
since $(L^t)^{-1}(\gamma_0-\gamma_1)=\alpha$ and $\brk{\alpha,\beta}=\alpha^t E\beta=\alpha^tLD\beta=(\gamma_0-\gamma_1)D\beta$.
\end{proof}

{As a corollary of this proof we have the following.}

\begin{cor}\label{prop: sum det wt = degree of sigma}
A semi-invariant $\sigma$ on $Pres_\Lambda(\gamma_1,\gamma_0)$ with det-weight $\beta$ is a homogeneous polynomial function of degree $\sum_i \gamma_{1,i}f_i \beta_i$ which is also equal to $\sum_i \gamma_{0,i}f_i \beta_i$ {assuming $K$ is infinite}. In particular, $\beta=0$ if and only if $\sigma$ is constant.\qed
\end{cor}

When $f\in Pres_\Lambda(\gamma_1,\gamma_0)$ is a monomorphism $P(\gamma_1)\into P(\gamma_0)$, we have:%the dimension vector of the cokernel is
\[
	\undim\coker f=\undim P(\gamma_0)-\undim P(\gamma_1)=(L^t)^{-1}(\gamma_0-\gamma_1)
\]
which is $\alpha$ in the proposition above. We want to view different presentations of the same module as being equivalent. To make this precise we make the following definitions.

%%\newpage
%%%%%%%%%%%%%%%%%%%%%%%%%%
%
%                Section  {Virtual representations}
%
%%%%%%%%%%%%%%%%%%%%%%%%%%

\subsection{Virtual representations}\label{ss: virtual representations} 

{``Virtual representations'' will be given by ``stablilizing'' presentation $f:P(\gamma_1)\to P(\gamma_0)$. These will form the objects of the ``virtual representation category'' and the elements of the ``virtual representation space.'' First,} note that 
\[
P(\gamma+\delta)=P(\gamma)\oplus' P(\delta)\,
\]
where $\oplus'$ denotes the ``shuffle sum'' given by collecting isomorphic summands together. We use this to make the equality strict. For example $(P_1\oplus P_2)\oplus' P_1$ denotes $P_1\oplus P_1\oplus P_2$. Given any three dimension vectors $\gamma_0,\gamma_1,\delta\in\NN^n$, consider the linear monomorphism
\[
	Pres_\Lambda(\gamma_1,\gamma_0)\into Pres_\Lambda(\gamma_1+\delta,\gamma_0+\delta)
\]
given by sending $f:P(\gamma_1)\to P(\gamma_0)$ to $f\oplus' 1_{P(\delta)}:P(\gamma_1)\oplus' P(\delta)\to P(\gamma_0)\oplus' P(\delta)$. We call this map \emph{stabilization}. This gives a directed system whose objects are all presentation spaces $Pres_\Lambda(\delta_1,\delta_0)$ having the property that $\delta_0-\delta_1=\gamma_0-\gamma_1$. This implies $\undim P(\delta_0)-\undim P(\delta_1)=\alpha=\undim P(\gamma_0)-\undim P(\gamma_1)\in \ZZ^n$. Equivalently, $\gamma_0-\gamma_1=L^t \alpha$.

\begin{defn}\label{def: virtual representation space}
For any $\alpha\in \ZZ^n$ we define the \emph{virtual representation space} $\Vrep(\Lambda,\alpha)$ to be the direct limit (colimit):
\[
	\Vrep(\Lambda,\alpha):=colim\,Pres_\Lambda(\gamma_1,\gamma_0)=colim \Hom_\Lambda(P(\gamma_1),P(\gamma_0))
\]
where the colimit is taken over all pairs $\gamma_0,\gamma_1\in\NN^n$ so that $\gamma_0-\gamma_1=L^t \alpha$. Elements of $\Vrep(\Lambda,\alpha)$ will be called \emph{virtual representations} of $\Lambda$ of  \emph{dimension vector} $\alpha\in\ZZ^n$. We take the direct limit topology on $\Vrep(\Lambda,\alpha)$. Since each presentation space is irreducible, it follows that $\Vrep(\Lambda,\alpha)$ is irreducible, i.e., any nonempty open subset is dense. 
\end{defn}

{
The main purpose of the virtual representation space is to make the weights of semi-invariants well-defined. See Definition \ref{def: virtual semi-invariant} below.
}

We now construct the category $\Vrep(\Lambda)$ of all virtual representations of $\Lambda$. The object set of this category will be the disjoint union
\[
	\mathcal O b(\Vrep(\Lambda)):=\bigsqcup_{\alpha\in\ZZ^n}\Vrep(\Lambda,\alpha).
\]
Representatives of $\Vrep(\Lambda,\alpha)$ are presentations $p:P(\gamma_1)\to P(\gamma_0)$ which we denote $P(\gamma_\ast)$. A morphism in $\Vrep(\Lambda)$ can be defined on representatives as in the following diagram
\[
	\xymatrixrowsep{15pt}\xymatrixcolsep{15pt}
\xymatrix{%begin xy matrix
P(\xi_\ast)=\ar[d]_{f=(f_0,f_1)}&P(\xi_1)\ar[r]_p\ar[d]^{f_1} & P(\xi_0)\ar[d]^{f_0}\\
P(\eta_\ast)=& P(\eta_1)\ar[r]^q & 
	P(\eta_0).
	}
\]
In other words, $(f_0,f_1)$ gives a chain map $P(\xi_\ast)\to P(\eta_\ast)$. Two such chain maps are equivalent $(f_0,f_1)\sim (f_0',f_1')$ if they are homotopic, i.e., if there is a map $h:P(\xi_0)\to P(\eta_1)$ so that $f_1'=f_1+hp$ and $f_0'=f_0+qh$. We define a morphism $X\to Y$ to be an equivalence class of such chain maps under the equivalence relation generated by homotopy as explained above and stabilization which means $(f_0,f_1)\sim (f_0\oplus' 1_P,f_1\oplus'1_P)$ for any projective module $P=P(\zeta)$.

Since direct sum does not commute with stabilization, to define direct sums in $\Vrep(\Lambda)$ we define the category $Pres(\Lambda)$ and show that it is equivalent to $\Vrep(\Lambda)$. $Pres(\Lambda)$ is the category whose objects are all chain complexes of finitely generated projective modules in degrees 0 and 1: $P(\gamma_\ast)=(p:P(\gamma_1)\to P(\gamma_0))$ and whose morphisms are homotopy classes of degree 0 chain maps. Objects of $Pres(\Lambda)$ will be called \emph{presentations}.

\begin{prop}\label{prop 2.3.2: stabilization is an equivalence}
The stabilization map $P(\gamma_\ast)\mapsto \left<P(\gamma_\ast)\right>$ is an equivalence of categories
\[
	Pres(\Lambda)\cong \Vrep(\Lambda).
\]
\end{prop}

\begin{proof}
As a chain complex, every presentation is homotopy equivalent to each of its stabilizations. Therefore, any two representatives of the same virtual representation are canonically isomorphic as objects of $Pres(\Lambda)$. Given any two objects $X,Y$ in $\Vrep(\Lambda)$, a morphism $f:X\to Y$ is represented by a morphism $\tilde f=(f_0,f_1):P(\xi_\ast)\to P(\eta_\ast)$ in $Pres(\Lambda)$. Since these representatives are unique up to canonical isomorphism in $Pres(\Lambda)$, $\tilde f$ is unique. So, $\Hom_{\Vrep(\Lambda)}(X,Y)\cong \Hom_{Pres(\Lambda)}(P(\xi_\ast), P(\eta_\ast))$. In other words the stabilization functor is full, faithful and dense. So, it is an equivalence.
\end{proof}

Since the kernel of $p:P(\gamma_1)\to P(\gamma_0)$ splits off of $P(\gamma_1)$, we get the following.
\begin{prop}\label{prop:two kinds of objects}
The indecomposable objects of $Pres(\Lambda)$ and $\Vrep(\Lambda)$ are 
\begin{enumerate}
\item projective presentations of indecomposable $\Lambda$-modules and
\item shifted indecomposable projective $\Lambda$-modules $P[1]$, i.e., $P\to 0$ $\in Pres(\Lambda)$.
\end{enumerate}
\end{prop}

\begin{defn}
The \emph{underlying module} of a presentation $P(\gamma_\ast)=(P(\gamma_1)\xrarrow p P(\gamma_0))$ is defined to be $|P(\gamma_\ast)|:=\ker p\oplus \coker p$. In particular, $|P[1]|=P$.
\end{defn}

\begin{rem}\label{rem: easy observations about presentation spaces}
Let $P(\xi_\ast)=(f:P(\xi_1)\to P(\xi_0))$ and $P(\eta_\ast)=(g:P(\eta_1)\to P(\eta_0))$ be objects in $Pres (\Lambda)$ and representatives of objects in $\Vrep(\Lambda)$. The following are equivalent.
\begin{enumerate}
\item $P(\xi_\ast)\cong P(\eta_\ast)$ in $Pres(\Lambda)$.
\item $P(\xi_\ast)\cong P(\eta_\ast)$ in $\Vrep(\Lambda)$.
\item $\ker f\cong \ker g$ and $\coker f\cong \coker g$ in $mod\text-\Lambda$.% have isomorphic kernels and cokernels.
\item $f,g$ are homotopy equivalent.
\item If, in addition, $\xi_0=\eta_0$ then $f,g$ are chain isomorphic.
\end{enumerate}
\end{rem}

For two objects $P(\xi_\ast),P(\eta_\ast)$ of $Pres(\Lambda)$ (or $\Vrep(\Lambda)$) we define 
$\Ext^1_{Pres(\Lambda)}(P(\xi_\ast),P(\eta_\ast))$ in the usual way as the space of homotopy classes of chain maps $P(\xi_\ast)\to  P(\eta_\ast)[1]$.

\begin{cor}\label{characterization of VRep}
$Pres(\Lambda)$ is equivalent to the full subcategory of the bounded derived category of $mod\text-\Lambda$ with objects all $P(\gamma_\ast)$ so that $\Hom_{\cD^b}(P(\gamma_\ast),Y[k])=0$ for all $Y\in mod\text-\Lambda$ and all $k\neq0,1$. Furthermore, $\Ext^1_{Pres(\Lambda)}(P(\gamma_\ast),P(\delta_\ast))=\Ext^1_{\cD^b}(P(\gamma_\ast),P(\delta_\ast))$ for all $P(\gamma_\ast),P(\delta_\ast)\in Pres(\Lambda)$.
\end{cor}

\begin{proof}
It is clear that all $P(\gamma_\ast)\in Pres(\Lambda)$ satisfy this condition. Conversely, suppose that $P(\gamma_\ast)$ satisfies the condition. Then $P(\gamma_\ast)\in mod\text-\Lambda$ or $P(\gamma_\ast)=Z[1]$ where $Z\in mod\text-\Lambda$. In the second case we have $\Hom_{\cD^b}(Z[1],Y[2])=0$ for all modules $Y$. This implies that $Z$ is projective.
\end{proof}

{Recall that the \emph{cluster category} $\cC_\Lambda$ of $\Lambda$ is the orbit category of the bounded derived category $\cD^b(mod\text-\Lambda)$ under the functor $F=\tau^{-1}[1]$ (see \cite{BMRRT}). Recall that a \emph{partial cluster tilting object} is an object $T$ of $\cC_\Lambda$ so that $\Ext_{\cC_\Lambda}^1(T,T)=0$ and if it has $n$ nonisomorphic summands it is called a \emph{cluster tilting object}. The fundamental domain of the functor $F$ consists of $\Lambda$-modules and shifted projective modules. Therefore we get the following.

\begin{cor}\label{cor 2.3.7: Pres L=CL}
The functor $\Psi:Pres(\Lambda)\to \cC_\Lambda$ which sends each object to its $F$-orbit is a faithful functor which induces a bijection between isomorphism classes of objects. Furthermore, $\Ext_{\cD^b}^1(P(\xi_\ast),P(\eta_\ast))=0=\Ext_{\cD^b}^1(P(\eta_\ast),P(\xi_\ast))$ if and only if $\Ext_{\cC_\Lambda}^1(\Psi P(\xi_\ast),\Psi P(\eta_\ast))=0=\Ext_{\cC_\Lambda}^1(\Psi P(\eta_\ast),\Psi P(\xi_\ast))$ for all $P(\xi_\ast),P(\eta_\ast)$ in $Pres(\Lambda)$.
\end{cor}
}

\begin{defn}\label{def: dim vector of a presentation}
The \emph{dimension vector} of a presentation $P(\gamma_\ast)=(p:P(\gamma_1)\to P(\gamma_0))$ is defined to be $\undim P(\gamma_\ast):=\undim P(\gamma_0)-\undim P(\gamma_1)=\undim \coker p-\undim\ker p$. This is the unique integer vector $\alpha\in\ZZ^n$ satisfying $L^t\alpha=\gamma_0-\gamma_1$ where $L$ is the left Euler matrix of $\Lambda$. 
\end{defn}
	
\begin{thm}\label{Virtual canonical decomposition theorem} Let $P(\gamma_\ast)=(p:P(\gamma_1)\to P(\gamma_0))$ be a presentation with dimension vector $\undim P(\gamma_\ast)=\alpha$ so that $\Ext^1_{Pres(\Lambda)}(P(\gamma_\ast),P(\gamma_\ast))=0$. Then the set of all presentations isomorphic to $P(\gamma_\ast)$ is an open dense subset of the $\kk$-affine space $Pres_\Lambda(\gamma_1,\gamma_0)$.
\end{thm}

\begin{proof}
Let $P=\ker p$. Then $P(\gamma_\ast)=P[1]\oplus P(\gamma_\ast')$ where $P(\gamma_\ast')=(P(\gamma_1')\xrarrow q P(\gamma_0))$ is a projective presentation of a $\Lambda$-module $M$ with $\Ext^1_\Lambda(M,M)=0$ and $\undim M=\beta$. By assumption, $0=\Ext^1_{Pres(\Lambda)}(P[1],P(\gamma_\ast'))=\Ext^1_{\cD^b(\Lambda)}(P[1],P(\gamma_\ast'))=\Hom_{\cD^b(\Lambda)}(P,P(\gamma_\ast'))=\Hom_\Lambda(P,M)$. Let $f:P(\gamma_1)\to P(\gamma_0)$ be a general morphism. Restrict $f$ to the components of $P(\gamma_1)$ to get $f_1:P\to P(\gamma_0)$ and $f_2: P(\gamma_1')\to P(\gamma_0)$. Since $q$ is a monomorphism and $f_2$ is a general map, it follows from Lemma \ref{lem: generic decomposition} that $f_2$ is a monomorphism with cokernel isomorphic to $M$. So, $f_2$ is homotopy equivalent and thus isomorphic to $q:P(\gamma_1')\to P(\gamma_0)$. Since $\Hom_\Lambda(P,M)=0$, $f_1=f_2\circ s$ for some $s:P(\gamma_0)\to P(\gamma_1')$. Then presentation $(f_1-f_2\circ s,f_2)=(0,f_2)$ is isomorphic to $f:P(\gamma_1)\to P(\gamma_0)$ and to $P[1]\oplus P(\gamma_\ast')=P(\gamma_\ast)$. Thus the general presentation $f:P(\gamma_1)\to P(\gamma_0)$ is isomorphic to $P(\gamma_\ast)$ in $Pres_\Lambda(\gamma_1,\gamma_0)$.
\end{proof}

Recall that a \emph{partial cluster tilting set} is a set $\{\beta_i\}$ of distinct real Schur roots and negative projective roots which are the dimension vectors of components of a partial cluster tilting object in the cluster category $\cC_\Lambda$. If the partial cluster tilting set has exactly $n$ elements it is called a \emph{cluster tilting set}. 

\begin{defn}\label{def: partial cluster tilting set}
 A \emph{partial cluster tilting object} in $Pres(\Lambda)$ is defined to be $\bigoplus P(\gamma^i_{\ast})$ such that $\{\undim P(\gamma^i_{\ast})\}$ is a partial cluster tilting set. If this object has $n$ nonisomorphic summands it is called a \emph{cluster tilting object} in $Pres(\Lambda)$.
\end{defn}

%\begin{cor}[Virtual Generic Decomposition Theorem]
%Let $\gamma=\sum_{i=1}^k r_i\beta_i\in\ZZ^n$ be an integer vector which is a nonnegative rational linear combination of the elements $\beta_i$ of a {partial cluster tilting set.} Then the $r_i$ are all integers, call them $n_i$, and the general presentation with dimension vector $\gamma$ is isomorphic to $\bigoplus_{i=1}^k M_i^{n_i}$ where $M_i$ are rigid with $\undim M_i=\beta_i$. In other words, the set of all elements of $\Vrep(\Lambda,\gamma)$ isomorphic to $\bigoplus_{i=1}^k M_i^{n_i}$ is open and dense.\end{cor}

{
\begin{thm}[Virtual Generic Decomposition Theorem]\label{thm 2.3.11: virtual generic decomposition theorem} Let $\{\beta_i\}$ be a {partial cluster tilting set}. Let $\alpha=\sum r_i\beta_i\in\ZZ^n$ where $r_i\in \QQ$. Then $r_i\in \ZZ$ and the general virtual representation in $\Vrep(\Lambda,\alpha)$ is isomorphic to $\bigoplus_i P(\gamma^i_{\ast})^{r_i}$ where $P(\gamma^i_{\ast})$ are rigid objects in $\Vrep(\Lambda,\beta_i)$. In other words, the set of all elements of $\Vrep(\Lambda,\alpha)$ isomorphic to $\bigoplus_i P(\gamma^i_{\ast})^{r_i}$ is open and dense.
\end{thm}
}

\begin{proof}
The underlying modules $|P(\gamma^i_{\ast})|, i=1,\cdots,k$ form an exceptional sequence by Lemma \ref{lem: Schofield's observation} if we put the shifted projectives last. {This can be extended to a complete exceptional sequence, say $\{M_j\}_{j=1}^n$. (See section \ref{sec:exc seq}.) The dimension vectors $\undim M_j$ generate $\ZZ^n$ by Proposition \ref{prop: properties of exceptional sequences} (3). Therefore, the integer vectors in the $\QQ$-span of the vectors in the subset $\{\undim M_i\}_{i=1}^k$ lie in the $\ZZ$-span of these vectors.} By Theorem \ref{Virtual canonical decomposition theorem}, the virtual representations isomorphic to $\bigoplus_{i=1}^k P(\gamma^i_{\ast})^{n_i}$ form an open dense subset of each presentation space and therefore of the colimit $\Vrep(\Lambda,\gamma)$.
\end{proof}

% subsection

%\newpage
%-----------------------------------------------------------------------------------
%            sub section {Virtual semi-invariants}
%-----------------------------------------------------------------------------------

\subsection{Virtual semi-invariants}

We return to the discussion of semi-invariants. We consider direct sums of presentations.

\begin{lem}\label{lem: weight of f1+f2 is same as weight of f1,f2}
Let $f:P(\gamma_1+\delta_1)\to P(\gamma_0+\delta_0)$ be a direct sum of two projective presentations $f=f_1\oplus f_2$ where $f_1:P(\gamma_1)\to P(\gamma_0)$ and $f_2:P(\delta_1)\to P(\delta_0)$. If $f$ admits a semi-invariant of {det-weight} $\beta$ then so does each $f_i$.
\end{lem}

\begin{proof}
Consider the composition:
\[
	Pres_\Lambda(\gamma_1,\gamma_0)\times Pres_\Lambda(\delta_1,\delta_0)\xrarrow\iota Pres_\Lambda(\gamma_1+\delta_1, \gamma_0+\delta_0)\xrarrow\sigma \kk
\]
where $\sigma$ is a semi-invariant of {det-weight} $\beta$ on $Pres_\Lambda(\gamma_1+\delta_1, \gamma_0+\delta_0)$ so that $\sigma(\iota(f_1,f_2))\neq0$. Then, semi-invariants on $Pres_\Lambda(\gamma_1,\gamma_0)$ and $ Pres_\Lambda(\delta_1,\delta_0)$ can be defined by $\sigma(\iota(-,f_2)):Pres_\Lambda(\gamma_1,\gamma_0)\to\kk$ and analogously for $ Pres_\Lambda(\delta_1,\delta_0)$. It is easy to see that these are regular functions and they are semi-invariants of {det-weight} $\beta$. Indeed, suppose that $g_1,g_2,h_1,h_2$ are automorphisms of $P(\gamma_1), P(\delta_1),P(\gamma_0),P(\delta_0)$. Then $g=g_1\oplus g_2$ and $h=h_1\oplus h_2$ are automorphisms of $P(\gamma_0)\oplus P(\delta_0)$ and $P(\gamma_1)\oplus P(\delta_1)$ respectively so that $\chi_i(g)=\chi_i(g_1)\chi_i(g_2)$ and $\chi_i(h)=\chi_i(h_1)\chi_i(h_2)$. Therefore
\[
	\sigma(g_1f_1h_1,g_2f_2h_2)=\sigma(gfh)=\sigma(\iota(f_1,f_2))\textstyle\prod\chi_i(g_1)^{\beta_i}\chi_i(g_2)^{\beta_i}\chi_i(h_1)^{\beta_i}\chi_i(h_2)^{\beta_i}
\]
So, $\sigma(\iota(-,f_2))$ is a semi-invariant on $Pres_\Lambda(\gamma_1,\gamma_0)$ of {det-weight} $\beta$ which is nonzero on $f_1$ and similarly with $\sigma(\iota(f_1,-))$.
\end{proof}

{Let $Pres(\Lambda,\alpha)=\bigsqcup Pres_\Lambda(\gamma_1,\gamma_0)$ denote the disjoint union of presentation spaces $Pres_\Lambda(\gamma_1,\gamma_0)$ over all pairs $\gamma_0,\gamma_1\in \NN^n$ so that $L^t\alpha=\gamma_0-\gamma_1$.}

\begin{prop}\label{prop: no semi-invariants on f1+...+fn}
Suppose that $\alpha_1,\cdots,\alpha_n\in\ZZ^n$ are linearly independent. Suppose that $f_i\in Pres(\Lambda,\alpha_i)$. Then $\bigoplus f_i\in Pres(\Lambda,\sum \alpha_i)$ does not admit a semi-invariant with nonzero {det-weight}.
\end{prop}

\begin{proof}
If $f=\bigoplus f_i$ admits a semi-invariant of {det-weight} $\beta$ then, by Lemma \ref{lem: weight of f1+f2 is same as weight of f1,f2}, so does every $f_i$. By Proposition \ref{prop: semi-invariants are perpendicular to dimension}, we conclude that $\brk{\alpha_i,\beta}=0$ for all $i$. So, $\beta=0$.
\end{proof}

\begin{rem}\label{rem: det weight and reduced norm}
For any semi-invariant $\sigma$, there is a power of $\sigma$ which has determinantal weight. This follows from the fact that $det_\kk$ is a power of the reduced norm $\overline n:M_k(F_i)\to K$. {We refer the reader to {Appendix B, Sec \ref{ss: Appendix B}} for the definition of reduced norm and the proof of the theorem (Theorem \ref{thm relating determinant and reduced norm}) that all characters are powers of the reduced norm.} {Consequently, in Proposition \ref{prop: no semi-invariants on f1+...+fn} above, $\bigoplus f_i$ does not admit a semi-invariant of any weight since, if it did, then some power of that semi-invariant would be a semi-invariant with nonzero determinantal weight.}
\end{rem}

\begin{defn}
By a semi-invariant on $Pres(\Lambda,\alpha )$ we mean a semi-invariant on one of the presentation spaces $Pres_\Lambda(\gamma_1,\gamma_0)$ in the disjoint union. Such a semi-invariant $\sigma$ will be called \emph{determinantal} if there is a module $M$ so that, for all $f:P(\gamma_1)\to P(\gamma_0)$ in $Pres_\Lambda(\gamma_1,\gamma_0)$, $\sigma(f)$ is the determinant of the induced map
\[
	\Hom_\Lambda(f,M):\Hom_\Lambda(P(\gamma_0),M)\to \Hom_\Lambda(P(\gamma_1),M).
\]
We denote this by $\sigma_M$. It is easy to see that $\sigma_M$ is a semi-invariant of {det-weight} $\undim M$. In case the {det-weight} of $\sigma_M$ is not well-defined, we take it to be $\undim M$ by definition.
\end{defn}

When the ground field $\kk$ is algebraically closed then Schofield \cite{S91} showed that the determinantal semi-invariants generate the ring of all semi-invariants in the Dynkin case and this theorem was extended in \cite{DW} to all quivers over an algebraically closed field.

\begin{cor}\label{cor 2.4.5}
Let $\alpha=\sum n_i\beta_i$ be an integer linear combination of the vectors $\beta_i$ in a {cluster tilting set.} If $n_i>0$ for all $1\le i\le n$, then $Pres(\Lambda,\alpha)$ has no semi-invariant with nonzero determinantal {weight}.
\end{cor}

\begin{proof} If there is a nonzero semi-invariant $\sigma$ on $\Hom_\Lambda(P,Q)$ with $\undim(Q)-\undim(P)=\alpha$ then $\sigma$ will be nonzero on the generic element of $\Hom_\Lambda(P,Q)$. By Corollary \ref{thm 2.3.11: virtual generic decomposition theorem}, the generic element splits as a direct sum of $n$ objects with linearly independent dimension vectors. But this contradicts Proposition \ref{prop: no semi-invariants on f1+...+fn}. Our Corollary follows.
\end{proof}

\begin{defn}\label{def: virtual semi-invariant}
A \emph{virtual semi-invariant} of {det-weight} $\beta$ on $\Vrep(\Lambda,\alpha)$ is a mapping
\[
	\sigma:\Vrep(\Lambda,\alpha)\to K
\]
whose restriction to each $Pres_\Lambda(\gamma_1,\gamma_0)\subset \Vrep(\Lambda,\alpha)$ is a semi-invariant of {det-weight} $\beta$. 
\end{defn}

By definition of direct limit, a virtual semi-invariant on $\Vrep(\Lambda,\alpha)$ is the same as a system of semi-invariants one on each $Pres_\Lambda(\gamma_1,\gamma_0)$ which are compatible with stabilization. One example is the determinantal semi-invariant $\sigma_M$ defined above. {Since each coordinate of $\gamma_0$ and $\gamma_1$ become arbitrarily large, the weight of a virtual semi-invariant is well-defined when $K$ is infinite.}

\section{Virtual stability theorem}\label{ss: virtual stability theorem}

In this section we will prove the Virtual Stability Theorem (\ref{thm 3.1.1: virtual stability theorem}) which states that the domain $D_\ZZ(\beta)$ of the semi-invariant with {det-weight} $\beta$ defined in \ref{defn: D(b), Dss(b), DZ(b), DZss(b)} is the subset of $\ZZ^n$ given by the stability conditions of \ref{thm 3.1.1: virtual stability theorem}(2). We also give a description of all elements of this set (Proposition \ref{prop 3.5.2: extended stability theorem}).

\subsection{Statements of the theorem}

Let $\beta,\beta'$ be real Schur roots.
We say that $\beta'$ is a real Schur subroot of $\beta$ if $M_\beta$ contains a submodule isomorphic to $M_{\beta'}$.

 \begin{thm}[Virtual Stability Theorem]\label{thm 3.1.1: virtual stability theorem}Let $K$  be any field. Let $\Lambda$ be a finite dimensional hereditary $K$-algebra with $n$ non-isomorphic simple modules. Let $\alpha\in \mathbb Z^n$ and $\beta$ a real Schur root. Then, the following are equivalent:
\begin{enumerate}
\item  There exists a virtual representation $f:P\to Q$ so that $\ \undim Q-\undim P=\alpha$ and $f$ induces an isomorphism
\[
f^\ast:\Hom_\Lambda(Q,M_\beta)\xrightarrow{\approx} \Hom_\Lambda(P,M_\beta).
\]
\item Stability conditions for $\alpha$ and $\beta$ hold:
 $\brk{\alpha,\beta}=0$ and
 $\brk{\alpha,\beta'}\le 0$ for all real Schur subroots $\beta'\subseteq \beta$.
\item There is a nonzero determinantal semi-invariant of {det-weight} $\beta$ on the virtual representation space $\Vrep(\Lambda,\alpha)$.
\end{enumerate}
\end{thm}

\begin{rem}
In the previous paper \cite{IOTW1}, the authors proved the Virtual Stability Theorem for hereditary algebras over an algebraically closed field and vectors $\alpha\in \mathbb Z^n$, which may have negative coordinates.
This was an extension of the results of \cite{DW} and \cite{King} from $\alpha\in \mathbb N^n$ to $\alpha\in \mathbb Z^n$. Here we extend this theorem to hereditary algebras over any field. We also note that Condition (2) is weaker and thus the theorem is stronger than the original theorems of \cite{DW} and \cite{King} since the condition $\brk{\alpha,\beta'}\le0$ is only required for real Schur subroots $\beta'$ of $\beta$ and not for all subroots of $\beta$.
\end{rem}

We now restate the theorem in terms of the following sets, usually referred  to as various ``domains of virtual semi-invariants" or ``supports of virtual semi-invariants".
 \begin{defn}\label{defn: D(b), Dss(b), DZ(b), DZss(b)} Let $\beta$ be a real Schur root. We define the following:
$$
D_\ZZ(\beta):=\{\alpha\in \ZZ^n\,:\, \text{Condition(1) holds}\} = \text{\emph{integral support} of {det-weight} } \beta,
$$
$$
D(\beta):=\text{convex hull of } D_\ZZ(\beta) \text{ in } \RR^n = \text{real  support of {det-weight} } \beta,
$$
\[
D^{ss}(\beta):=\{x\in\RR^n\,:\, \text{$\brk{x,\beta}=0$ and
 $\brk{x,\beta'}\le 0$ for all real Schur subroots $\beta'\subseteq \beta$}\} \]\[= \text{\emph{support} of real semi-stability conditions},
 \]
\[D_\ZZ^{ss}(\beta):=\{\alpha\in\ZZ^n\,:\, \text{$\brk{\alpha,\beta}=0$ and
 $\brk{\alpha,\beta'}\le 0$ for all real Schur subroots $\beta'\subseteq \beta$}\}\]\[=D^{ss}(\beta)\cap \ZZ^n = 
\text{\emph{support} of integral semi-stability conditions.}\]
\end{defn}
 
%Note that $D_\ZZ(\beta)$ is an additive monoid. It contains 0 and is closed under sum.

The Virtual Stability Theorem \ref{thm 3.1.1: virtual stability theorem} can now be restated as:
\begin{thm}[Virtual Stability Theorem']\label{virtual stability theorem'}Let $K$  be any field, $\Lambda$  a finite dimensional hereditary $K$-algebra with $n$  simple modules. Let $\beta$ be a real Schur root. Then 
 \[
 	D^{ss}_\ZZ(\beta)=D_\ZZ(\beta).
 \]
\end{thm}

The proof of the theorem occupies the rest of this section. We will first prove the theorem for infinite fields and in  subsection \ref{sec: extend to finite K} we will extend the proof to all fields. We start with the simple lemma showing  the equivalence of conditions (1) and (3) in the Virtual Stability Theorem \ref{thm 3.1.1: virtual stability theorem}, hence reducing the proof to showing that $D^{ss}_\ZZ(\beta)=D_\ZZ(\beta)$, i.e. Virtual Stability Theorem' \ref{virtual stability theorem'}.

\begin{lem} Let $K$  be an infinite field, $\Lambda$  a finite dimensional hereditary $K$-algebra with $n$  simple modules, $\alpha\in \mathbb Z^n$ and $\beta$ a real Schur root. The following are equivalent:
\begin{enumerate}
\item  There exists a virtual representation $f:P\to Q$ with $\undim Q-\undim P=\alpha$ so that $f$ induces an isomorphism
$
f^\ast:\Hom_\Lambda(Q,M_\beta)\xrightarrow{\approx} \Hom_\Lambda(P,M_\beta).
$
\item There is a nonzero determinantal semi-invariant of {det-weight} $\beta$ on virtual representation space $\Vrep(\Lambda,\alpha)$.
\end{enumerate}
\end{lem}
\begin{proof}(1)$\implies$(2) It follows from (1) that 
$dim_K\Hom_\Lambda(Q,M_\beta)=dim_K\Hom_\Lambda(P,M_\beta)$  and therefore determinant of $h^*$ is defined for all $h\in Pres(\Lambda,\alpha)$, is non-zero for $h=f$ and is compatible with stabilization. Hence $\sigma =$ determinant is a (determinantal) semi-invariant on $\Vrep(\Lambda,\alpha)$.

(2)$\implies$(1) Given a determinantal virtual semi-invariant of {det-weight} $\beta$, we have an isomorphism $f^\ast:\Hom_\Lambda(Q,M)\cong \Hom_\Lambda(P,M)$ for some $M$ with $\undim M=\beta$. Since being an isomorphism is an open condition, $f^\ast$ must also be an isomorphism for $M=M_\beta$.
\end{proof}

\begin{lem}\label{lem 3.1.6: D(b) in Dss(b)}
 Let $\beta$ be a real Schur root. Then
$D_\ZZ(\beta)\subseteq D^{ss}_\ZZ(\beta)$.
\end{lem}

\begin{proof} Let $\alpha\in D_\ZZ(\beta)$. Then there is $f: P\to Q$ in $Pres(\Lambda,\alpha)$ such that $f^*: \Hom (Q, M_{\beta})\to \Hom(P,M_{\beta})$ is an isomorphism. So, $\brk{\alpha,\beta}=\dim_K\Hom_\Lambda(Q,M_\beta)-\dim_K\Hom_\Lambda(P,M_\beta)=0$. The induced map  $\Hom (Q, M_{\beta'})\to \Hom(P,M_{\beta'})$ is a monomorphism for all real Schur subroots $\beta'\subseteq \beta$. Therefore $\brk{\alpha,\beta'}\leq0$, i.e. stability condition (2) holds and $\alpha\in D^{ss}_\ZZ(\beta)$.
\end{proof}

\subsection{Perpendicular categories of $M_\beta$ and associated exceptional sequences}
To a real Schur root $\beta$ we associate an exceptional sequence $(M_\beta,E_1,\cdots\!,E_{n-1})$ which will play a crucial role in the proof of the theorem. We identify the $\Lambda$-module $M_\beta$ with its projective presentation in $\Vrep(\Lambda,\beta)$.

\begin{defn}\label{def 3.2.1: Vperp}
(a) For any $X$ in $\Vrep(\Lambda)$ let $^{\perp_V}\!X$ be the left $\Hom_{\Vrep(\Lambda)}$-, $\Ext^1_{\Vrep(\Lambda)}$-perpendicular category of $X$ in $\Vrep(\Lambda)$, i.e., 
$^{\perp_V}\!X$ is the full subcategory of $\Vrep(\Lambda)$ with objects $Y$ so that $\Hom_{\Vrep(\Lambda)}(Y,X)=0=\Ext_{\Vrep(\Lambda)}^1(Y,X)$. $X^{\perp_V}$ is defined similarly.

(b) For any $M\in mod\text-\Lambda$, let $^{\perp}\!M$ be the left $\Hom_{\Lambda}$-, $\Ext^1_{\Lambda}$-perpendicular category of $M$ in $mod\text-\Lambda$ with objects $N$ so that $\Hom_\Lambda(N,M)=0=\Ext^1_\Lambda(N,M)$. $M^\perp$ is defined similarly.
\end{defn}

The following lemma will relate perpendicular categories in $\Vrep(\Lambda)$ and in $mod\text-\Lambda$ allowing us to use some well known theorems for module categories.

%\begin{lem}  Let $^{\perp_V}\!M_\beta$ be the perpendicular category in $\Vrep(\Lambda)$ and let  $^{\perp}M_\beta$ be the left $\Hom_{\Lambda}$-, $\Ext^1_{\Lambda}$-perpendicular category of $M_\beta$ in $mod\text-\Lambda$. Then: $^{\perp_V}\!M_\beta \cap mod\text-\Lambda=\,^{\perp}\!M_\beta$.\end{lem}

\begin{lem}\label{lem 3.2.2: X-Vperp=|X|perp}
\emph{(a)} For any $\Lambda$-module $M$, $^{\perp_V}\!M \cap mod\text-\Lambda=\,^{\perp}\!M$. 

\emph{(b)} For any $X\in \Vrep(\Lambda)$, $X^{\perp_V}\cap mod\text-\Lambda=|X|^\perp$.
\end{lem}

\begin{defn}\label{def: 3.2.3: wide subcategory}
A \emph{wide subcategory} of $mod\text-\Lambda$ for any hereditary algebra $\Lambda$ is defined to be an extension closed full subcategory $\cW\subseteq mod\text-\Lambda$ which is abelian and exactly embedded (a sequence in $\cW$ is exact in $\cW$ if and only if it is exact in $mod\text-\Lambda$). A wide subcategory is said to have \emph{rank} $k$ if it is isomorphic to the module category of an hereditary algebra with $k$ simple objects.
\end{defn}

\begin{rem}\label{rem: properties of wide subcategories}
We need the following well-known properties of wide subcategories. \cite{Ingalls-Thomas}
\begin{enumerate}
\item Every finitely generated wide subcategory of $mod\text-\Lambda$ is isomorphic to the module category of an hereditary algebra.
\item If $M$ is a $\Lambda$-module whose components form an exceptional sequence with $k$ terms, the right $\Hom_\Lambda$- $\Ext_\Lambda^1$- perpendicular category $M^\perp$ of $M$ is a finitely generated wide subcategory of $mod\text-\Lambda$ of rank $n-k$. The same holds for $\,^\perp\!M$.
\item For any finitely generated wide subcategory $\cW$ of $mod\text-\Lambda$ we have $(^\perp\cW)^\perp=\cW$ and $\,^\perp(\cW^\perp)=\cW$.
\end{enumerate}
\end{rem}

In our case $^\perp M_\beta$ is a wide subcategory of rank $n-1$ since $M_\beta$ is indecomposable. Let $E_1,\!\cdots\!,E_{n-1}$ be the simple objects of $^\perp M_\beta$. These objects are exceptional and using Proposition \ref{exceptional sequences}(1) can be ordered in such a way that the following sequence is an exceptional sequence in $mod\text-\Lambda$:
\begin{equation}\label{eq: exceptional}
	(M_\beta,E_1,\cdots,E_{n-1}).
\end{equation}
By Lemma \ref{lem 3.2.2: X-Vperp=|X|perp} we also have:
%Using similar observations about the right perpendicular categories, we notice: 
$$(E_1\oplus\cdots\oplus\widehat{E_k}\oplus\cdots\oplus E_{n-1})^{\perp_V}\cap mod\text-\Lambda = 
(E_1\oplus\cdots\oplus\widehat{E_k}\oplus\cdots\oplus E_{n-1})^\perp,$$
and we use this to define, for each $k=1,\dots\!,n\!-\!1$, the following subcategories of $mod\text-\Lambda$ , 
\begin{equation}\label{eq: Wk}
	\cW_k:=(E_1\oplus\cdots\oplus\widehat{E_k}\oplus\cdots\oplus E_{n-1})^\perp \subset mod\text-\Lambda.
\end{equation}
These are wide subcategories of $mod\text-\Lambda$ of rank 2 which contains $M_\beta$ by definition of the $E_is$.

\begin{lem}\label{lem: Mb is not simple} Let $\beta$ be a real Schur root. Let $(E_1,\cdots,E_{n-1})$ be an exceptional sequence of simple objects of $^\perp M_\beta$ and let $P_k'$ be the projective cover of $E_k$ in $^\perp M_\beta\subset mod\text-\Lambda$. Then $P_k'$ is a projective $\Lambda$-module if and only if $M_\beta$ is a simple object in $\cW_k$.
\end{lem}

\begin{proof}  Several exceptional sequences will be created out of $(E_1,\cdots,E_{n-1})$ and will be used in the proof. Since all $E_i$ are simple objects and $P_k'$ is projective in  $^\perp M_\beta\subset mod\text-\Lambda$ it follows that $\Hom_{(^\perp M_\beta)}(P_k',E_i)=0$ for $i\neq k$ and also $Ext^1_{(^\perp M_\beta)}(P'_k,E_i)=0$. Therefore:\\
(a) $\quad\quad (E_1,\cdots,\widehat{E_k},\cdots,E_{n-1},P_k') \text{ is an exceptional sequence in } ^\perp\!M_\beta.$\\
Since $^\perp\!M_\beta
\hookrightarrow mod\text-\Lambda$ 
is exact embedding it follows that $\Hom_\Lambda(P_k',E_i)=0=\Ext^1_\Lambda(P_k',E_i)$ for all $i\neq k$. This
 together with the fact that $\{E_1,\cdots,\widehat{E_k},\cdots,E_{n-1},P_k'\} \subset ^\perp\!M_\beta$ implies:\\
(b) $\quad\quad (M_{\beta}, E_1,\cdots,\widehat{E_k},\cdots,E_{n-1}, P_k') \text{ is an exceptional sequence in }mod\text-\Lambda.$\\
After applying Proposition  \ref{prop: properties of exceptional sequences}(6) to the exceptional sequence (a), one obtains:\\
(c) $\quad\quad (I_k', E_1,\cdots,\widehat{E_k},\cdots,E_{n-1}) \text{ is an exceptional sequence in } ^\perp M_\beta,$ and \\
(d) $\quad\quad (M_{\beta}, I_k', E_1,\cdots,\widehat{E_k},\cdots,E_{n-1}) \text{ is an exceptional sequence in }mod\text-\Lambda.$\\
After applying Proposition  \ref{prop: properties of exceptional sequences}(6) to the exceptional sequence (b), one obtains:\\
(e) $\quad\quad (X, M_{\beta}, E_1,\cdots,\widehat{E_k},\cdots,E_{n-1}) \text{ is an exceptional sequence in } mod\text-\Lambda,$ 
where \\$X=\tau_{\Lambda} P'_k$ if and only if $P'_k$ is not a projective $\Lambda$-module, and $X$ is the injective envelope of $P_k'/rad P_k'$ if and only if $P_k'$ is a projective $\Lambda$-module. \\
(f) $\quad\quad (M_{\beta}, I_k')$ and $(X, M_{\beta})$ are exceptional sequences in $\cW_k$. This follows from (d), (e) and definition of $\cW_k$. \\
(g)  There is a $\cW_k$-irreducible map $M_{\beta} \to I_k'$ if and only if $M_{\beta} \oplus I_k'$ is not semi-simple, since rank$(\cW_k)=2$. \\
(h) There is a $\cW_k$-irreducible map  $X \to M_{\beta}$ if and only if $X \oplus M_{\beta}$ is not semi-simple, since rank$(\cW_k)=2$.\vs2

\underline{Claim 1:} If $M_\beta$ is not simple in $\cW_k$ then $P_k'$ is not a projective $\Lambda$-module.\\
Proof: Since $M_\beta$ is not simple it follows from (g) and (h) that there is an almost split sequence $X\into M_\beta^m\onto I_k'$ in $\cW_k$. Since $P_k'\in\,^\perp M_\beta$ we have $\Ext^1_\Lambda(P_k',X)\cong \Hom_\Lambda(P_k',I_k')\neq0$. So, $P_k'$ is not projective in $mod\text-\Lambda$.\vs2

\underline{Claim 2:}  If $M_\beta$ is simple in $\cW_k$ then $P_k'$ is a projective $\Lambda$-module.\\
Proof: If $M_\beta$ is simple in $\cW_k$, either $M_\beta$ is simple injective or simple projective. \vs2
\underline{Case 2a:} If $M_\beta$ is a simple injective object in $\cW_k$, then there is no $\cW_k$-irreducible map $M_{\beta} \to I_k'$. So it follows by (g) that $I_k'$ is a simple projective object in $\cW_k$. Since $\cW_k$ has rank=2, there are only two simple objects. Therefore $X$ is not simple and by (h) there is a $\cW_k$-irreducible map  $X \to M_{\beta}$. Since  $M_{\beta}$ is simple injective in $\cW_k$ it follows that $X$ is injective envelope of the simple object $I'_k$ in $\cW_k$. If $X$ is injective $\Lambda$-module, it follows by (e) that $P_k'$ is a projective $\Lambda$-module. If $X$ is not injective $\Lambda$-module, then $X=\tau_{\Lambda}P'_k$. But in this case there is a non-zero composition of $\cW_k$-, and therefore $\Lambda$-maps $P'_k \to S_k=I'_k \hookrightarrow X =\tau_{\Lambda}P'_k$. However this would imply $\Ext^1_{\Lambda}(P'_k,P'_k)\neq 0$ giving a contradiction to the fact that $P'_k$ is rigid. Hence, $P_k'$ is a projective $\Lambda$-module.\vs2
\underline{Case 2b} If $M_\beta$ is simple projective in $\cW_k$ then there is no $\cW_k$-irreducible map $X\to M_\beta$ and therefore $X$ must be simple $\cW_k$ object by (h), hence simple injective. Since the rank of $\cW_k$ is 2, it follows that $I'_k$ is not simple, hence there is a $\cW_k$ irreducible map $M_{\beta} \to I_k'$. Therefore $I'_k$ is projective $\cW_k$ object, projective cover of $X$. So, there is a short exact sequence in $\cW_k$: $0\to M_{\beta}^m\to I_k'\to X\to 0$ where $m\ge1$. Since $\Hom_\Lambda(P_k',M_\beta)=0$ and $\Hom_\Lambda(P_k',I_k')\neq 0$ we have a nonzero $\Lambda$ morphism $P_k'\to X$. Therefore, $P_k'$ is a projective $\Lambda$-module by (e). This proves the proof of the lemma.
\end{proof}

\subsection{Subsets $\Delta^+(\beta)\subseteq \Delta(\beta)\subseteq D_\ZZ(\beta)\subseteq D^{ss}_\ZZ(\beta)$} 
In this subsection we define two new subsets, which will be used in the proof of the Virtual Stability Theorem' \ref{virtual stability theorem'}, i.e., we will prove $D_\ZZ(\beta)= D^{ss}_\ZZ(\beta)$.

\begin{defn} Let $\beta$ be a real Schur root and let $(M_\beta,E_1,\cdots,E_{n-1})$ be the exceptional sequence as defined in Equation \eqref{eq: exceptional}. Let:
\[
	\Delta^+(\beta):=\{{\,}\sum_{1\le i\le n-1}  k_i\undim E_i\,:\, k_i\in\NN\,\}.
\]
\end{defn}
\begin{lem}\label{lem: D+ contains all integer points in its convex hull}
The set $\Delta^+(\beta)$ contains all integer points in its convex hull in $\RR^n$.
\end{lem}

\begin{proof}
Since $(M_\beta,E_1,\cdots,E_{n-1})$ is an exceptional sequence, by Proposition \ref{prop: properties of exceptional sequences}, every vector in $\ZZ^n$ can be expressed uniquely as an integer linear combination of the vectors $\undim E_i$ and $\beta$. Since all elements in the convex hull of $\Delta^+(\beta)$ can be written as nonnegative real linear combinations of the vectors $\undim E_i$ it follows that when the integer points in this convex hull are written in this way, the nonnegative coefficients are necessarily integers. So, they are nonnegative integers.
\end{proof}

\begin{rem}\label{rem 3.3.3}
The following are simple useful facts about perpendicular categories and the set $D_\ZZ(\beta)$.
\begin{enumerate}
\item Let $P(\gamma_\ast)\in \Vrep(\Lambda)$. If $P(\gamma_\ast)\in\, ^{\perp_V}\!M_\beta$ then $\undim P(\gamma_\ast)\in D_\ZZ(\beta)$ 
\item Let $N\in mod\text-\Lambda$. If $N\in\, ^\perp M_\beta$ then $\undim N\in D_\ZZ(\beta)$.
\item Let $P_j$ be an indecomposable projective $\Lambda$-module. Then\\ $P_j\in\,^\perp M_\beta\iff P_j\in\,^{\perp_V}\! M_\beta\iff P_j[1]\in\,^{\perp_V}\! M_\beta\iff \beta_j=0$.
\end{enumerate}
\end{rem}

\begin{defn}\label{def: Delta beta}
Let $J_\beta=\{j\in \ZZ\,|\, \beta_j=0\}$. Then% the following set of vectors is contained in $D_\ZZ(\beta)$.
\[%begin{equation}\label{eq: Delta beta}
	\Delta(\beta):=\{\sum_{1\le i\le n-1} k_i\undim E_i+\sum_{j\in J_\beta} \ell_j \undim P_j\,:\, k_i\in\NN, \ell_j\in \ZZ\}\subseteq D_\ZZ(\beta).
\]%end{equation}
\end{defn}

%We have the following easy observation.
\begin{prop}\label{prop 3.3.5: Delta subset D}
$%begin{equation}\label{eq: Delta subset D}
	\Delta^+(\beta)\subseteq \Delta(\beta)\subseteq D_\ZZ(\beta)\subseteq D^{ss}_\ZZ(\beta).
$ %end{equation}
\end{prop}

\begin{proof}
The first inclusion follows from the definitions. The second follows from Remark \ref{rem 3.3.3}. The third is Lemma \ref{lem 3.1.6: D(b) in Dss(b)}.
\end{proof}
To prove the stability theorem we will show that $D^{ss}_\ZZ(\beta)\subseteq \Delta(\beta)$ and therefore the last three sets in \ref{prop 3.3.5: Delta subset D} are equal. To do this we first consider  the case when $J_\beta$ is empty, i.e., when $\beta$ is sincere.

\subsection{Sincere case}\label{ss: sincere case} When $\beta$ is sincere we have $\Delta^+(\beta)=\Delta(\beta)$. Thus we are reduced to showing that $D^{ss}_\ZZ(\beta)\subseteq \Delta^+(\beta)$. Note that $^\perp M_\beta\subseteq mod\text-\Lambda$ when $\beta$ is sincere.

\begin{lem}\label{lem: stability theorem hold when beta is sincere}
If $\beta$ is sincere then $D^{ss}_\ZZ(\beta)= \Delta^+(\beta)$ and therefore $D^{ss}_\ZZ(\beta)=D_\ZZ(\beta)$.
\end{lem}

\begin{proof}
Since $\beta$ is sincere, it follows that there are no projective $\Lambda$-modules in $^\perp M_\beta$ and therefore none of the $P'_k$ are projective $\Lambda$-modules. This implies that $M_\beta$ is not a simple object in $\cW_k$ (as defined in equation \ref{eq: Wk}) for each $k=1,\dots\!,n\!-\!1$ by Lemma \ref{lem: Mb is not simple}. We will use this fact to construct certain Schur subroots $\gamma_k\subseteq \beta$. \vs2
\underline{Claim 1}. For each $k=1,2,\cdots,n-1$ there is a real Schur subroot $\gamma_k$ of $\beta$ so that
\begin{enumerate}
\item $\brk{\undim E_i,\gamma_k}=0$ if $i\neq k$
\item $\brk{\undim E_k,\gamma_k}< 0$.
\end{enumerate}
\underline{Construction of $\gamma_k$:}
For each $k$, the category $\cW_k=(E_1\oplus\cdots\oplus\widehat{E_k}\oplus\cdots\oplus E_{n-1})^\perp\subset mod\text-\Lambda$ is a wide subcategory of rank 2 which contains $M_\beta$ by definition of the $E_i$s.
Let $R_k,S_k$ be the simple objects of $\cW_k$. Since $M_\beta$ is not simple, there exists a nontrivial extension of $S_k$ by $R_k$ (or $R_k$ by $S_k$): \[
	S_k^p\into M_\beta\onto R_k^q
\]
where $p,q\ge1$. Let $\gamma_k=\undim S_k$. Then $\gamma_k$ is a proper real Schur subroot of $\beta$.\vs2
\noi \underline{Properties of $\gamma_k$:}\\
(1) $\brk{\undim E_i,\gamma_k}=\brk{\undim E_i,\undim S_k}=0\text{ for all }i\neq k$
since $S_k\in \cW_k=(\oplus_{i\neq k} E_i)^\perp$. \\
(2) $\brk{\undim E_k,\gamma_k}< 0$. We prove this in two steps:

Step 1:
Since $\Hom_\Lambda(E_k,M_\beta)=0$ we must also have $\Hom_\Lambda(E_k,S_j)=0$. Therefore
$
	\brk{\undim E_k,\gamma_k}=\brk{\undim E_k,\undim S_k}\le0. 
$

Step 2: $\brk{\undim E_k,\undim S_k}\neq 0$ since all vectors $z$ satisfying $\brk{\undim E_i,z}=0$ for all $i$ are scalar multiples of $\beta$ which is not possible since $\gamma_k\subsetneq \beta$. This finishes the proof of Claim 1.\vs2
\noi\underline{Claim 2:}
$D^{ss}_\ZZ(\beta)\subseteq \Delta^+(\beta)$. \\
Proof of claim 2:
 Since $(M_\beta,E_1,\cdots,E_{n-1})$ is an exceptional sequence, by Proposition \ref{prop: properties of exceptional sequences}, every vector in $\ZZ^n$ can be expressed uniquely as an integer linear combination of $\beta$ and the vectors $\undim E_i$. Let $\alpha\in D^{ss}_\ZZ(\beta)$. Then $\alpha$ is an integer linear combination of the roots $\undim E_i$, say    $\alpha=\sum_{i=1}^{n-1} a_i\undim E_i$. The stability conditions which define $D^{ss}_\ZZ(\beta)$ (Definition \ref{defn: D(b), Dss(b), DZ(b), DZss(b)}) and $\brk{\undim E_i,\gamma_k}=0$   imply
\[
	\brk{\alpha,\gamma_k}=a_k\brk{\undim E_k,\gamma_k}\le0.
\]
Since $\brk{\undim E_k,\gamma_k}<0$ by (2), this implies that $a_k\ge0$ for each $k$. Since all $a_i$ are integers it follows that $\alpha\in \Delta^+(\beta)$. 
 
 This finishes the proof that $D^{ss}_\ZZ(\beta)=D_\ZZ(\beta)$ for $\beta$ a sincere Schur root.
\end{proof}

\subsection{Non-sincere case}\label{ss: non-sincere case} (Proof of Theorem \ref{virtual stability theorem'}) In this subsection we will prove $D^{ss}_\ZZ(\beta)=D_\ZZ(\beta)$ for all real Schur roots $\beta$. In order to deal with non-sincere roots, we need the following lemma. 
\begin{lem}\label{lem: reduction to sincere case} Let $Q$ be a quiver, $\beta$ a real Schur root which is not sincere, with $\beta_j=0$. Let $Q_{(j)}$ be the quiver obtained from $Q$ by deleting vertex $j$ and all adjacent edges. Then:\begin{enumerate}
\item $D^{ss}_\ZZ(Q,\beta)=\{\alpha+m\undim P_j\,|\, \alpha\in D^{ss}_\ZZ(Q_{(j)},\beta), m\in\ZZ\}=D^{ss}_\ZZ(Q_{(j)},\beta)+\ZZ \undim P_j$, 
\item 
$D_\ZZ(Q,\beta)=\{\alpha+m\undim P_j\,|\,\alpha\in D_\ZZ(Q_{(j)},\beta), m\in\ZZ\}=D_\ZZ(Q_{(j)},\beta)+\ZZ \undim P_j$.
\end{enumerate}
\end{lem}

\begin{proof} (1) Since $P_j$ is one dimensional (over $F_j$) at vertex $j$, for any integer vector $\alpha\in \ZZ^n$, $\alpha-\alpha_j\undim P_j$ lies in $\ZZ^{n-1}$. Since $\brk{\undim P_j,\beta'}=0$ for all subroots $\beta'\subseteq \beta$, it follows that  $\alpha\in D^{ss}_\ZZ(\beta)$ if and only if $\alpha-\alpha_j\undim P_j$ lies in $D^{ss}_\ZZ(Q_{(j)},\beta)$. 

(2) Since $P_j$ and $P_j[1]$ are in the perpendicular category $^{\perp_V}\!M_\beta$, the same is true for $D_\ZZ(Q,\beta)$: $\alpha\in D_\ZZ(Q,\beta)$ iff there is a virtual representation $P(\gamma_\ast)$ of dimension $\alpha$ which lies in $^{\perp_V}\!M_\beta$. Then $P(\gamma_\ast)\oplus P_j^{|\alpha_j|}$ and $P(\gamma_\ast)\oplus P_j[1]^{|\alpha_j|}$ are virtual representations in $^{\perp_V}\! M_\beta$ and one of them has dimension $\alpha-\alpha_j\undim P_j$ which lies in $D_\ZZ(Q_{(j)},\beta)$. Conversely, $D_\ZZ(Q_{(j)},\beta)+\ZZ\undim P_j$ is contained in $D_\ZZ(\beta)$. So, they are equal.
\end{proof}

\begin{proof}[Proof of Virtual Stability Theorem \ref{virtual stability theorem'}] (when the field $K$ is infinite). The proof is by induction on the number of vertices of the quiver $Q$. Let $\beta$ be a real Schur root. If $\beta$ is sincere then $D^{ss}_\ZZ(\beta)=D_\ZZ(\beta)$ by Lemma \ref{lem: stability theorem hold when beta is sincere}.

If $\beta$ is not sincere then by Lemma \ref{lem: reduction to sincere case}(1): $D^{ss}_\ZZ(Q,\beta)$, is the set of all integer vectors of the form $\alpha+m\undim P_j$ where $\alpha$ lies in $D^{ss}_\ZZ(Q_{(j)},\beta)$ and $m\in\ZZ$. Since $Q_{(j)}$ has $n\!-\!1$ vertices, by induction  $D^{ss}_\ZZ(Q_{(j)},\beta)=D_\ZZ(Q_{(j)},\beta)$. Then by Lemma \ref{lem: reduction to sincere case}(2) it follows that $D^{ss}_\ZZ(\beta)=D_\ZZ(\beta)$. This finishes the proof of Theorem \ref{virtual stability theorem'}.
\end{proof} 
The extended version of the stability theorem includes also the equality $D^{ss}_\ZZ(\beta)=\Delta(\beta)$ which was proved for $\beta$ sincere in Lemma \ref{lem: stability theorem hold when beta is sincere}, which we now extend to the general real Schur root $\beta$.

\begin{prop}\label{prop 3.5.2: extended stability theorem}
Let  $\beta$ be a real Schur root. Then
\[
	D^{ss}_\ZZ(\beta)=D_\ZZ(\beta)=\Delta(\beta).
\]
%where the terms are defined in Definition \ref{defn: D(b), Dss(b), DZ(b), DZss(b)} and Definition \ref{def: Delta beta}.
\end{prop}

\begin{proof}
When $\beta$ is sincere, this is Lemma \ref{lem: stability theorem hold when beta is sincere}. So, suppose $\beta$ is not sincere. Let $j\in J_\beta$ and let $\Lambda_{(j)}=\Lambda/\Lambda e_j\Lambda$. Then the quiver $Q_{(j)}$ as in Lemma \ref{lem: reduction to sincere case} is the quiver of $\Lambda_{(j)}$. Then, by induction on $n$ we have
\[
	D^{ss}_\ZZ(Q_{(j)},\beta)=\Delta(Q_{(j)},\beta)=\{\sum_{1\le i\le n-2} a_i\undim E_i'+\sum_{k\in J_\beta,k\neq j} b_k \undim P_k\,:\, a_i\in \NN, b_k\in \ZZ\},
\]
where $E_i'$ are the simple objects of $^\perp M_\beta$ in $mod\text-\Lambda_{(j)}$. By Lemma \ref{lem: reduction to sincere case} we conclude that
\[
	D^{ss}_\ZZ(\beta)=\Delta(Q_{(j)},\beta)+\ZZ \undim P_j=\{\sum_{1\le i\le n-2} a_i\undim E_i'+\sum_{k\in J_\beta} b_k \undim P_k\,:\, a_i\in \NN, b_k\in \ZZ\}.
\]
Since each $E_i'$ and each $P_k$ is a module in $^\perp M_\beta$, their dimension vectors are nonnegative $\ZZ$-linear combinations of the dimension vectors of the simple objects $E_i$ of $^\perp M_\beta$. Therefore,
\[
	D^{ss}_\ZZ(\beta)\subseteq \{\sum_{1\le i\le n-1} a'_i\undim E_i+\sum_{k\in J_\beta} b'_k \undim P_k\,:\, a'_i\in \NN, b'_k\in \ZZ\}=\Delta(\beta).
\]
By Proposition \ref{prop 3.3.5: Delta subset D}, the opposite inclusion $\Delta(\beta)\subseteq D_\ZZ(\beta)\subset D^{ss}_\ZZ(\beta)$ holds. %This finishes the proof of the proposition.
\end{proof}

\subsection{Extension to arbitrary fields $K$}\label{sec: extend to finite K} (Proof of Theorem \ref{thm 3.1.1: virtual stability theorem}) Suppose that the ground field $K$ is finite. Then, we still have the trivial implication $(1)_\Lambda\then (3)_\Lambda$. Since $-\otimes_KK(t)$ is an exact functor, $(3)_\Lambda\then (3)_{\Lambda(t)}$ which we have shown to be equivalent to $(1)_{\Lambda(t)}$ and $(2)$ which does not refer to $K$. It remains to show that these imply $(1)_\Lambda$.

Recall that, for every real Schur root $\beta$, the simple objects $E_i$ and projective objects $P_j$ of $^\perp M_\beta\cap mod\text-\Lambda(t)$ are exceptional $\Lambda(t)$-modules. By Theorem \ref{thm: extension to finite fields}, these are isomorphic to $E_i'(t),P_j'(t),M_\beta'(t)$ for unique exceptional $\Lambda$-modules $E_i',P_j',M_\beta'$. Then, for any $\alpha\in D^{ss}_\ZZ(\beta)=\Delta_{\Lambda(t)}(\beta)$, we have
\[
	\alpha=\sum k_i\undim E_i'(t)+\sum \ell_j\undim P_j'(t)=\sum k_i\undim E_i'+\sum \ell_j\undim P_j'
\]
where $k_i\in\NN$ and $\ell_j\in\ZZ$. So, the direct sum of the $\Lambda$-modules $E_i'^{k_i}$, projective modules $P_j'^{\ell_j}$ for $\ell_j\ge0$ and the shifted projective modules $P_j'^{|\ell_j|}[1]$ for $\ell_j\le0$ give a virtual representation $P(\gamma_\ast)=(f:P(\gamma_1)\to P(\gamma_0))$ so that, $\Hom_\Lambda(f,M_\beta'):\Hom_\Lambda(P(\gamma_0),M_\beta')\cong \Hom_\Lambda(P(\gamma_1),M_\beta')$. This shows that $(2)\then (1)_\Lambda$. So, Theorem \ref{thm 3.1.1: virtual stability theorem} holds for finite $K$. This completes the proof of Theorem \ref{thm 3.1.1: virtual stability theorem} for all finite dimensional hereditary algebras over any field.

\begin{rem} When $K$ is any perfect field, $\Lambda\otimes_K\overline K$ is an hereditary algebra over the algebraically closed field $\overline K$ and it should be possible to extend the Virtual Stability Theorem from \cite{IOTW1} to $\Lambda$. However, if $K$ is not perfect and $F_i=K(a^{1/p})$ is the division algebra at vertex $i$ then the socle of $S_i\otimes_K\overline K$ has infinite projective dimension. So, $\Lambda\otimes_K\overline K$ is not hereditary in that case. That is the reason we did not take this approach.
\end{rem}

%\newpage

\section{$c$-vectors and semi-invariants}\label{sec: c-vector theorem}

\setcounter{subsection}{-1}

In this section we use the Virtual Stability Theorem for semi-invariants (\ref{thm 3.1.1: virtual stability theorem}) to prove two fundamental theorems relating {determinantal weight}s of semi-invariants, {cluster tilting objects} and $c$-vectors corresponding to a cluster tilting object. Theorem \ref{thm 4.1.5: clusters and SIs} gives the relation between semi-invariants and {cluster tilting objects}. Theorem \ref{c-vector thm}, relates the semi-invariants of a cluster tilting object to the corresponding $c$-vectors. (Section \ref{ss: def of c-vectors}) These theorems (\ref{thm 4.1.5: clusters and SIs} and \ref{c-vector thm}) are inspired by work of Speyer and Thomas \cite{ST}. In type $A$ these theorems are re-interpreted in terms of finite and infinite trees in \cite{IOs}, \cite{ITW}.

\subsection{Preview} We illustrates Theorems \ref{thm 4.1.5: clusters and SIs} and \ref{c-vector thm} in an example.

\begin{rem}\label{rem 4.0.1: Pres = C}
Since $c$-vectors come from cluster theory, we use the well-known language of cluster categories \cite{BMRRT}. We recall from Corollary \ref{cor 2.3.7: Pres L=CL} that there is a bijection between presentations in $Pres(\Lambda)$ and objects of the cluster category $\cC_\Lambda$ so that the presentation of any $\Lambda$-module is sent to that module, and the shifted projective $P[1]$ is sent to the same shifted projective object. Objects of $Pres(\Lambda)$ are Ext-orthogonal if and only if the corresponding objects of the cluster category $\cC_\Lambda$ are Ext-orthogonal since:
\[
	\Ext^1_{\cC_\Lambda}(X,Y)%=\Ext^1_{Pres(\Lambda)}(X,Y)\oplus \Ext^1_{Pres(\Lambda)}(X,\tau^{-1}Y[1])
	\cong \Ext^1_{Pres(\Lambda)}(X,Y)\oplus D\Ext^1_{Pres(\Lambda)}(Y,X)
\]
where $D=\Hom_K(\cdot,K)$.
%that  as follows.  This bijection has the property that $\Ext_{\cC_\Lambda}^1(X',Y')=0=\Ext_{\cC_\Lambda}^1(Y',X')$ if and only if $\Ext_{\Vrep(\Lambda)}^1(X,Y)=0=\Ext_{\Vrep(\Lambda)}^1(Y,X)$, where $X',Y'$ are objects in $\cC_\Lambda$ corresponding to $X,Y\in \Vrep(\Lambda)$. This means we can use properties of cluster tilting objects in a cluster category to work with partial cluster tilting objects in the virtual representation category.
\end{rem}

\begin{eg}\label{eg 4.0.2: A3 three circles}
The figure below illustrates the relation between cluster tilting objects, domains of semi-invariants and $c$-vectors for the quiver $Q=1\ot 2\ot 3$ of type $A_3$. The picture indicates the unit sphere $S^2$ in $\RR^3$ stereographically projected to $\RR^2$ with the center being the dimension vector of $\Lambda=P_1\oplus P_2\oplus P_3$. The 9 vertices are the (normalized dimension vectors of) indecomposable objects of the cluster category $\cC_\Lambda$ labeled as objects of $Pres(\Lambda)\cong \Vrep(\Lambda)$. The 6 lines are $D(\beta)\cap S^2$, where $\beta$ are the 6 positive roots, dimension vectors of indecomposable modules. There are 14 regions which are spherical triangles whose vertices are components of cluster tilting objects. For example, the upper left triangle has vertices $T_1=S_1=P_1$, $T_2=P_2[1]$, $T_3=P_3[1]$ with walls $D(\beta_i)$, where $\beta_i$ are given by Theorem \ref{thm 4.1.5: clusters and SIs}. The second theorem of this section, Theorem \ref{c-vector thm}, shows that the $c$-vectors corresponding to each cluster tilting object $\bigoplus T_i$ are, up to sign, equal to the det-weights $\beta_i$ of the semi-invariants defined on the walls $D(\beta_i)$ of the conical simplex spanned by $\undim T_i$. For example, in the upper left spherical simplex, the $c$-vectors are $c_1=-\beta_1=-e_1$, $c_2=\beta_2=e_2$, $c_3=\beta_3=e_3$. The sign of $c$-vectors is positive on the outside of each curve and negative on the inside.
\end{eg}

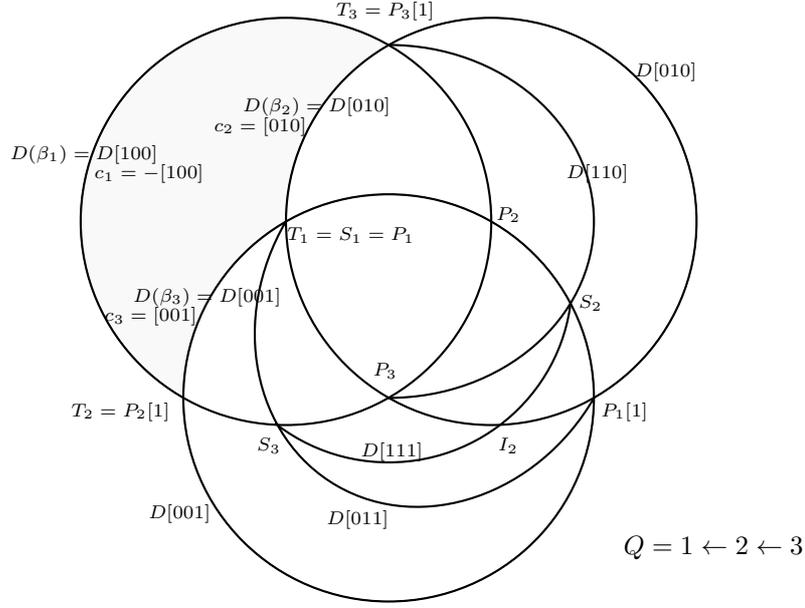
\begin{center}
\begin{tikzpicture}[scale=.9] % Huge YYYYYYYY
%\draw[help lines=.2,color=red] (-5.75,-5) grid (5,5);	\draw[help lines=1, thick,color=blue] (-5.75,-5) grid (5,5);	\foreach \x in {-5,...,5} \draw (\x,-.2) node{\x};	\foreach \y in {-5,...,5} \draw (-.2,\y) node{\y};	\clip (-7.75,-5.5) rectangle (6.25,4.8);
%	\draw[fill] (0,0) circle [radius=2pt];
%	\draw[thick] (-.5,-5);% node{Figure 1: A picture defines a group.};
		\draw[fill=gray!05!white,draw=black] (-2.25,1.3) circle [radius=3cm]; % circle 1
		\draw[ fill=white] (.75,1.3) circle [radius=3cm]; % circle 2
		\draw[ fill=white] (-.75,-1.3) circle [radius=3cm]; % circle 3
		\draw[ thick] (-2.25,1.3) circle [radius=3cm]; % circle 1
		\draw[ thick] (.75,1.3) circle [radius=3cm]; % circle 2
		\draw[ thick] (-.75,-1.3) circle [radius=3cm]; % circle 3
		\begin{scope}
		\clip (-.75,-2) rectangle (4.25,4.5);
		\draw[ thick] (-.75,1.3) ellipse [x radius=3cm,y radius=2.6cm]; % 21 semicircle
		\end{scope}
%	\draw (-2.25,-2.8) rectangle (3.75,2.8);
%\draw[fill, color=blue] (0,0) circle [radius=2pt];
		\begin{scope}[rotate=60]
		\clip (0,-2.7) rectangle (-3.1,2.7);
		\draw[ thick] (0,0) ellipse [x radius=3cm,y radius=2.6cm]; % 32 semicircle
		\end{scope}
%		\draw[thick, color=red] (-.75,.43) circle [radius=2.68cm];
\begin{scope}
\clip (-2.4,-3) rectangle (2,.1);
		\draw[ thick] (-.75,.43)  circle [radius=2.68cm]; % 321 arc
		\end{scope}
%	\draw[font=\huge,color=red] (0,0) node{C};
\draw (-5.2,2) node[above]{\tiny$  D(\beta_1)=D[100]$} node[right]{\tiny$ c_1=-[100]$};%\mat{1\\0\\0}$};
\draw (3.3,3.5) node{\tiny$ D[010]$};
%(-2,2.2)node[left]{\tiny $c_2=[010]$};%\mat{0\\1\\0}$};
\draw (-3.2,-3) node[left]{\tiny$  D[001]$};% (-3,.8)node[left]{\tiny$ c_3=[001]$};%\mat{0\\0\\1}$};
\draw (2.3,2) node{\tiny$ D[110]$};
\draw (-1.2,-3.1) node{\tiny$ D[011]$};% (-1.9,-.8)node[left]{$\tiny c_2'=\mat{0\\1\\1}$};
\draw (-.7,-2.1) node{\tiny$ D[111]$};
\draw (-1.3,1.1) node{\tiny$T_1=S_1=P_1$};
\draw (1,1.4) node{\tiny$P_2$};
\draw (2.2,0.1) node{\tiny$S_2$};
\draw (-.8,-.9) node{\tiny$P_3$};
\draw (-2.5,-2) node{\tiny$S_3$};
\draw (-3.8,-1.5) node[left]{\tiny$T_2=P_2[1]$};
\draw (2.7,-1.5) node{\tiny$P_1[1]$};
\draw (-.8,4.4) node{\tiny$T_3=P_3[1]$};
\draw (1,-2) node{\tiny$I_2$};
\draw (4,-3.5) node{\small$Q=1\leftarrow 2\leftarrow 3$};
\draw(-1.8,2.7) node[above]{\tiny$D(\beta_2)=D[010]$}node[left]{\tiny $c_2=[010]$} ;
\draw(-3.4,-.1) node[above]{\tiny$D(\beta_3)=D[001]$} (-3.4,-.1) node[left]{\tiny$ c_3=[001]$};
\end{tikzpicture}
\end{center}

\subsection{Structure of semi-invariant domains}

If $v_1,\cdots,v_k$ are vectors in $\RR^n$, the \emph{conical polyhedron} spanned by the $v_i$ is the set of all nonnegative linear combinations of the $v_i$. We will denote it by $C(v_1,\cdots,v_k)$. If the vectors $v_i$ are linearly independent we call the conical polyhedron a \emph{conical simplex}.

{For each real Schur root $\beta$, the domain $D(\beta)$ of the determinantal semi-invariant with {det-weight} $\beta$} is equal to the codimension-one conical polyhedron $\Delta(\beta)$ in $\RR^n$ by Proposition \ref{prop 3.5.2: extended stability theorem}. This is a conical polyhedron since it is the set of nonnegative linear combinations of the vectors $\undim E_i$, $\undim P_j$ and $-\undim P_j$.

\begin{rem}\label{rem 4.1.1: consider objects T to be in Pres}
(1) We consider objects of $\cC_\Lambda$ as objects in $Pres(\Lambda)\cong \Vrep(\Lambda)$ following Remark \ref{rem 4.0.1: Pres = C}, Corollary \ref{cor 2.3.7: Pres L=CL} and Proposition \ref{prop 2.3.2: stabilization is an equivalence}, and, for $X$ in $Pres(\Lambda)$, denote by $X^{\perp_V}$, $^{\perp_V}\!X$ perpendicular categories of $X$ in $Pres(\Lambda)$. \\
(2) Recall that $|X|^\perp=X^{\perp_V}\cap\, mod\text-\Lambda$ by Lemma \ref{lem 3.2.2: X-Vperp=|X|perp}, where  $M^\perp$  and $^\perp\!M$ are $\Hom_\Lambda$- $\Ext_\Lambda^1$-  perpendicular categories of $\Lambda$-module $M$ in $mod\text-\Lambda$ where $|M|=M$ for modules $M$ and $|P[1]|=P$ for shifted projective modules $P[1]$.\\
(3) Let $R$ be a rigid object in $Pres(\Lambda)$. Then 
\[
M_\beta\in |R|^{\perp}\iff M_\beta\in R^{\perp_V}\iff\undim R\in D(\beta).
\]
This follows from Theorem \ref{thm 2.3.11: virtual generic decomposition theorem}, the definition of $D(\beta)$ and (2) above.\\
(4)  Let $v\in D(\beta)$. Then, by the Virtual Stability Theorem \ref{thm 3.1.1: virtual stability theorem}, $v$ lies in the interior of $D(\beta)$ if and only if $\brk{v,\beta'}<0$ for all proper real Schur subroots $\beta'\subsetneq \beta$.
\end{rem}

If $T_0=T_1\oplus \cdots\oplus T_k$ is a partial cluster tilting object in $\cC_\Lambda$, the dimension vectors $\undim T_i$ are linearly independent. So they span the conical simplex $C(\undim T_1,\cdots,\undim T_k)$ which we abbreviate by $C(T_0)$.

\begin{lem}\label{lemma when T0 lies in D(b)}
Let $T_0=T_1\oplus \cdots \oplus T_k$ be any partial cluster tilting object. Then \\
\emph{(a)} $|T_0|^\perp$ is isomorphic to $mod\text-\Gamma$ where $\Gamma$ is an hereditary algebra with $n-k$ simple objects. \\
\emph{(b)} The conical simplex $C(T_0)$ is contained in $D(\beta)$ if and only if the single vector $\undim T_0=\sum_{i=1}^k \undim T_i$ lies in $D(\beta)$. \\
\emph{(c)} Let $M_\beta$ be an exceptional module in $|T_0|^\perp$. Then $\undim T_0$ lies in the interior of $D(\beta)$ if and only if $M_\beta$ is a simple object of $|T_0|^\perp$.
\end{lem}

\begin{proof} (a) follows from Remark \ref{rem: properties of wide subcategories}.

(b) If $\undim T_i\in D(\beta)$ then $\undim T_0\in D(\beta)$ since $D(\beta)$ is convex. Conversely, suppose $\undim T_0\in D(\beta)$. Then, by the Virtual Stability Theorem \ref{thm 3.1.1: virtual stability theorem} there exists an object $V\in Pres(\Lambda)$ with $\undim V=\undim T_0$ which admits a determinantal semi-invariant of det weight $\beta$. Since this is an open conditions, the generic object of this dimension has the same property. This is $T_0$. So, $T_0\in \,^{\perp_V}M_\beta$ which implies that each $T_i\in \,^{\perp_V}M_\beta$. So, $C(T_0)\subseteq D(\beta)$.

(c) Suppose $M_\beta$ is a simple object of $|T_0|^\perp$. Suppose that $\undim T_0$ does not lie in the interior of $D(\beta)$. Then $\undim T_0\in\partial D(\beta)$. So, $\brk{\undim T_0,\beta'}=0$ for some proper real Schur subroot $\beta'\subsetneq\beta$. Then, by the Virtual Stability Theorem \ref{thm 3.1.1: virtual stability theorem}, $\undim T_0$ lies in $D(\beta')$ and therefore $M_{\beta'}$ is an object of $|T_0|^\perp$ by (b). But $M_{\beta'}$ is a subobject of $M_\beta$ contradicting the assumption that $M_\beta$ is simple in $|T_0|^\perp$. So, $\undim T_0$ is in the interior of $D(\beta)$.

Suppose $M_\beta$ is not simple in $|T_0|^\perp$. Then $M_\beta$ contains a simple subobject $M_{\beta''}\in |T_0|^\perp$. Any such ${\beta''}$ is a real Schur root. So, $\brk{\undim T_0,{\beta''}}=0$. Therefore, $\undim T_0\in D(\beta)\cap D({\beta''})\subset \partial D(\beta)$. So, $\undim T_0\in\partial D(\beta)$ when $M_\beta$ is not simple in $|T_0|^\perp$.
\end{proof}

\begin{rem}\label{rem 4.1.3}
Let $M_{\alpha_1},\cdots,M_{\alpha_{n-k}}$ be the simple objects of $|T_0|^\perp$. Then, the vector $\undim T_0$ lies in the interior of exactly $n-k$ semi-invariant domains $D(\alpha_1),\cdots,D(\alpha_{n-k})$.
\end{rem}

The following proposition will be used in the proof of the $c$-vector theorem.

\begin{prop}\label{prop 4.1.4: every D(beta) has T(beta)}
Let $T_0=T_1\oplus \cdots \oplus T_{n-2}$ be a partial cluster tilting object with $n-2$ summands. Let $M_\beta$ be an exceptional object of $|T_0|^\perp$. Then
\begin{enumerate}
\item If $M_\beta$ is a nonsimple exceptional object of $|T_0|^\perp$ there is, up to isomorphism, only one object $T(\beta)$ so that $T_0\oplus T(\beta)$ is a partial cluster tilting object and $\undim T(\beta)\in D(\beta)$.
\item If $M_\beta$ is a simple object in $|T_0|^\perp$ there are two nonisomorphic objects $T',T''$ in the cluster category $\cC_\Lambda$ of $mod\text-\Lambda$ so that $T_0\oplus T'$ and $T_0\oplus T''$ are partial {cluster tilting objects} and so that $T', T''$ lie in $^{\perp_V}\! M_\beta$.
\end{enumerate}
\end{prop}

\begin{proof}
In Case (1), by the lemma, the partial cluster tilting object $T_0$ lies on the boundary of the polyhedral region $D(\beta)$. Therefore, there is at most one way to complete it to a cluster tilting object in $D(\beta)$. Thus, it suffices to show the existence of a nonzero object $T(\beta)\in\,^{\perp_V}\!M_\beta$ so that $T_0\oplus T(\beta)$ is a partial cluster tilting object and, in Case (2), we need to show that there are two objects $T',T''$, where either $T',T''\in\,^{\perp_V}\! M_\beta$ or $T'\in\,^{\perp_V}\! M_\beta$ and $T''=P[1]$, where $P\in\,^\perp\! M_\beta$ is a projective $\Lambda$-module.

The existence of $T(\beta)$ is straightforward using basic properties of {cluster tilting objects}. Since $T_0$ is an almost complete cluster tilting object in the cluster category of $^{\perp_V}\! M_\beta$, there are two objects $T',T''$ in this cluster category which complete the cluster tilting object. At least one of them, say $T'$, is a module in $^{\perp_V}\! M_\beta$. Letting $T(\beta)=T'$, this proves Case (1).

Case (2) $M_\beta$ is simple in $|T_0|^\perp$. If both $T',T''$ are modules we are done. So, suppose that $T''=P[1]$ for some projective object $P\in\,^\perp\! M_\beta$. Then we claim that $P$ is projective in $mod\text-\Lambda$ making $T''=P[1]$ an object of the cluster category of $mod\text-\Lambda$.

Suppose that $P$ is the projective cover of the simple object $E_k\in\,^\perp\! M_\beta$. Then the dimension vectors of the objects $T_1,\cdots,T_{n-2}$ lie on the face of the positive simplex $\Delta^+(\beta)$ opposite the vertex $E_k$. By Lemma \ref{lemma when T0 lies in D(b)}(c), $D(\beta)$ contains a small neighborhood of the point $\undim T_0$ inside the hyperplane $H_\beta$. After rescaling, any such neighborhood contains an integer point having negative $E_k$-coordinate. By the Virtual Stability Theorem \ref{thm 3.1.1: virtual stability theorem} such a point has the form $\sum k_i \undim E_i+\sum \ell_j\undim P_j$, where $E_i$ are the simple objects of $^\perp\! M_\beta$ and $P_j$ are the projective objects of $^\perp\! M_\beta$ which are also projective in $mod\text-\Lambda$. By construction, at least one of these $P_j$ must have $E_k$ in its composition series. But then the projective cover $P$ of $E_k$ in $^\perp\! M_\beta$ is a submodule of $P_j$ which is projective in both $^\perp\! M_\beta$ and $mod\text-\Lambda$. So, $T''=P[1]$ lies in the cluster category of $mod\text-\Lambda$ as claimed.
\end{proof}

\begin{thm}\label{thm 4.1.5: clusters and SIs}
Let $T=T_1\oplus \cdots\oplus T_n$ be a cluster tilting object for $\Lambda$. Then:
\begin{enumerate}
\item[(a)] The dimension vectors $\undim T_i$ span a conical simplex in $\RR^n$ whose walls are $D(\beta_i)$ for uniquely determined real Schur roots $\beta_i$. 
\item[(b)] $\End_\Lambda(M_{\beta_i})\cong \End_{\cC_\Lambda}(T_i)$ for each i.
\item[(c)] The interior of the conical simplex spanned by $\{\undim T_i\}_{i=1}^n$ does not meet any $D(\beta)$. 
\item[(d)] The objects $T_i$ can be numbered in such a way that $\End_{\cC_\Lambda}(T_i)\cong\End_\Lambda(S_i)=F_i$ where $S_i$ are the simple $\Lambda$-modules.
\item[(e)] Furthermore, 
$
	\brk{\undim T_i,\beta_j}= \delta_{ij} \varepsilon_jf_j
$, where $f_j=\dim_KF_j=\dim_K\End_{\cC_\Lambda}(T_j)=f_{\beta_j}$ with the notation $f_\beta=\dim_K \End_\Lambda(M_{\beta})$ and $\varepsilon_j=\pm1$ is the sign of $\brk{\undim T_j,\beta_j}$.%, more precisely, $\varepsilon_j=f_j/\brk{\undim T_j,\beta_j}$.
\end{enumerate}
\end{thm}

\begin{proof}
(a) Each face of the conical simplex is spanned by $\undim T_i$ with one $T_j$ deleted. Then $|T/T_j|^\perp$ has a unique simple object, say, $M_{\beta_j}$ and $\undim T_i\in D(\beta_j)$ for $i\neq j$ by Remark \ref{rem 4.1.3}.

(b) By Schofield (\ref{lem: Schofield's observation}), the $T_i$ can be renumbered to form an exceptional sequence $(T_1,\cdots,T_n)$. For each $j$ we have another exceptional sequence $(M_{\beta_j},T_1,\cdots,\widehat{T_j},\cdots,T_n)$. By Proposition \ref{prop: properties of exceptional sequences} (5), this implies (b). By \ref{prop: properties of exceptional sequences} (3), this also implies that $\undim T_j=\pm \undim M_{\beta_j}$ plus a linear combination of $\undim T_i$ for $i\neq j$. We need this to prove (e).

(c) The interior of the conical simplex $\sigma$ cannot lie in any $D(\alpha)$. If it did, then $D_\ZZ(\alpha)$ would contain a integer point, say $v$, in the interior of $\sigma$. But then the general virtual representation $P\to Q$ with dimension vector $v$ would lie in $D_\ZZ(\alpha)$. However, by the virtual canonical decomposition theorem this representation is a direct sum of the representations $T_i$ and each $T_i$ occurs. So, $\undim T_i\in D_\ZZ(\alpha)$ for all $i$. This would make $D_\ZZ(\alpha)$ $n$ dimensional contradicting the fact that it has codimension one.

(d) follows from Corollary \ref{cor: End(Xi) is End(S_i)}.

(e) follows from (a) and (b): Since $T_i\in D(\beta_j)$ for $i\neq j$, we have $\brk{\undim T_i,\beta_j}=0$ for $i\neq j$. And $f_j=f_{\beta_j}$ by (b). %And $\brk{\undim T_j,\beta_j}\neq0$ since otherwise $\undim T_i$ would be $n$ linearly independent vectors in the same hyperplane $H(\beta_j)$. And $\brk{\undim T_j,\beta_j}\neq0$ is a multiple of $f_j$ by Proposition \ref{prop: pairing is divisible by fM}. It remains to show that this multiple is $\pm1$. 
In the proof of (b) we observed that $\undim T_j=\pm \undim M_{\beta_j}$ plus a linear combination of $\undim T_i$ for $i\neq j$. Since $\brk{\undim T_i,\beta_j}=0$ for all $i\neq j$, this implies
%This follows from the properties of exceptional sequences. Namely, by Proposition \ref{prop: properties of exceptional sequences} (3) and a little linear algebra we can conclude that $\undim T_k=\pm \undim M_{\beta_k}$ plus a linear combination of $\undim T_j$ for $j\neq k$. Since $\brk{\undim T_j,\beta_k}=0$ for all $j\neq k$, this implies that
\[
	\brk{\undim T_j,\beta_j}=\pm \brk{\beta_j,\beta_j}=\pm f_{\beta_j}=\pm f_j
\]
We denote the sign by $\varepsilon_j$. This completes the proof of (e).
\end{proof}

Using this theorem we can now define the matrices $\Gamma_T$ whose columns are equal, by definition, to det-weights of semi-invariants up to sign and will be shown to be equal to the $c$-vectors of the cluster tilting object $T$, up to sign, by Theorem \ref{c-vector thm} below.

\begin{defn}\label{def 4.1.6: Gamma-T}
For any cluster tilting object $T=\bigoplus_{i=1}^n T_i$ for $\Lambda$, let $\Gamma_T$ be the $n\times n$ integer matrix with columns $\gamma_i=\varepsilon_i\beta_i$ where $\beta_i$ are the unique real Schur roots so that $\undim T_i\in D(\beta_j)$ for $i\neq j$ and $\varepsilon_i=\pm1$ is the sign of $\brk{\undim T_i,\beta_i}$.
\end{defn}
%satisfying $\brk{\undim T_i,\beta_j}=0$ for $i\neq j$ and$f_i=\dim_K\End_{\cC_\Lambda}(T_i)=\brk{\beta_i,\beta_i}$

\begin{cor}\label{cor: HT equation}
Let $V$ be the $n\times n$ matrix with columns $\undim T_i$. Then
\[
	V^tE\Gamma_T=D
\]
where $E$ is the Euler matrix and $D$ is the diagonal matrix with diagonal entries $f_i$.
\end{cor}

One very important observation about the significance of the sign $\varepsilon_i$ is the following.

\begin{prop}
Suppose that $\varepsilon_k=\sgn\brk{\undim T_k,\beta_k}>0$. Then $T_k$ is a module and there does not exist any epimorphism $B\onto T_k$ where $B$ is a module in $add\,T/T_k$.
\end{prop}

\begin{proof}
If $T_k=P[1]$ then $\Ext^1_{Pres(\Lambda)}(P[1],X)=0$ for any $X\in Pres(\Lambda)$. So, $T_k$ is a module. If $\brk{\undim T_k,\beta_k}>0$ then $\Hom_\Lambda(T_k,M_\beta)\neq 0$. For any epimorphism $B\onto T_k$ we get $\Hom_\Lambda(B,M_\beta)\neq0$. This is impossible for $B\in add\,T/T_k$ since $T/T_k\in D(\beta)$.
\end{proof}

\begin{cor}
Suppose that $\varepsilon_k=+$ and let $T_k'$ be the mutation of $T_k$ so that $T_k'\oplus T/T_k$ is a cluster tilting object. Then $\Ext_{Pres(\Lambda)}^1(T_k',T_k)\neq0$ and $\Ext_{Pres(\Lambda)}^1(T_k,T_k')=0$.
\end{cor}

\begin{proof}
It is well-known \cite{BMRRT} that one of these two extension groups is zero and the other is nonzero. So, suppose that $\Ext_{Pres(\Lambda)}^1(T_k,T_k')\neq0$. If $T_k,T_k'$ are both modules, there would be a short exact sequence $0\to T_k'\to B\to T_k\to0$ where $B\in add\,T/T_k$. This is not possible by the Proposition. So, either $T_k$ or $T_k'$ is a shifted projective. But $\Ext_{Pres(\Lambda)}(T_k,P[1])=0$. So, we must have $T_k=P[1]$. Then $\brk{T_k,\beta_k}\le0$ contradicting the assumption that $\varepsilon_k>0$.
%The other possibility is excluded by the Proposition.
\end{proof}

\subsection{Definition of $c$-vectors}\label{ss: def of c-vectors}

An $n\times n$ integer matrix $B$ is called \emph{skew symmetrizable} if there is a diagonal matrix $D$ with positive integer diagonal entries so that $DB$ is skew-symmetric. $D$ is called the \emph{symmetrizer} of $B$. An \emph{extended exchange matrix} is defined to be a $2n\times n$ matrix $\tilde B=\mat{B\\C}$ whose top half $B$ is skew-symmetrizable.

\begin{defn}\cite{FZ} For any extended exchange matrix $\tilde B=(b_{ij})$ and any $1\le k\le n$, the \emph{mutation $\mu_k \tilde B$ of $\tilde B$ in the $k$-direction} is defined to be the matrix $\tilde B'=(b_{ij}')$ defined by
\begin{equation}\label{eq: mutation of B}
	b_{ij}'=\begin{cases}  -b_{ij}& \text{if } i=k \text{ or } j=k\\
   b_{ij}+b_{ik}|b_{kj}| & \text{if } b_{ik}b_{kj}>0\\
  b_{ij} & \text{otherwise}
    \end{cases}
\end{equation}
For any finite sequence of positive integers $k_1,k_2,\cdots,k_r\le n$ we have the \emph{iterated mutation} $\mu_{k_r}\cdots \mu_{k_1}\tilde B$ of $\tilde B$.
\end{defn}

\begin{defn}\cite{FZ}
Let $B_0$ be a fixed skew-symmetrizable matrix which we call the \emph{initial exchange matrix}. Then $\tilde B_0=\mat{B_0\\I_n}$ is called the \emph{initial extended exchange matrix}. Consider the set of all $n\times n$ matrices $C$ which appear at the bottom of matrices $\tilde B_C=\mat{B_C\\C}$ given by iterated mutation of the {initial extended exchange matrix}. The columns of all such matrices $C$ are called the \emph{$c$-vectors} of $B_0$. The matrices $C$ are called the \emph{$c$-matrices} of $B_0$.
\end{defn}

We recall that a vector $v$ is called \emph{sign coherent} if its nonzero coordinates have the same sign. We write $v>0$ if this sign is positive and $v\neq0$. We will use a theorem of Nakanishi and Zelevinsky which can be phrased as follows.

\begin{thm}\cite{NZ}\label{thm 4.2.3: NZ on candidate c-vectors} Let $B_0$ be a skew-symmetrizable matrix with symmetrizer $D$ and let $\cX$ be a set of $n\times n$ integer matrices $C$ with the following properties.
\begin{enumerate}
\item $I_n\in\cX$
\item For any $C\in \cX$, the columns of $C$ are sign coherent and nonzero.
\item For any $C\in \cX$, the matrix $
	B_C:=D^{-1}C^tDB_0C
$ has integer entries $b_{ij}$.
\item Let $C\in\cX$ and $1\le k\le n$, then $\cX$ contains the matrix $C'=\mu_kC$ with columns $c_j'$ given as follows, where $b_{kj}$ are entries of $B_C$.
\[%\begin{equation}\label{eq: mutation of C}
	c_j'=\begin{cases} -c_k & \text{if } j=k\\
	c_j +|b_{kj}|c_k & \text{if } b_{kj}c_k>0\\
   c_j  & \text{otherwise}
    \end{cases}
\]%\end{equation}
\item $\cX$ is minimal with the above properties.
\end{enumerate}
Then $\cX$ is the set of $c$-matrices of $B_0$ and the columns of $C\in\cX$ are the $c$-vectors of $B_0$.
\end{thm}

To specify the $c$-vectors corresponding to a cluster tilting object, we need to choose an initial cluster tilting object. Let $B_0=L^t-R$, where $L$, $R$ are the left and right Euler matrices (Section \ref{sec1.1}.). Then $B_0$ is an $n\times n$ skew-symmetrizable matrix since $DB_0=E^t-E$ is skew-symmetric, where $E$ is the Euler matrix. We will use $B_0$ as the {initial exchange matrix} and $\widetilde B_0=\mat{B_0\\ I_n}$ as the {initial extended exchange matrix}. The following easy observation will be useful.

\begin{lem}\label{why B0 is correct}
For any two vectors $x,y\in\RR^n$ we have
$
	\brk{y,x}-\brk{x,y}=x^t DB_0 y.
$
\end{lem}

\subsection{$c$-vector theorem}\label{ss: c-vector theorem}

%We can now state the main theorem of this section. %The definition of $c$-vectors is given in section \ref{ss: def of c-vectors} below.

\begin{thm}[c-vector theorem]\label{c-vector thm} Let $\Lambda$ be any finite dimensional hereditary algebra over any field. Let $\cC_\Lambda$ be the cluster category of $\Lambda$.
Let the initial cluster tilting object {in $\cC_\Lambda$} be $\Lambda[1]=\bigoplus_{i=1}^n P_i[1]$. Then the $c$-vectors associated to the cluster tilting object $T=\bigoplus_{i=1}^n T_i$ are $c_i=-\varepsilon_i\beta_i$ where $\beta_i$ are the associated det-weights and each $\varepsilon_i$ is the sign of $\brk{\undim T_i,\beta_i}$.
\end{thm}

The plan for the proof of this theorem is as follows. Let $\cX$ denote the set of matrices
\[
	\cX:=\{-\Gamma_T=-[\gamma_1,\cdots,\gamma_n]\,|\,\gamma_i=\varepsilon_i\beta_i\}.
\] 
where $\varepsilon_i$ is the sign of $\brk{\undim T_i,\beta_i}$.
We will show that $\cX$ satisfies the conditions of Theorem \ref{thm 4.2.3: NZ on candidate c-vectors} ((1), (2) are Lemma \ref{lem: conditions (1) and (2)}, (3) is Lemma \ref{lem: condition (3)}, (4) is Proposition \ref{prop: Gamma matrix transforms correctly} and (5) follows from the fact that mutation acts transitively on the set of cluster tilting objects \cite{Hubery}). Therefore, by Theorem \ref{thm 4.2.3: NZ on candidate c-vectors}, $\cX$ is equal to the set of all $c$-matrices of the initial exchange matrix $B_0=L^t-R$.

\begin{lem}\label{lem: conditions (1) and (2)}
\emph{(1)} $\cX$ contains the identity matrix $I_n$.

\emph{(2)} The columns of $\cX$ are sign coherent.
\end{lem}

\begin{proof}
(1) For $T=\Lambda[1]$, $\Gamma_{\Lambda[1]}=-I_n$. This follows from Definition \ref{def 4.1.6: Gamma-T} since $S_i\in P_j^\perp$ for $i\neq j$ which implies that $-\undim P_j\in D(e_i)$ where $e_i=\undim S_i$ is the $i$-th unit vector and $\brk{\undim P_i[1],e_i}=-1$. Therefore $I_n\in\cX$.

(2) Since the columns $\gamma_i$ of $\Gamma_T$ are, up to sign, dimension vectors of indecomposable modules $M_{\beta_i}$, they are sign coherent.
\end{proof}

%Corollary \ref{cor: HT equation} using the fact that the dimension vectors of the projective modules gives the columns of $L^{-1}=DE^{-1}$ where $L$ is the left Euler matrix and therefore the dimension vectors of the shifted projective objects gives the columns of $V=-DE^{-1}$. The unique solution of the equation $VE\Gamma_{\Lambda[1]}=D$ from Corollary \ref{cor: HT equation} is therefore $\Gamma_{\Lambda[1]}=-I_n$. 

\begin{lem}\label{lem: condition (3)}
Let $T=T_1\oplus\cdots\oplus T_n$ be a cluster tilting object and $\Gamma_T$ the associated matrix with columns $\gamma_i=\varepsilon_i\beta_i$. Then the matrix $B_\Gamma=B_{-\Gamma}=D^{-1}\Gamma^t DB_0\Gamma$ has integer entries. Hence $\cX$ satisfies Condition (3) in Theorem \ref{thm 4.2.3: NZ on candidate c-vectors}.
\end{lem}

\begin{proof}
By Lemma \ref{why B0 is correct}, the entries of $B_\Gamma$ are
$%begin{equation}\label{eq: entries of B}
	b_{ij}=f_i^{-1}(\left<\gamma_j,\gamma_i\right>-\left<\gamma_i,\gamma_j\right>)
$. %end{equation}
The columns $\gamma_i$ of $\Gamma_T$ are, up to sign, dimension vectors of exceptional modules $M_{\beta_i}$ and $\End_\Lambda(M_{\beta_i})\cong \End_{\cC_\Lambda}(T_i)\cong F_i$. Also, $f_i=\dim_KF_i$. Therefore, $b_{ij}$ are integers by Proposition \ref{prop: pairing is divisible by fM}. 
\end{proof}

We need to show that the set $\cX$ satisfies condition (4) in Theorem \ref{thm 4.2.3: NZ on candidate c-vectors} which is the following proposition whose proof will occupy the rest of this section.

\begin{prop}\label{prop: Gamma matrix transforms correctly}
Under the mutation $\mu_k$ of $T$, the matrix $\Gamma_T$ changes to $\Gamma_T'$ with columns $\gamma_j'$ given as follows where $b_{ij}$ are the entries of $B_\Gamma$.
\[
	\gamma_j'=\begin{cases}
	-\gamma_k & \text{ if } j=k\\
\gamma_j +|b_{kj}|\gamma_k & \text{ if } b_{kj}\gamma_k<0\\
	\gamma_j & \text{ otherwise}
	\end{cases}
\] 
\end{prop}
The inequality $b_{kj}\gamma_k<0$ is reversed from Theorem \ref{thm 4.2.3: NZ on candidate c-vectors}(4) since $\gamma_j,\gamma_k$ will turn out to be negative $c$-vectors. We will prove Proposition \ref{prop: Gamma matrix transforms correctly} first in the special case when $n=2$. We will then show that the general case follows from the special case. 

\subsection{Consecutive roots} In order to set up the reduction to the rank 2 case, we need to rephrase Proposition \ref{prop: Gamma matrix transforms correctly} in terms of the ``consecutive roots'' $-\gamma_j, \gamma_k,\gamma_j'$.

\begin{defn}\label{def: ordered pairs}
Let $T_0\in Pres(\Lambda)$ be a partial cluster tilting object with $n-2$ summands. Define $\cS_\Lambda(T_0)$ to be the set of all ordered pairs $(\gamma,U)$ where
\begin{enumerate}
\item $U$ is an exceptional object of $Pres(\Lambda)$ so that $T_0\oplus U$ is rigid.%a partial cluster tilting object.
\item $\gamma=\pm\beta$ where $M_\beta$ is the unique exceptional module in $|T_0\oplus U|^\perp$.
\end{enumerate}
\end{defn}

\begin{rem}\label{rem: U is unique}
Note that, by Proposition \ref{prop 4.1.4: every D(beta) has T(beta)}, $U$ is uniquely determined by $\gamma$ except when $M_{|\gamma|}$ is a simple object of $|T_0|^\perp$ in which case there are exactly two possibilities for $U$.
\end{rem}

\begin{prop}\label{prop: next function}
For each $(\gamma,U)\in\cS_\Lambda(T_0)$, there is a unique $(\gamma',U')\in\cS_\Lambda(T_0)$ so that
\begin{enumerate}
\item[$\cS1.$] $T_0\oplus U\oplus U'$ is a cluster tilting object in $Pres(\Lambda)$.
\item[$\cS2.$] $\brk{\undim U,\gamma'}>0$.
\item[$\cS3.$] $\brk{\undim U',\gamma}<0$.
\end{enumerate}
\end{prop}

\begin{proof}
There are two objects $U',U''$ so that $T_0\oplus U\oplus U'$, $T_0\oplus U\oplus U''$ are cluster tilting objects. These objects must lie on opposite sides of the hyperplane $H_\gamma=\{x\in\RR^n\,|\, \brk{x,\gamma}=0\}$. Therefore, up to reordering, we have: $
	\brk{\undim U',\gamma}<0,\quad \brk{\undim U'',\gamma}>0
$. So, $U'$ is uniquely determined by $\cS1$ and $\cS3$.

Let $M_{\beta'}$ be the unique exceptional object in $|T_0\oplus U'|^\perp$. Then $(\beta',U')$, $(-\beta',U')$ are the elements of $\cS_\Lambda(T_0)$ with second entry $U'$. Let $\gamma'=\sgn\brk{\undim U,\beta'}\beta'$. Then $(\gamma',U')$ is the unique pair satisfying $\cS1,\cS2,\cS3$.
\end{proof}

\begin{defn}
Let $\rho(\gamma,U)$ denote the unique pair $(\gamma',U')$ given by Proposition \ref{prop: next function}. A sequence of pairs $(\gamma_1,U_1), (\gamma_2,U_2),(\gamma_3,U_3),\cdots\in\cS_\Lambda(T_0)$ will be called \emph{consecutive pairs} if $\rho(\gamma_i,U_i)=(\gamma_{i+1},U_{i+1})$. And $\gamma_1,\gamma_2,\gamma_3,\cdots$ will be called \emph{consecutive roots} if there exist $\{U_i\}$ so that $\{(\gamma_i,U_i)\}$ is a sequence of consecutive pairs.
\end{defn}

\begin{cor}
$\rho:\cS_\Lambda(T_0)\to \cS_\Lambda(T_0)$ is a bijection.
\end{cor}

\begin{proof}
If $\rho(\gamma,U)=(\gamma',U')$ then one sees easily that $\rho(-\gamma',U')=(-\gamma,U)$. So, $s\circ \rho\circ s$ is the inverse of $\rho$ where $s(\gamma,U)=(-\gamma,U)$.
\end{proof}

\begin{lem}\label{lem 4.4.6}
Suppose that $T=T_0\oplus T_j\oplus T_k$ and $T'=\mu_kT=T_0\oplus T_j\oplus T_k'$. Let $\gamma_j=\varepsilon_j\beta_j$, $\gamma_k=\varepsilon_k\beta_k$ be the $j$th and $k$th $\gamma$-vectors of $T$. Let $\gamma_j'=\varepsilon_j'\beta_j',\gamma_k'=\varepsilon_k' \beta_k'$ be the corresponding $\gamma$ vectors of $T'$. Then $\gamma_k'=-\gamma_k$ and $
	(-\gamma_j,T_k),(\gamma_k,T_j),(\gamma_j',T_k')
$ are consecutive pairs in $\cS_\Lambda(T_0)$. In particular, $-\gamma_j,\gamma_k,\gamma_j'$ are consecutive roots.% and all triples of consecutive roots are given in this way.
\end{lem}

\begin{proof}
By definition of $\varepsilon_j,\varepsilon_k$ we have $\brk{\undim T_k,\gamma_k}>0$ and $\brk{\undim T_j,\gamma_j}>0$. Therefore, $\rho(-\gamma_j,T_k)=(\gamma_k,T_j)$. For $T'$, $\beta_k'$ is by definition the unique positive root so that $T_0\oplus T_j\in D(\beta_k')$. So, we must have $\beta_k'=\beta_k$. The signs $\varepsilon_k,\varepsilon_k'$ must be opposite since $T_k,T_k'$ must lie on opposite sides of the set $D(\beta_k)$. (If they were on the same side, the cones spanned by $T$ and $T'$ would overlap.) So, $\gamma_k'=-\gamma_k$.
Then, $\rho(\gamma_k,T_j)=\rho(-\gamma_k',T_j)=(\gamma_j',T_k')$.
\end{proof}

\begin{lem}\label{lem: reformulation in terms of consecutive roots}
Proposition \ref{prop: Gamma matrix transforms correctly} follows from the following equation for all triples of consecutive roots: $\gamma,\gamma',\gamma''$:
\[%\begin{equation}\label{eq: formula for gamma''}
	\gamma''=\begin{cases}
	%-\gamma_k & \text{ if } j=k\\
-\gamma +|b|\gamma' & \text{ if } b\gamma'<0\\
	-\gamma & \text{ otherwise}
	\end{cases}
\]%\end{equation}
where $b=f_{\gamma'}^{-1}(\brk{\gamma',\gamma}-\brk{\gamma,\gamma'})$ and $f_{\gamma'}=\dim_K\End_\Lambda(M_{|\gamma'|})$.
\end{lem}

\begin{proof} Suppose that the formula above for $\gamma''$ holds for all triples of consecutive roots. Then it holds in the particular case $\gamma=-\gamma_j$, $\gamma'=\gamma_k$ and $\gamma''=\gamma_j'$. Substituting these values of $\gamma,\gamma',\gamma''$ transforms the given equation into the formula for $\gamma_j'$ given in Proposition \ref{prop: Gamma matrix transforms correctly} except for the missing statement $\gamma_k'=-\gamma_k$ which was shown in Lemma \ref{lem 4.4.6} above. 
\end{proof}

\subsection{Proof of Proposition \ref{prop: Gamma matrix transforms correctly} in rank 2 case} The results of this section are well-known. We include them for clarity. Let $H$ be a finite dimensional hereditary algebra of rank 2. Then $H$ will be Morita equivalent to the tensor algebra of a modulated quiver. (See Appendix A.) So, we assume that
\[
	H=\mat{F_1 & 0 \\ M & F_2}
\]
where $F_1,F_2$ are division algebras over $K$ and $M$ is an $F_2\text-F_1$-bimodule. A (right) $H$-module can be viewed as a representation $V=(V_1,V_2,f:V_2\otimes_{F_2}M\to V_1)$ of the modulated quiver
\[
	F_1\xleftarrow{M}F_2.
\]
Recall that $\undim V=(\dim_{F_1}V_1,\dim_{F_2}V_2)$ and $f_i=\dim_KF_i$. Let $d_i=\dim_{F_i}M$. Then
\[
	\dim_\kk M=m=f_1d_1=f_2d_2.
\]
The projective $H$-modules are $P_1^H=(F_1,0,0)$ and $P_2^H=(M,F_2,id:F_2\otimes M\to M)$. The injective $H$-modules are given by a dual construction $I_2^H=(0,F_2,0)$, $I_1^H=(F_1,M^\ast,ev)$ where $M^\ast=\Hom_{F_1}(M,F_1)$ and $ev:M^\ast\otimes M\to F_1$ is the evaluation map. The simple $H$ modules are $P_1^H,I_2^H$. These have dimension vectors
\begin{equation}\label{eq: equation for dim Y2, dim Z2}
	\undim P_1^H=(1,0),\quad \undim P_2^H=(d_1,1),\quad 
	\undim I_1^H=(1,d_2), \quad \undim I_2^H=(0,1)
\end{equation}

It is well-known that $H$ has \emph{finite type}, i.e., has only finitely many exceptional representations up to isomorphism, if and only if $d_1d_2\le 3$. These are the quivers $A_1\times A_1,A_2,B_2,G_2$. We let $s$ denote the number of indecomposable modules. So, $s=2,3,4,6$ or $\infty$.

All exceptional $H$-modules are either preprojective or preinjective. We denote the preprojective modules $Y_i$ and the preinjective modules $Z_j$ keeping in mind that $Z_j=Y_{s-j+1}$ in the finite case. The preprojective and preinjective component(s) of the Auslander-Reiten quiver of $H$ is given by:

\begin{center}
\begin{tikzpicture}[scale=.8]
\begin{scope}
%\draw (6.7,3) node{Auslander-Reiten quiver (preprojective and preinjective terms) of the category $\cH$:};
\foreach \x in {1.5,4.5,7.75,10.75}
{\draw[thick,->] (\x cm,1.9cm)--+(.9,-.9) ;} % arrow, upper sequence to lower
\foreach \x/\xtext in {1.25/$Y_2=P_2^H$,4.25/$Y_4$,10.5/$Z_3$,13.7/$Z_1=I_2^H$}
\draw (\x cm,2.3cm) node{\xtext} ; % Upper sequence
\draw (6.725,1.5) node{$\cdots$};
\foreach \x/\xtext in {-.25/$Y_1=P_1^H$,2.75/$Y_3$,9.25/$Z_4$,12.1/$Z_2=I_1^H$, 14.4/.}
\draw (\x cm,.7cm) node{\xtext} ; % Lower sequence
\foreach \x in {0,3,9.5,12.5}
\draw[thick,->] (\x cm,1.1cm)--+(.9,.9) ; % arrow, lower sequence to upper
\end{scope}
\end{tikzpicture}
\end{center}
The arrows denote irreducible maps $Y_i\to Y_{i+1}$ and $Z_j\to Z_{j-1}$. The Auslander-Reiten translation $\tau_H$ acts by ``shifting two spaces to the left'' and we have $H$-almost split sequences (subscripts of $d$ should be taken modulo 2).
\[%\begin{equation}\label{eq: ass for H}
	Y_i\cof Y_{i+1}^{d_{i+1}}\onto Y_{i+2},\quad Z_{j+2}\cof Z_{j+1}^{d_j}\onto Z_j\quad (i,j\ge 1).
\]%\end{equation}
So, the dimension vectors of $Y_i, Z_j$ are given recursively using \eqref{eq: equation for dim Y2, dim Z2} by%, for $i\ge2$ (or, for $2\le i\le s$ if $\cH$ has finite type and $Y_s=Z_1$),
\begin{equation}\label{eq: recursion for dim Yi}
	\undim Y_i=d_{i-1}\undim Y_{i-1}-\undim Y_{i-2}\quad i\ge3
\end{equation}
%The vectors $\undim Z_j$ are given by $\undim Z_1=\alpha_2$, $\undim Z_2=\alpha_1+d_2\alpha_2$ and, for $j\ge2$ (or, for $2\le j\le s$ if $\cH$ has finite type),
\begin{equation}\label{eq: recursion for dim Zj}
	\undim Z_j=d_i\undim Z_{j-1}-\undim Z_{j-2}\quad j\ge3.
\end{equation}

The Auslander-Reiten quiver of the cluster category  $\cC_H$ of $H$ \cite{BMRRT} has two more exceptional objects, $Y_1[1]=P_1^H[1]$ and $Y_2[1]=P_2^H[1]$, which come between $Z_1$ and $Y_1$: %before the $P_i$ and after the $I_i$:
\begin{center}
\begin{tikzpicture}[scale=.8]
\begin{scope}
\foreach \x in {1.5,4.5,7.75,10.85}
{\draw[thick,->] (\x cm,1.95cm)--+(.9,-.9) ;} % arrow, upper sequence to lower
\foreach \x/\xtext in {1.25/$Z_3$,4.25/$Z_1$,7.25/$Y_2[1]$,10.5/$Y_2$}
\draw (\x cm,2.3cm) node{\xtext} ; % Upper sequence
\draw (14,1.5) node{$\cdots$};
\draw (-.5,1.5) node{$\cdots$};
\foreach \x/\xtext in {2.75/$Z_2$,5.75/$Y_1[1]$,9/$Y_1$,12/$Y_3$, 15/.}
\draw (\x cm,.7cm) node{\xtext} ; % Lower sequence
\foreach \x in {0,3,6,9.25,12.25}
\draw[thick,->] (\x cm,1.05cm)--+(.9,.9) ; % arrow, lower sequence to upper
\end{scope}
\end{tikzpicture}
\end{center}
This is called the \emph{transjective component} of the Auslander-Reiten quiver of $\cC_H$. 

\begin{prop} The {cluster tilting objects} of $\cC_H$ are sums of pairs of consecutive objects in the above quiver: $
	Z_{i+1}\oplus Z_i,\quad Z_1\oplus Y_1[1],\quad Y_1[1]\oplus Y_2[1],\quad  Y_2[1]\oplus Y_1,\quad Y_i\oplus Y_{i+1}$.
\end{prop}

\begin{lem}
For $1<i<s$, $^\perp Y_i=add\,Y_{i+1}$ and $^\perp Z_{i+1}=add\,Z_i$. For the simple objects $Y_1,Z_1$, $^\perp Y_1=add(Y_2\oplus Y_2[1])$ and $^\perp Z_1=add(Y_1\oplus Y_1[1])$.
\end{lem}

\begin{prop}\label{prop: action of rho on SH(0)}
\emph{(a)} The elements of $\cS_H(0)$ are $(\pm\undim Z_{j+1},Z_{j})$, $(\pm\undim Y_i,Y_{i+1})$ for $i,j\ge1$ and the 6 pairs $(\pm\undim Y_1, Y_2[1])$, $(\pm\undim Z_1, Y_1)$ and $ (\pm\undim Z_1, Y_1[1])$. 

\emph{(b)} The action of $\rho$ on these pairs is given by the following list and by the reverse of the list (given by changing the sign of the first entries and reversing the order). %We use the notation $\undim Y_i=\alpha_i,\undim Z_j=\beta_j$.% In particular, there are exactly two orbits of the action of $\ZZ$ on $\cS_\cH(0)$ given by $\rho$.
\begin{enumerate}
\item $\rho(\undim Z_{j+2},Z_{j+1})=(\undim Z_{j+1},Z_{j})$ for $j=1,\cdots,s-2$.
\item $\rho(\undim Z_2,Z_1)=(\undim Z_1,Y_1[1])$
\item $\rho(\undim Z_1,Y_1[1])=(-\undim Y_1,Y_2[1])$
\item $\rho(-\undim Y_1,Y_2[1])=(-\undim Z_1,Y_1)$
\item $\rho(-\undim Z_1,Y_1)=(\undim Y_1,Y_2)$
\item $\rho(\undim Y_i,Y_{i+1})=(\undim Y_{i+1},Y_{i+2})$ for $i=1,\cdots,s-2$.
\end{enumerate}
\end{prop}

\begin{proof} (a) By Remark \ref{rem: U is unique}, $\gamma$ uniquely determines $U$ in the pair $(\gamma,U)$ when $M_{|\gamma|}$ is not a simple object. By definition $U\in \,^\perp M_{|\gamma|}$. So, $U$ must be the object after $M_{|\gamma|}$ in the Auslander-Reiten quiver of $H$. This gives the pairs $(\pm\undim Z_j,Z_{j-1})$, $(\pm\undim Y_i,Y_{i+1})$. When $\gamma$ is a simple root, $U=P$ or $P[1]$ where $P$ is the projective which does not map to $M_{|\gamma|}$. This gives $(\pm\undim Y_1,Y_2)$ and the remaining 6 pairs.

(b) The computation of $\rho(\gamma,U)$ in (1),(2),(3) and (6) are examples of the general formula: $\rho(\undim \tau_H X,W)=(\undim W,X)$ which holds in rank 2. (4) and (5) follow from (3) by change of sign.
\end{proof}

%The subscript of $B$ should be taken modulo 2 where $B_0$ and $B_1$ are as given in the table. Note that $B_1=-B_0^t$. We will use this table to prove Proposition \ref{prop: Gamma matrix transforms correctly} for $\cH$.

\begin{proof}[Proof of Proposition \ref{prop: Gamma matrix transforms correctly} in rank 2]

By Proposition \ref{prop: action of rho on SH(0)}, there are two sequences of consecutive roots in $\cS_H(0)$:
\[
	\cdots, \undim Z_3,\undim Z_2,\undim Z_1,-\undim Y_1,-\undim Z_1,\undim Y_1,\undim Y_2,\undim Y_3,\cdots
\]
\[
	\cdots, -\undim Y_3,-\undim Y_2,-\undim Y_1,\undim Z_1,\undim Y_1,-\undim Z_1,-\undim Z_2,-\undim Z_3,\cdots
\]
We consider only the first. The second is similar. The calculations are summarized in the following chart.
\[
\begin{array}{ccccccccc}
\gamma& \gamma'& \gamma'' & f_{\gamma'} & \brk{\gamma',\gamma} & \brk{\gamma,\gamma'} & b & \sgn(b\gamma') & \text{formula for $\gamma''$}\\
\hline
\undim Z_3 & \undim Z_2 & \undim Z_1 & f_1 & 0 & m & -d_1 & - & d_1\undim Z_2-\undim Z_3\\
\undim Z_2 & \undim Z_1 & -\undim Y_1 & f_2 & 0 & m & -d_2 & - & d_2\undim Z_1-\undim Z_2\\
\undim Z_1 & -\undim Y_1 & -\undim Z_1 & f_1 & 0 & m & -d_1 & +& -\undim Z_1\\
-\undim Y_1 & -\undim Z_1 & \undim Y_1 & f_2 & 0 & m & -d_2 & +& \undim Y_1\\
-\undim Z_1 & \undim Y_1 & \undim Y_2 & f_1 & -m & 0 & -d_1 & - & d_1\undim Y_1+\undim Z_1\\
\undim Y_1 & \undim Y_2 & \undim Y_3 & f_2 & 0 & m & -d_2 & -& d_2\undim Y_2-\undim Y_1\\
\end{array}
\]
That $\gamma''$ agrees with the formula from Lemma \ref{lem: reformulation in terms of consecutive roots} follows from the formulas \eqref{eq: equation for dim Y2, dim Z2}, \eqref{eq: recursion for dim Yi}, \eqref{eq: recursion for dim Zj}. So, Proposition \ref{prop: Gamma matrix transforms correctly} holds in the case $n=2$ by Lemma \ref{lem: reformulation in terms of consecutive roots}.
\end{proof}

\subsection{Proof of Proposition \ref{prop: Gamma matrix transforms correctly} in general case} Let $T_0$ be a partial cluster tilting object for $\Lambda$ with $n-2$ summands. Let $L_1,L_2$ be the indecomposable injective objects of the rank 2 hereditary abelian subcategory $|T_0|^\perp$ of $mod\text-\Lambda$. Let $H=\End_\Lambda(L)\op$ where $L=L_1\oplus L_2$.

\begin{prop}
The functor $F=\Hom_\Lambda(\cdot,L):mod\text-\Lambda\to mod\text-H$ induces an isomorphism of categories $|T_0|^\perp\cong mod\text-H$.
\end{prop}

The functor $F$ does not have the properties that we need on objects outside the subcategory $|T_0|^\perp$. So, we will replace it with a mapping $\eta$ which, unfortunately, is defined only on objects. We now set up the notation for this mapping.

Let $\Psi_+$ be the set of all dimension vectors of exceptional objects in $|T_0|^\perp$, let $\Psi=\Psi_+\cup -\Psi_+$ and let $V\subset \RR^n$ be the two dimensional subspace spanned by $\Psi$. Then every vector $v\in V$ is given uniquely as $v=x_1\alpha_1+x_2\alpha_2$ where $x_1,x_2\in\RR$ and $\alpha_1,\alpha_2\in \Psi_+$ are the dimension vectors of the simple objects $M_{\alpha_1},M_{\alpha_2}$ of $|T_0|^\perp$ corresponding to $L_1,L_2$. 

\begin{rem}  {(a)} The Euler-Ringel pairing on $\RR^2$ is given by:
\[
	\brk{(x_1,x_2),(y_1,y_2)}_H=\brk{x_1\alpha_1+x_2\alpha_2,y_1\alpha_1+y_2\alpha_2}
\]
{(b)} The linear isomorphism $\pi:V\to \RR^2$ given by $\pi(x_1\alpha_1+x_2\alpha_2)=(x_1,x_2)$ is also an isometry, i.e., $\brk{v,w}=\brk{\pi(v),\pi(w)}_H$ for all $v,w\in V$.\\
{(c)} Furthermore, $\pi(\Psi_+)$ is the set of all dimension vectors of exceptional $H$-modules.\\
(d) Let $\lambda_i\!=\!\undim L_i$ and $f_j\!=\!\dim_K\End_\Lambda(M_{\alpha_j})\!=\!\dim_K\Hom_\Lambda(M_{\alpha_j},L_j)$. Then $
	\brk{\alpha_i,\lambda_j}\!=\!f_j\delta_{ij}$.
\end{rem}

Let $\theta:\RR^n\to V$ be the linear mapping given by
\[
	\theta(x)=\sum_{i=1,2} f_i^{-1}\brk{x,\lambda_i}\alpha_i\,.
\]
\begin{lem}\label{lem 4.6.3}
\emph{(a)} $\theta$ is a projection, i.e., $\theta(v)=v$ for all $v\in V$.\\
\emph{(b)} $\brk{\theta(x),\beta}=\brk{x,\beta}$ for all $x\in\RR^n$, $\beta\in\Psi$.\\
\emph{(c)} $\brk{\pi\theta(x),\pi(\beta)}_H=\brk{x,\beta}$ for all $x\in\RR^n$, $\beta\in\Psi$.\\
\emph{(d)} $\theta(\undim T_0)=0$.
%\emph{(c)} $\pi$ induces an isomorphism $\pi':V\cong \RR^2$ which is an isometry with respect to the Euler-Ringel pairing, i.e., $\brk{v,w}=\brk{\pi(v),\pi(w)}_H$ for all $v,w\in V$.
\end{lem}

\begin{proof}
(a) follows from the observation that $\theta(\alpha_i)=\alpha_i$. (b) follows from the calculation: $\brk{\theta(x),\lambda_j}=\brk{x,\lambda_j}$ and the fact that $\lambda_1,\lambda_2$ span $V$. (c) follows from (b) and the fact that $\pi:V\to\RR^2$ is an isometry. (d) follows from the fact that $L_i\in |T_0|^\perp$.
\end{proof}

Suppose that $\overline\beta\in \pi\Psi_+\subset\ZZ^2$. Then the hyperplane $\{\overline x\in\RR^2\,|\, \brk{\overline x,\overline\beta}_H=0\}$ is a line through the origin in $\RR^2$. And $D_H(\overline\beta)$ is a closed subset of this line given by:
\[
	D_H(\overline\beta)=\{x\in \RR^2\,|\, \brk{x,\overline\beta}_H=0\text{ and }\brk{x,\overline\beta'}_H\le0 \text{ for all }\overline\beta'\subset\overline\beta\}.
\]

\begin{lem}\label{lem: D(b) goes into D(b)}
\emph{(a)} For any $\beta\in\Psi_+$, $\pi\theta(D_\Lambda(\beta))\subseteq D_H(\pi(\beta))$. \\
\emph{(b)} If $\pi\theta(x)\in D_H(\pi(\beta))$, $x\in \RR^n$, then $
	\undim T_0+\delta x\in D_\Lambda(\beta)
$ for sufficiently small $\delta>0$.
\end{lem}

\begin{proof} (a) Suppose $x\in D_\Lambda(\beta)$. Then $\brk{\pi\theta(x),\pi(\beta)}_H=\brk{x,\beta}=0$ and $\brk{\pi\theta(x),\pi(\beta')}_H=\brk{x,\beta'}\le0$ for all $\beta'\subset \beta$ in $\Psi_+$. So, $\pi\theta(x)\in D_H(\pi(\beta))$.

(b) Suppose that $\pi\theta(x)\in D_H(\pi(\beta))$. Then, for any $\delta>0$ we have:
\[
	\brk{\undim T_0+\delta x,\beta}=\brk{\pi\theta(\undim T_0+\delta x),\pi(\beta)}_H=\delta \brk{\pi\theta(x),\pi(\beta)}=0
\]
For any subroot $\beta'\subseteq \beta$ with $\beta'\in\Psi_+$ we have:
\[
	\brk{\undim T_0+\delta x,\beta'}=\brk{\pi\theta(\undim T_0+\delta x),\pi(\beta')}_H=\delta \brk{\pi\theta(x),\pi(\beta')}\le0
\]
For any subroot $\beta''\subseteq \beta$ with $\beta''\notin\Psi_+$ we have $\brk{\undim T_0,\beta''}<0$. (If $\brk{\undim T_0,\beta''}=0$ then $\undim T_0\in D(\beta'')$ which implies that $\beta''\in\Psi_+$ by Lemma \ref{lemma when T0 lies in D(b)}(b).) So, $\brk{\undim T_0,\beta''}\le-1$ and
\[
	\brk{\undim T_0+\delta x,\beta''}=\brk{\undim T_0,\beta''}+\delta \brk{x,\beta''}<\tfrac12%1/2
\]
for sufficiently small $\delta$. Therefore, $\undim T_0+\delta x\in D(\beta)$ for all sufficiently small $\delta>0$.
\end{proof}

We need the following criterion equivalent to $\cS1$ from Proposition \ref{prop: next function}.

\begin{lem}\label{lem: S1'}
Suppose that $U,U'$ are exceptional objects of $Pres(\Lambda)$ so that $T_0\oplus U,T_0\oplus U'$ are rigid. Then $T_0\oplus U\oplus U'$ is a cluster tilting objects of $\Lambda$ if and only if $\cS1'$ holds:% the following condition holds.

$\cS1'$: $\forall a,b,c\in\RR_{>0},\forall\beta\in\Phi_+(\Lambda)$,
$
	a\undim T_0+b\undim U+c\undim U'\notin D_\Lambda(\beta).
$
\end{lem}

\begin{proof}
The necessity of $\cS1'$ was shown in Corollary \ref{cor 2.4.5}.

To show sufficiency, suppose $T_0\oplus U\oplus U'$ is not a cluster tilting object. Let $V,V'$ be the two objects making $T=T_0\oplus U\oplus V$ and $T'=T_0\oplus U\oplus V'$ into cluster tilting objects in $Pres(\Lambda)$. Then $\undim(T_0\oplus U)$ lies in the interior of $C(T)\cup C(T')$ where $C(T)$ is the conical simplex spanned by $\{\undim T_i\}$, the components of $T$. Since $U'\neq V,V'$, by the virtual generic decomposition theorem \ref{thm 2.3.11: virtual generic decomposition theorem} it follows that $U'\notin C(T)\cup C(T')$. So, the straight line from $\undim(T_0\oplus U)$ to $\undim U'$ goes through the boundary of $C(T)\cup C(T')$. But, 
\[
	\partial (C(T)\cup C(T'))\subset \partial C(T)\cup \partial C(T')
\]
is a union of domains $D(\alpha)$. So, there exist $a,b>0$ so that
\[
	a(\undim T_0+\undim U)+b\undim U'\in D(\alpha)
\]
for some $\alpha$ contradicting $\cS1'$.
\end{proof}

\begin{prop}\label{prop: pairs to pairs}
Let $(\gamma,U)\in \cS_\Lambda(T_0)$. Then
\begin{enumerate}
\item[(a)] There is a unique object $\eta(U)\in Pres(H)$ so that $\undim \eta(U)=\pi\theta(\undim U)$.
\item[(b)] $(\pi(\gamma),\eta(U))\in \cS_H(0)$.
\item[(c)] If $\rho(\gamma,U)=(\gamma',U')$ then
$
	\rho_H(\pi(\gamma),\eta(U))=(\pi(\gamma'),\eta(U')).
$
\end{enumerate}
\end{prop}

\begin{proof}
(a) Let $|\gamma|=\beta\in\Psi_+$. Since $\undim U\in D_\Lambda(\beta)$, we have, by Lemma \ref{lem: D(b) goes into D(b)}, that $\pi\theta(\undim U)\in D_H(\pi(\beta))$. Let $M_{\overline\alpha}$ be the object which comes right after $M_{\pi(\beta)}$ in the Auslander-Reiten quiver of $H$ except in the case when $M_{\pi(\beta)}$ is the simple injective object in which case we let $M_{\overline\alpha}$ be the simple projective object. Then $\overline\alpha$ is the unique positive root of $H$ in $D_H(\pi(\beta))$. Since $\overline\alpha,\pi\theta(\undim U)\in\ZZ^2$ are collinear and the coordinates of $\overline\alpha$ are relatively prime, $\pi\theta(\undim U)$ must be an integer multiple of $\overline\alpha$. But $\rho(\gamma,U)=(\gamma',U')$ implies that $\brk{\undim U,\gamma'}>0$. And $\brk{\undim U,\gamma'}=\pm f_{\gamma'}$ by Theorem \ref{thm 4.1.5: clusters and SIs}(e). So, by Lemma \ref{lem 4.6.3}(c) this implies
\[
	\brk{\pi\theta(\undim U),\pi(\gamma')}_H=\brk{\undim U,\gamma'}=f_{\gamma'}.
\]
But $\brk{\overline\alpha,\pi(\gamma')}_H$ is also an integer multiple of $f_{\gamma'}$. So, $\pi\theta(\undim U)=\pm\overline\alpha$. If $\pi(\undim U)=-\overline\alpha\in D_H(\pi(\beta))$ then $M_{\overline\alpha}$ must be projective in $mod\text-H$ and $\pi\theta(\undim U)=\undim M_{\overline\alpha}[1]$. Otherwise, $\pi\theta(\undim U)=\undim M_{\overline\alpha}$. So, either $\eta(U)=M_{\overline\alpha}[1]$ or $\eta(U)=M_{\overline\alpha}$ is the unique object in $Pres(H)$ with $\undim \eta(U)=\pi\theta(\undim U)\in D_H(\pi(\beta))$.

(b) Since $\undim \eta(U)=\pi\theta(\undim U)\in D_H(\pi(\beta))$, $(\pi(\beta),\eta(U))\in\cS_H(0)$ which implies $(\pi(\gamma),\eta(U))\in\cS_H(0)$ since $\gamma=\pm \beta$ and, therefore, $\pi(\gamma)=\pm \pi(\beta)$.

(c) To show that $\rho_H(\pi(\gamma),\eta(U))=(\pi(\gamma'),\eta(U'))$ we will verify $\cS2,\cS3,\cS1'$ from Proposition \ref{prop: next function} and Lemma \ref{lem: S1'}.

$\cS2$. $\brk{\undim \eta(U),\pi(\gamma')}_H=\brk{\pi\theta(\undim U),\pi(\gamma')}_H=\brk{\undim U,\gamma'}>0$.

$\cS3$. $\brk{\undim \eta(U'),\pi(\gamma)}_H=\brk{\pi\theta(\undim U'),\pi(\gamma)}_H=\brk{\undim U',\gamma}<0$.

$\cS1'$. Suppose not. Then there exist $a,b>0$ and $\alpha\in \Psi_+$ so that $a\undim \eta(U)+b\undim\eta(U')\in D_H(\pi(\alpha))$. By Lemma \ref{lem: D(b) goes into D(b)}, this implies, for sufficiently small $\delta>0$, that
\[
	\undim T_0+\delta (a\undim U+b\undim U')\in D_\Lambda(\alpha).
\]
By Corollary \ref{cor 2.4.5} this is not possible since $T_0\oplus U\oplus U'$ is a cluster tilting object.
\end{proof}

\begin{rem}
The mapping $\eta$ is not a functor. However, it has very nice properties. The mapping $\eta$ gives a bijection between the set of all objects $U$ which are $\Ext$-orthogonal to $T_0$ and the set of all exceptional objects in $Pres(H)$. The fact that $\eta$ is surjection onto this set follows from Proposition \ref{prop: pairs to pairs}(c) and the fact that $Pres(H)$ has only two orbits of the action of $\rho$. To show that $\eta$ is 1-1, suppose $\eta(U)=\eta(U')$. Then $\undim \eta(U)=\undim\eta(U')$ lie in the same $D_H(\pi(\beta))$. This implies $\undim U,\undim U'$ lie in $D(\beta)$ which implies $\beta$ is simple and $\undim \eta(U)=-\undim\eta(U')$, a contradiction. For more details, see \cite{IT13}
\end{rem}

\begin{proof}[Proof of Proposition \ref{prop: Gamma matrix transforms correctly} in general] 
Suppose that $(\gamma,U),(\gamma',U'),(\gamma'',U'')$ are consecutive pairs in $\cS_\Lambda(T_0)$. Then, by Proposition \ref{prop: pairs to pairs}, $\pi(\gamma),\pi(\gamma'),\pi(\gamma'')$ are consecutive roots for the rank 2 hereditary algebra $H$. Therefore, by the calculation in the last subsection, the formula in Lemma \ref{lem: reformulation in terms of consecutive roots} holds for $\pi(\gamma),\pi(\gamma'),\pi(\gamma'')$.

But $\pi$ is an isometry. So, the formula also holds for $\gamma,\gamma',\gamma''$. By Lemma \ref{lem: reformulation in terms of consecutive roots} this implies Proposition \ref{prop: Gamma matrix transforms correctly}. 
\end{proof}

We can now prove the $c$-vector theorem.

\begin{proof}[Proof of Theorem \ref{c-vector thm}]
The statement is that the $c$-vectors of a cluster tilting object $T$ are $-\gamma_i$. This holds for the initial cluster tilting object $\Lambda[1]$ by definition. It is well-known that cluster mutation acts transitively on the set of {cluster tilting objects}. (See \cite{Hubery}.) Therefore, it suffices to show that the equation $c_i=-\gamma_i$ remains true under mutation. But this is what was shown in Proposition \ref{prop: Gamma matrix transforms correctly} with the aid of Theorem \ref{thm 4.2.3: NZ on candidate c-vectors}.
\end{proof}

\subsection{Example} Figure \ref{Fig C3} illustrates several concepts discussed in the paper. Take the modulated quiver
\[
	F_1=\CC\xlarrow{M_{21}=\CC}F_2= \RR\xlarrow{M_{32}=\RR}F_3= \RR
\] 
The tensor algebra of this quiver is of finite type with 9 indecomposable objects:
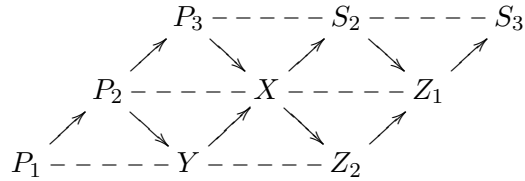
\[\xymatrixrowsep{12pt}\xymatrixcolsep{12pt}
\xymatrix{%begin xy matrix
&&P_3\ar[dr]\ar@{--}[rr] &&S_2\ar[dr]\ar@{--}[rr] &&S_3\\
&P_2\ar[ur]\ar[dr]\ar@{--}[rr] &&X\ar[dr]\ar[ur]\ar@{--}[rr] && Z_1\ar[ur]\\
P_1\ar[ur]\ar@{--}[rr] &&Y\ar[ur]\ar@{--}[rr] && Z_2\ar[ur]
	}%end xy matrix
\]
Consider $L$, the intersection with the unit sphere $S^2\subseteq \RR^3$ with the union $\bigcup D(\beta)$ of all nine semi-invariant domains. Figure \ref{Fig C3} shows the stereographic projection of $L$ onto the plane.

\begin{figure}[htbp]  %{Fig C3}
\begin{center}
\begin{tikzpicture}[scale=.9]
\begin{scope}[xscale=-1]%[scale=.75]
\begin{scope}[xscale=-1]
	\draw[fill] (-.75,-1.3) circle [radius=2pt];
	\draw[fill] (-2.22,1.72) circle [radius=2pt];
	\draw[fill] (.65,2.05) circle [radius=2pt];
	\draw[fill] (-.75,3.9) circle [radius=2pt];
\draw (2.7,2.2) node{$e_1$};
\draw (2.1,1.6) node{$e_2$};
\draw (1.55,1) node{$e_3$};
\draw (-2.4,-3.5) node{\tiny$D[100]$} % D(e3)
(3.5,3.4) node{\tiny$D[001]$} % D(e1)
(-5,3.4) node{\tiny$D[010]$}; % D(e2)
		\draw[thick] (-2.25,1.3) circle [radius=3cm]; % circle 1
		\draw[thick] (.75,1.3) circle [radius=3cm]; % circle 2
		\draw[very thick] (-.75,-1.3) circle [radius=2.3cm]; % circle 3
	\begin{scope}
		\clip (-.75,-3) rectangle (4.25,4.3);
%		\draw[thick, color=blue] (-.75,-3) rectangle (4.25,4.3);
		\draw[thick] (-.75,1.3) ellipse [x radius=3.4cm,y radius=2.6cm]; % 21 semicircle
	\end{scope}
	\begin{scope} % semicircles inside 1 and outside 3
	\clip (-2.25,1.3) circle [radius=3cm];
		\draw[very thick] (-1.74,.4) circle [radius=2.38cm]; % circle 3
		\draw[thick] (-1.35,-.25) circle [radius=2.15cm]; % circle 3
	\end{scope}
\begin{scope}
	\draw[very thick] (-1.85,2.78) .. controls (.5,4) and (3,2) .. (1.47,-0.68);
	\draw[thick] (-2.22,1.73) .. controls (0,3.3) and (1.8,2) .. (1.47,-0.68);
\end{scope}
	\draw[thick] (-1.85,2.78) .. controls (-1,3) and (0,2.8) .. (.65,2.05);
\draw (1.8,-.8) node{$S_3$};
\draw (-2.05,3.1) node{$S_2$};
\draw (-.7,-1.7) node{$S_1$};
\draw (-.7,-1.8) node[below]{$=P_1$};
\end{scope}
\draw (-.25,.4) node{$P_3$};
\draw (1.8,.4) node{$P_2$};
\draw (-.3,3.3) node{$Z_1$};
\draw (-.9,2.3) node{$Z_2$};
\draw (2.5,1.9) node{$Y$};
\draw (0.5,2.1) node{$X$};
\draw (0.7,4.3) node{$P_1[1]$};
\draw (3,-1.8) node[left]{$P_3[1]$};
\draw (-2,-1.8) node{$P_2[1]$};
\draw (3.7,2.5) node{\tiny$D[120]$};
\begin{scope}[xshift=-9.9cm]
\draw (5,-2.5) node{$Q=1\xleftarrow{(2,1)} 2\leftarrow 3$}
(5.1,-3.3) node{$=\mathbb C \leftarrow \mathbb R\leftarrow \mathbb R$};
\end{scope}
\end{scope}
\end{tikzpicture}
\caption{The three circles are domains of semi-invariants with simple {det-weights}. Other {det-weight}s are dimension vectors of other representations. For example, edges $e_1,e_2,e_3$ are domains of $(0,1,1), (1,2,2), (1,1,1)$. The four dark vertices $S_1,Z_2,Y,P_1[1]$ indicate the objects with endomorphism ring $\CC$. The semi-invariant domains $D(1,0,0), D(1,2,0)$ and $e_2=D(1,2,2)$ which correspond to $S_1,Y,Z_2$ by Proposition \ref{prop: vertices and opposite walls} are also darkened.}
\label{Fig C3}
\end{center}
\end{figure}
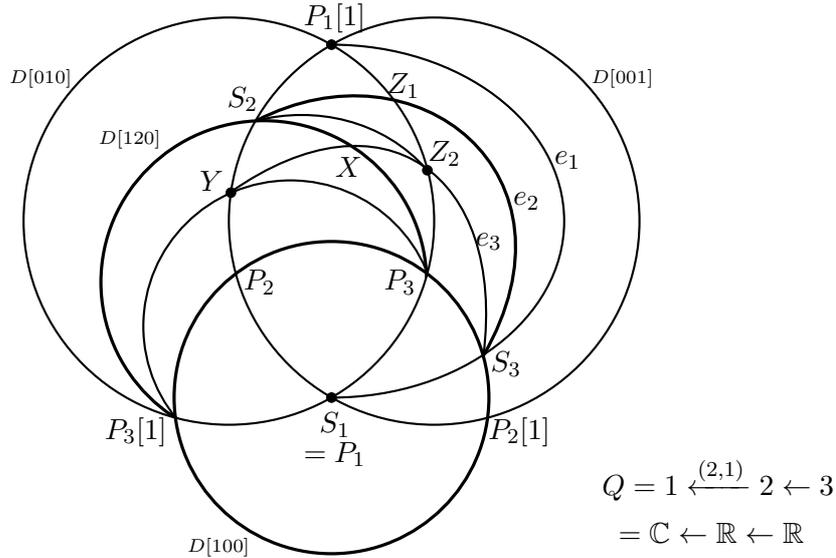
In reading Figure \ref{Fig C3} the following easy observation is helpful.

\begin{prop}\label{prop: vertices and opposite walls}
Let $T=\bigoplus T_i$ be a cluster tilting object with associated matrix $\Gamma_T=(\gamma_i)$. Suppose that $\gamma_k=\beta_k$ is positive and all other columns of $\Gamma_T$ are negative. Then $\undim T_k=\beta_k$. In other words, when the following triangle appears in a picture, $T_2=M_{\beta_2}$.

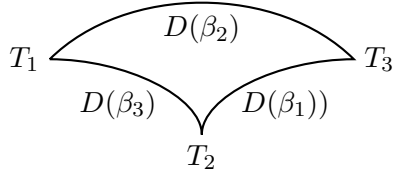
\begin{center}
\begin{tikzpicture}
\draw[thick] (-2,1) .. controls (-1,1) and (0,.5) .. (0,0)
.. controls (0,.5) and (1,1)..(2,1) .. controls (1,2) and (-1,2) .. (-2,1);
\draw (-2,1)node[left]{$T_1$}
(0,0)node[below]{$T_2$}
(2,1)node[right]{$T_3$};
\draw (1.1,0.4) node{$D(\beta_1))$}
%(0,2.1) node{$D(\beta_2)$}
(0,1.4) node{$D(\beta_2)$}
(-1.1,0.4) node{$D(\beta_3)$};
\end{tikzpicture}
\end{center}
\end{prop} 

\begin{proof}
As we say in the proof of the $c$-vector theorem, rank 2 case, for any $j\neq k$, the modules $M_{\gamma_j}$ and $M_{\gamma_k}$ are consecutive objects in the Auslander-Reiten quiver of the rank 2 perpendicular category $|T_0|^\perp$ where $T_0=\bigoplus_{i\neq j,k}T_i$. Therefore, $\brk{\gamma_k,\gamma_j}=0$ for all $j\neq k$. We also have $\brk{\undim T_k,\gamma_j}=0$ for all $j\neq k$. Since $\Gamma_T$ is an invertible matrix (by Corollary \ref{cor: HT equation}), this implies that $\undim T_k$ is a scalar multiple of $\gamma_k$. So, $\undim T_k=\beta_k$.
\end{proof}

\begin{eg} Examples of Proposition \ref{prop: vertices and opposite walls} in Figure \ref{Fig C3}.
\begin{enumerate}
\item $\undim Z_1=(0,1,1)$ and $e_1=D(0,1,1)$
\item $\undim Z_2=(1,2,2)$ and $e_2=D(1,2,2)$ which extends from $S_3$ to $S_2$
\item $\undim P_3=(1,1,1)$ and $e_3=D(1,1,1)$ which extends from $S_3$ throught $Z_2,X$ to $Y$.
\item $\undim X=(1,2,1)$ and $D(1,2,1)$ is the edge connecting $Z_2$ and $S_2$.
\end{enumerate}
\end{eg}

Figure \ref{Fig C3} also illustrates the following concepts used in the paper. For $\beta=(1,1,1)$, the simple objects of the category $^\perp\! M_\beta$ are $S_3$ and $Y$ with dimension vectors $\alpha_1=(0,0,1)$ and $\alpha_2=(0,2,1)$. These form the corners (endpoints in this dimension) of the convex region $D(\beta)$. The other roots in this region are positive integer linear combinations: $\undim Z_2=2\alpha_1+\alpha_2$ and $\undim X=\alpha_1+\alpha_2$.

\subsection{Applications}

In concurrently written papers we use the results of this paper to:
\begin{enumerate}
\item Develop the theory of signed exceptional sequences and show they are in bijection with ordered cluster tilting objects \cite{IT13}. We have seen a special case: $\cS_\Lambda(T_0)$ is the set of all signed exceptional sequences for $|T_0|^\perp$.
\item Develop the theory of semi-invariant picture groups and compute their cohomology in type $A_n$ \cite{IOTW4}.
\item Show that, for acyclic modulated quivers of finite type, the maximal green sequences are in bijection with the positive expressions for the Coxeter element in the picture group \cite{IT14}.
\item For any acyclic modulated quiver with a bimodule $M_{ij}:i\to j$ of infinite type, show that any maximal green sequence mutates at $j$ before $i$ \cite{BHIT}.
\end{enumerate}

Finally, we point out that Theorem \ref{c-vector thm} implies the sign coherence of $c$-vectors (that in each $c$-vector the coordinates have the same sign) a theorem which has been proven many times and in fact the present version of this paper grew out of a desire to understand the proof given by Speyer-Thomas \cite{ST}. Proposition \ref{Weights are nonnegative} gives the conceptual proof of this fact. Namely, semi-invariants defined on presentation spaces are necessarily sign coherent.

In future work, we plan to extend the results of this paper to modulated quivers with oriented cycles.

\section{Appendix A: Associated modulated quiver}\label{ss: Appendix}

In this appendix we discuss the problem of when a finite dimensional hereditary algebra over a field $K$ is Morita equivalent to the tensor algebra of its associated modulated quiver.

\begin{thm}\label{thm: when hereditary algebras are modulated}
$\Lambda$ is Morita equivalent to $T(Q,\cM)$ if and only if, for each arrow $i\to j$, the $F_i$-$F_j$-bimodule epimorphism
\begin{equation}\label{bimodule epimorphism}
	\Hom_\Lambda(P_j,rP_i)\onto M_{ij}=\Hom_\Lambda(P_j,rP_i/r^2P_i)
\end{equation}
has a section. (Recall that $F_i=\End_\Lambda(P_i)$.)
\end{thm}

\begin{proof}
This condition is necessary since it holds on the category of representations of $T(Q,\cM)$. Conversely, suppose the condition holds on $mod\text-\Lambda$. Choose a section $\sigma_{ij}:M_{ij}\to \Hom_\Lambda(P_j,rP_i)$ of \eqref{bimodule epimorphism} for every $i\to j$ in $Q_1$. For every $\Lambda$-module $X$, let $X_i=\Hom_\Lambda(P_i,X)$. This is a right $F_i$-module. For each arrow $i\to j$ in $Q_1$, define the morphism $X_i\otimes_{F_i}M_{ij}\to X_j$ to be the composition:
\[
	X_i\otimes_{F_i}M_{ij}\xrarrow{r\otimes \sigma_{ij}}\Hom_\Lambda(rP_i,rX)\otimes_{F_i} \Hom_\Lambda(P_j,rP_i)\xrarrow c\Hom_\Lambda(P_j,rX)\into \Hom_\Lambda(P_j,X)=X_j
\]
where $r:\Hom_\Lambda(P_i,X)\to \Hom_\Lambda(rP_i,rX)$ is the restriction map and $c$ is composition. Since each morphism in this sequence is natural in $X$, this defines a functor
\[
	\varphi:mod\text-\Lambda\to Rep(Q,\cM)
\]
which is clearly exact and faithful since it takes nonzero objects to nonzero objects.

We claim that $\varphi P_i$ is the projective cover $P_i^T$ of $S_i$ in $Rep(Q,\cM)$. This follows by induction on the length of $P_i$ and the fact that the structure maps $c(r\otimes \sigma_{ij}):M_{ij}\to \Hom_\Lambda(P_j,rP_i)$ of $\varphi P_i$ are, together, adjoint to the isomorphism $\bigoplus_j M_{ij}\otimes_{F_j}P_j\cong rP_i$.

Thus, $\Hom_\Lambda(P_i,X)=X_i=\Hom_{T(Q,\cM)}(P_i^T,\varphi X)$ and it follows that $\varphi$ is an equivalence between the full subcategories of projective objects of $mod\text-\Lambda$ and $Rep(Q,\cM)$. Being exact,  $\varphi$ extends to an equivalence of the module categories.
\end{proof}

\begin{eg}
Let $L=\FF_2(t)$ with subfields $K=\FF_2(t^4)\subset F=\FF_2(t^2)\subset L$. We have a short exact sequence of $L$-bimodules:
\begin{equation}\label{eq: nonsplit sequence of L-bimodules}
	0\to L\otimes_FL\xrarrow jL\otimes_KL\xrarrow pL\otimes_FL\to 0
\end{equation}
where $j$ sends $1\otimes 1$ to $t^2\otimes 1+1\otimes t^2$ and $p$ takes $1\otimes 1$ to $1\otimes 1$. This sequence does not split since $L\otimes_KL$ is indecomposable as an $L$-bimodule. This follows from the $L$-algebra isomorphism $\varphi: L[X]/(X^4)\to L\otimes_\kk L$ given by $\varphi(X)=t\otimes 1+1\otimes t$ where we consider $L\otimes_\kk L$ as an $L$-algebra using $L\otimes 1$.

Let $\Lambda$ be the tensor algebra of the modulated quiver 
\[
\xymatrixrowsep{10pt}\xymatrixcolsep{20pt}
\xymatrix{%begin xy matrix
&F_2 \ar[dr]^{M_{23}}&&& & F\ar[dr]^L\\
F_1\ar[ur]^{M_{12}}\ar[rr]^{\widetilde M_{13}}&&F_3&=& L\ar[ur]^L\ar[rr]^{L\otimes_KL}&&L
	}%end xy matrix
\]
modulo the relation that the composition $L\otimes_FL$ of the top two arrows is identified with the image of $j$ in $L\otimes_KL$. Then $\Lambda$ is hereditary since the radical of each projective module is projective, e.g., $rP_1\cong P_2\oplus P_3^2$. However, the bimodule morphism $\widetilde M_{13}=\Hom_\Lambda(P_3,rP_1)\onto M_{13}$ is not split because it is equal to the map $p$ in \eqref{eq: nonsplit sequence of L-bimodules}. By Theorem \ref{thm: when hereditary algebras are modulated}, $\Lambda$ is not Morita equivalent to the tensor algebra of its associated modulated quiver.
\end{eg}

% Section:

%\newpage
%%%%%%%%%%%%%%%%%%%%%%%%%%
%
%                Section  {Appendix on reduced norm}
%
%%%%%%%%%%%%%%%%%%%%%%%%%%

\section{Appendix B: Reduced norm}\label{ss: Appendix B}

This appendix reviews the definition and properties of the reduced norm \cite{Jacobson} and uses them to compare the determinantal weight with the ``true weight'' of a semi-invariant on presentation spaces {as claimed in Remark \ref{rem: det weight and reduced norm}.} We assume that $K$ is an infinite field.

\subsection{Definitions}

For $A$ a finite dimensional algebra over $K$, the \emph{general element} of $A$ is
\[
	a(\xi)=\sum \xi_iu_i \in A\otimes_K K(\xi)
\]
where $u_1,\cdots,u_n$ is a vector space basis for $A$ over $K$ and $\xi_1,\cdots,\xi_n$ are a transcendence basis for $K(\xi)=K(\xi_1,\cdots,\xi_n)$. Let 
\[
	m_{a(\xi)}(\lambda)=\lambda^m+c_1(\xi)\lambda^{m-1}+\cdots +c_m(\xi)\quad \in K(\xi)[\lambda]
\]
be the minimal polynomial of $a(\xi)$ over $K(\xi)$. The degree $m$ of $m_{a(\xi)}(\lambda)$ is called the \emph{degree} of $A$ over $K$. We call it the \emph{reduced degree} in cases where the word ``degree'' is already defined as in the case of field extensions. 

It is easy to see that the reduced degree of a finite separable extension of $K$ is equal to its vector space dimension over $K$ (the usual notion of degree). However, this is not true in general for inseparable extensions and division algebras.

If $D$ is a finite dimensional division algebra over its center $C$ then $\dim_CD=d^2$ where $d$ is the degree of $D$ over $C$. Furthermore, there is an open dense subset of $D$ consisting of all elements $b\in D$ so that $C(b)$ is a separable field extension of $C$ of degree $d$. Each of these is called a \emph{maximal separable subfield} of $D$.% (See [Jacobson].)

\begin{eg}
Let $A=\HH$ and $\kk=\RR$. The minimal polynomial of the general element $a=t+xi+yj+zk\in \HH$ is $m_a(\lambda)=\lambda^2-2t\lambda+t^2+x^2+y^2+z^2$. So, $\HH$ has degree 2 over $\RR$. For any $b\in \HH$ which is not in $\RR$, $\RR(b)\cong \CC$ is a maximal (separable) subfield of $\HH$.
\end{eg}

\begin{lem}\cite{Jacobson}
$m_{a(\xi)}(\lambda)$ is a polynomial in $\xi_1,\cdots,\xi_n,\lambda$ and $c_j(\xi)\in K[\xi]$ is a homogeneous polynomial of degree $j$ in the variables $\xi_i$.
\end{lem}

The \emph{reduced characteristic polynomial} of $b\in A$ is the specialization of $m_{a(\xi)}(\lambda)$ given by
\[
	m_b(\lambda)=\sum_{i=0}^m c_i(b_1,\cdots,b_n)\lambda^{m-i}\in K[\lambda]
\]
where $b=\sum b_i u_i$, $b_i\in K$ and $c_0=1$. We will use the notation $c_i(b)=c_i(b_1,\cdots,b_n)$.

\begin{prop}\cite{Jacobson}
\begin{enumerate}
\item[(0)] $m_b(\lambda)$ depends only on $b\in A$. (The coefficients $c_i(b)$ are independent of the choice of basis $u_1,\cdots,u_n$.)
\item $m_b(b)=0$. Equivalently, the minimal polynomial $\mu_b(\lambda)$ of $b$ is a factor of $m_b(\lambda)$.
\item Every root of $m_b(\lambda)$ is a root of $\mu_b(\lambda)$.
\item The set of all $b\in A$ for which $m_b(\lambda)$ is the minimal polynomial of $b$ is an open dense subset of $A$.
\item $m_b(\lambda)$ is invariant under extension of scalars, i.e., $m_b(\lambda)=m_{b\otimes1}(\lambda)$ if $b\otimes 1\in A\otimes_KL$ is the image of $b$ for any extension field $L$ of $K$.
\end{enumerate}
\end{prop}

The following observation follows easily from Properties (2) and (3).

\begin{lem}\label{lem: reduced degree of purely inseparable extension}
The reduced degree of a finite purely inseparable extension $F$ of $K$ is the smallest power $q=p^\mu$ of $p=char\,K$ so that $F^q\subseteq K$. Furthermore the reduced characteristic polynomial is $m_b(\lambda)=\lambda^q-b^q$ for every $b\in F$.
\end{lem}

\begin{eg}%[Exercise from Lang's Algebra]
Let $A=\FF_p(s,t)$ and $\kk=\FF_p(s^p,t^p)$. Then $a^p\in\kk$ for any $a\in A$ and the minimal polynomial of the general element $a\in A$ is $m_a(\lambda)=\lambda^p-a^p$. So, the reduced degree of $A$ over $\kk$ is $p$ although $A$ is a field extension of $\kk$ of degree $p^2$.
\end{eg}

\begin{defn}\label{def: reduced norm} The \emph{reduced norm} $\overline n:A\to K$ is defined to be the homogeneous polynomial function of degree $m$, the degree of $A$ over $K$, given on any $b\in A$ by
$
	\overline n(b)=(-1)^m c_m(b).
$
\end{defn}
The main properties of the reduced norm are the following.
\[
	\overline n(ab)=\overline n(a)\overline n(b)
,\quad	\overline n(1)=1.
\]
Any polynomial function $\chi:A\to K$ satisfying these two properties will be called a \emph{character} on $A$. Another easy consequence of Properties (2) and (3) is the following. If $A,B$ are finite dimensional algebras over $K$ and $(a,b)\in A\times B$, then $
	m_{(a,b)}(\lambda)=m_a(\lambda)m_b(\lambda)
$. This implies in particular that the degree of $A\times B$ over $K$ is the sum of the degrees of $A,B$ over $K$. Also the reduced norm over $A\times B$ is the product:
\[
	\overline n_{A\times B}(a,b)=\overline n_A(a)\overline n_B (b).
\]

\subsection{Theorems related to this paper}
%The main theorem relevant to our paper is the following.

\begin{thm}\label{thm relating determinant and reduced norm}
Let $D$ be a finite dimensional division algebra over $\kk$ which has degree $d$ over its center $C$ and suppose that $C$ has reduced degree $c$ over $\kk$. Then

\emph{(a)} For any $k\ge1$, $M_k(D)$ has degree $dkc$. 

\emph{(b)} Any character $M_k(D)\to \kk$ is a nonnegative power of the reduced norm.
\end{thm}

\begin{proof} 

We first compute the degree of $M_k(D)$ over $K$. Let $L$ be the maximal separable subfield of $D$. Then $L$ is separable over $C$ of degree $d$ and it is well-known that $M_k(D)\otimes_CL\cong M_{dk}(L)$. Let $E$ be the separable closure of $\kk$ in $L$. Then $F=E\cap C$ is the separable closure of $\kk$ in $C$ and $L=CE$. Let $s=[F:\kk]$. Then $c=q s$ where $q=p^\mu$ is the reduced degree of $C$ over $F$. By Lemma \ref{lem: reduced degree of purely inseparable extension}, the reduced degree of $L$ over $E$ is also $q$ and $q$ is minimal so that $L^q\subseteq E$. Let $S$ be the splitting field of $E$ over $K$. Then $C\otimes_FS\cong CS$ is a separable field extension of $CE=L$. So, 
\[
	M_k(D)\otimes_FS=M_k(D)\otimes_C C\otimes_FS=M_k(D)\otimes_C L\otimes_LCS\cong M_{dk}(CS).
\]

\underline{Claim} The degree of $M_{dk}(CS)$ over $S$ is $qdk$ and the reduced norm $M_{dk}(CS)\to S$ is the $q$-th power of the determinant over $CS$. \vs2

Proof: Any $a\in M_{dk}(CS)$ satisfies its characteristic polynomial $f(\lambda)=\det(\lambda-a)\in CS[\lambda]$ with degree $dk$. Then $f(\lambda)^q$ is a polynomial in $S[\lambda]$ of degree $qdk$ satisfied by $a$. So, the degree of $M_{dk}(CS)$ over $S$ is $\le qdk$. Now consider the inclusion of the diagonal matrices:
\[
	CS^{dk}=CS\times\cdots\times CS\into M_{dk}(CS)
\]
Since the general element of $CS$ has degree $q$ over $S$, the general element of $CS^{dk}$ has degree $qdk$ over $S$. So, the degree of $M_{dk}(CS)$ over $S$ is $\ge qdk$. So, it is equal to $qdk$. Furthermore, the reduced characteristic polynomial is $\det(\lambda-a)^q$ and the reduced norm is $\det(a)^q$.
\vs2

%We return to the proof of the theorem. 

(a) Since $S$ is the splitting field of $E$ over $K$ and $F$ is an intermediate field, we have $F\otimes_KS\cong S^s$ where $s=[F:K]$. Since (reduced) degree is invariant under extension of scalars, the degree of $M_k(D)$ over $K$ is equal to the degree of $M_k(D)\otimes_KS$ over $S$. But
\[
	M_k(D)\otimes_KS=M_k(D)\otimes_FF\otimes_KS=M_k(D)\otimes_FS^s=M_{dk}(CS)^s
\]
which has degree $s$ times the degree of $M_{dk}(CS)$ over $S$. By the claim above this is $s$ times $qdk$ which is $dkqs=dkc$ proving (a).%the first statement of the theorem.

%(b) Thus, the reduced norm  $\overline n:M_k(D)\to K$ has degree $qkc=qdsk$. Now c

(b) Consider any character $\chi:M_k(D)\to K$. We note that arbitrary (polynomial) characters must be homogeneous polynomials. By extending scalars we get a character \[\chi_S:M_k(D)\otimes_KS\cong M_{dk}(CS)^s\to S\]which must be a product of $s$ characters $(\chi_S)_i:M_{dk}(CS)\to S$. By symmetry given by the action of $Gal(S/K)$, these $s$ characters are equal. By restriction to diagonal matrices we get a character $CS^{dks}\to S$. But a character on $CS^{dks}$ is a product of characters one for each factor. By symmetry, these characters must all be equal: $\chi_S|CS^{dks}=(\chi_0)^{dks}$. But each character $\chi_0:CS\to S$ is a power of the reduced norm $\overline n_{CS}:CS\to S$ since $\chi_0(x)=x^m$ and this lies in $S$ only when $m$ is a multiple of $q$, say $m=qt$, $\chi_0=\overline n_{CS}^t$. Therefore,
\[
	\chi_S|CS^{dks}=(\chi_0)^{dks}=\overline n_{CS}^{tdks}
\]
which has degree equal to $qtdks$. When $\chi$ is the reduced norm $\overline n$ we get $t=1$. Therefore, in general we get $(\chi_S)_i=\overline n^t$ when restricted to the diagonal matrices where $\overline n$ is the reduced norm of $M_{dk}(CS)$ over $S$. However, any invertible matrix is equivalent to a diagonal matrix under row and column operations which are given by multiplication by elements of the commutator subgroup of $GL(dk,CS)$. Since $S^\ast$ is abelian, each group homomorphism $(\chi_S)_i:GL(dk,CS)\to S^\ast$ is uniquely determined by its restriction to diagonal invertible matrices. So, $\chi_S=\overline n^t$ for all elements of $GL(dk,CS)^s$. Since this is an open dense subset of $M_{dk}(CS)^s$, $\chi_S=\overline n^t$ as homogeneous polynomials over $S$. But both polynomials have coefficients in $K$. So, they give $\chi=\overline n^t$ as characters $M_k(D)\to K$.
\end{proof}

\begin{rem}\label{rem 6.2.2: every character is a (frac) power of det}
Theorem \ref{thm relating determinant and reduced norm} implies that every character $M_k(D)\to \kk$ is a nonnegative fractional power of the $\kk$-determinant $det_\kk$: $det_\kk=\overline n^{f/dc}$ where $f=\dim_\kk D$. Thus, the ``true weight'' of a semi-invariant with determinantal weight $\beta$ is the vector whose $i$-th coordinate is $\beta_if_i/d_ic_i$ where $d_ic_i$ is the (reduced) degree of $F_i$ over $K$. In particular, if $m$ is the least common multiple of the integers $f_i/d_ic_i$ then the $m$-th power $\sigma^m$ of any semi-invariant on a presentation space $\Hom_\Lambda(P(\gamma_1),P(\gamma_0))$ has determinantal weight.
\end{rem}

\begin{cor}\label{cor: reduced exchange matrix is integral}
Suppose that $F_1,F_2$ are division algebras over $K$ of dimensions $f_1,f_2$ and degrees $n_1,n_2$ over $K$. Let $M$ be an $F_1$-$F_2$-bimodule with $\dim_KM=m$. Then
\[
	\frac{mn_1}{f_1n_2}\ , \frac{mn_2}{f_2n_1}\ \in \ZZ.
\]
%are integers.
\end{cor}

\begin{proof}
The reduced norm gives a character
\[
	\End_{F_1}(M)\cong M_{m/f_1}(F_1)\xrarrow{\overline n_1}K
\]
which is polynomial of degree $mn_1/f_1$. Composing with the inclusion $F_2\into \End_{F_1}(M)$ we get a character $\chi:F_2\to K$ of degree $mn_1/f_1$. By Theorem \ref{thm relating determinant and reduced norm}, $\chi$ is an integer power of the reduced norm $\overline n_2:F_2\to K$ which has degree $n_2$. Therefore $n_2$ divides $mn_1/f_1$ making $mn_1/f_1n_2$ an integer. The other case is similar.
\end{proof}

%\newpage

\begin{defn}
Let $\Lambda$ be a finite dimensional hereditary algebra over a field $K$. Let $B_\Lambda=L^t-R$ be the exchange matrix of $\Lambda$. Define the \emph{reduced exchange matrix} of $\Lambda$ to be
\[
	\overline B_\Lambda= ZB_\Lambda Z^{-1}
\]
where $Z$ is the diagonal matrix with entries $z_i=f_i/n_i$ where $n_i$ is the degree of $F_i$ over $K$ and $f_i=\dim_KF_i$. The entries of $\overline B_\Lambda$ are 
\[
	\overline b_{ij}=\frac{n_j}{f_jn_i}(\left<e_j,e_i\right>-\left<e_i,e_j\right>)
\]
where $e_i$ are the unit vectors. Since $|\left<e_j,e_i\right>|$ is the dimension of an $F_i\text-F_j$-bimodule, $\overline b_{ij}$ are integers by Corollary \ref{cor: reduced exchange matrix is integral}. Given a cluster tilting object $T$ with exchange matrix $B_T$ and $c$-matrix $C_T$, we define the \emph{reduced exchange matrix} and the matrix of \emph{reduced $c$-vectors} by $\overline B_T=ZB_TZ^{-1}$ and $\overline C_T=ZC_TZ^{-1}$.
\end{defn}

Since mutation of exchange matrices and extended exchange matrices commutes with conjugation, $\overline B_T$ and $\overline C_T$ have integer coordinates and are obtained from $\mat{\overline B_\Lambda\\ I_n}$ by mutation. We claim that the reduced $c$-vectors are the reduced weights of the reduced norm semi-invariants which we now define.

\begin{defn}The \emph{reduced norm semi-invariant} $\overline\sigma_{\beta}$ is the polynomial function
\[
	Pres_\Lambda(\gamma_1,\gamma_0)=\Hom_\Lambda(P(\gamma_1),P(\gamma_0))\to K
\]
which sends $f:P(\gamma_1)\to P(\gamma_0)$ to the reduced norm of
\[
	\Hom(f,1):\Hom_\Lambda(P(\gamma_0),M_\beta)\to \Hom_\Lambda(P(\gamma_1),M_\beta)
\]
considered as a linear map of $F_\beta$-vector spaces. We define the \emph{reduced weight} of a semi-invariant $\sigma$ on presentation space $Pres_\Lambda(\gamma_1,\gamma_0)$ to be the vector $w\in\NN^n$ so that $\sigma(gfh)=\prod \overline n_i(g)^{w_i}\sigma(f)\overline n_i(h)^{w_i}$ where $\overline n_i(g)$ is the reduced norm of the $GL(\gamma_{0i},F_i)$-component of $g\in\Aut_\Lambda(P(\gamma_0))$ and similarly for $\overline n_i(h)$. 
\end{defn}

\begin{lem} 
The reduced weight of the reduced norm semi-invariant $\overline\sigma_\beta$ is
\[
	\overline\beta=\frac1{z_\beta}(z_1\beta_1,z_2\beta_2,\cdots,z_n\beta_n).
\]
\end{lem}

\begin{proof}
By Remark \ref{rem 6.2.2: every character is a (frac) power of det}, we have $\sigma_\beta=\overline\sigma_\beta^{z_\beta}$ where $z_\beta=\dim_KF_\beta/\deg_KF_\beta$. Since the det-weight of $\sigma_\beta$ is $\beta$ we have:
\begin{eqnarray*}
	\overline\sigma_\beta(gfh) & = & \sigma_\beta(gfh)^{1/z_\beta}\\
	&=& \prod \chi_i(g)^{\beta_i/z_\beta}\sigma_\beta(f)^{1/z_\beta}\chi_i(h)^{\beta_i/z_\beta}\\
	&=& \prod \overline n_i(g)^{n_i\beta_i/z_\beta}\overline\sigma_\beta(f) \overline n_i(g)^{n_i\beta_i/z_\beta}
\end{eqnarray*}
where $\chi_i(g)=\overline n_i(g)^{z_i}$ is the det-weight of the $GL(\gamma_{0i},F_i)$-component of $g\in\Aut_\Lambda(P(\gamma_0))$. So, the reduced weight of $\overline\sigma_\beta$ is $(n_i\beta_i/z_\beta)=\overline\beta$.
\end{proof}

In the notation of Corollary \ref{cor: HT equation} we have the following.

\begin{lem}\label{lem: reduced HT equation}
For any cluster tilting object $T$ of $\Lambda$ we have
\[
	V^t\overline E\,\overline\Gamma_T=N
\]
where $N=DZ^{-1}$ is the diagonal matrix with entries $n_i$, $\overline E=EZ^{-1}=LN$, $\overline\Gamma_T=Z\Gamma_TZ^{-1}$.% with $V,E,\Gamma_T$ being the 
\end{lem}

\begin{proof}
By Corollary \ref{cor: HT equation} we have: $V^t\overline E\,\overline\Gamma_T=V^tE\Gamma_TZ^{-1}=DZ^{-1}=N$.
\end{proof}

We can now restate the $c$-vector theorem in terms of reduced $c$-vectors.

\begin{thm}[Reduced Norm $c$-vector Theorem]\label{reduced c-vector theorem}
The reduced $c$-vectors associated to a cluster tilting object $T$ are
\[
	\overline c_j=-\varepsilon_j\overline\beta_j
\]
where $\overline\beta_j$ is the reduced weight of the reduced norm semi-invariant $\overline\sigma_{\beta_j}$.
\end{thm}

\begin{proof}
Since conjugation of exchange matrices and $c$-matrices commutes with mutation, given that $c_j$ is the $j$-th $c$-vector of the object $T$, the reduced vector $\overline c_j$ is the $j$-th $c$ vector of $T$ using $\overline B_\Lambda$ as initial exchange matrix. Since $c_j=-\varepsilon_j\beta_j$ by Theorem \ref{c-vector thm}, reduction of both sides, using the fact that $z_j=z_{\beta_j}$, gives $\overline c_j=-\varepsilon_j\overline\beta_j$.
\end{proof}

\begin{eg}
Consider the $\RR$-modulated quiver $\HH\ot \CC\ot \CC$. This has $9$ indecomposable modules giving the same picture as Figure \ref{Fig C3}. Using the same label for these modules as in Figure \ref{Fig C3} we have:
\[
\begin{array}{cccc}
	\text{label}&\beta & z_\beta & \overline\beta=\frac1{z_\beta}(\beta_1,\beta_2,2\beta_3)\\
	\hline
	S_1 & (1,0,0) & 2 & (1,0,0)\\
	P_2 & (1,1,0) & 1 & (2,1,0)\\
	P_3 & (1,1,1) & 1 & (2,1,1)\\
	Y & (1,2,0) & 2 & (1,1,0)\\
	X & (1,2,1) & 1 & (2,2,1)\\
	S_2 & (0,1,0) & 1 & (0,1,0)\\
	Z_2 & (1,2,2) & 2 & (1,1,1)\\
	Z_1 & (0,1,1) & 1 & (0,1,1)\\
	S_3 & (0,0,1) & 1 & (0,0,1)
\end{array}
\]
As an example, take the injective module $Z_1$. This has a determinantal semi-invariant of det-weight $\beta=(1,2,2)$ since a presentation for $Z_1$ is $P_2\oplus P_1\to P_3\oplus P_2\to Z_1$. We take homomorphisms to $Z_2$ to get:
\begin{equation}\label{eq: example of reduced norm semi-invariant}
	\Hom_\Lambda(P_3\oplus P_2,Z_2)=\CC^2\oplus \CC^2\to \Hom_\Lambda(P_2\oplus P_1,Z_2)=\CC^2\oplus \HH
\end{equation}
The determinantal semi-invariant $\sigma_\beta$ is given by considering this as an isomorphism of $8$-dimensional real vector spaces and taking determinant. This has determinantal weight $(1,2,2)$ since the automorphism of $P_3$ given by $z=a+bi\in\CC^\ast$ has real determinant $\left|\begin{matrix}a & b\\-b&a\end{matrix}\right|=a^2+b^2$ and multiplies the $8\times 8$ determinant by $(a^2+b^2)^2$ (since it multiplies the first two $\CC$ coordinates) which is the second power of the det-weight $|z|^2$ of $z$. Similarly any $z\in\Aut(P_2)$ also changes the $8\times 8$ determinant by $|z|^4$ making the det-weight of $\sigma_\beta$ equal to $(?,2,2)$. The first coordinate of the det-weight is 1 since $h\in \Aut(P_1)$ changes the $8\times 8$ determinant by $|h|^4$ which is the det-weight of $h$.

The reduced norm semi-invariant $\overline\sigma_\beta$ is given by considering \eqref{eq: example of reduced norm semi-invariant} as an isomorphism of $2$-dimensional vector spaces over $\HH$ and taking the reduced norm over $\HH$ which is the square root of the real determinant. So, any automorphism of $P_3$ or $P_2$ given by $z\in\CC^\ast$ will change the reduced norm semi-invariant by $|z|^2$ which is the norm of $z$. Also, any automorphism of $P_1$ given by $h\in \HH^\ast$ will change $\overline\sigma_\beta$ by $|h|^2=\overline n(h)$. So, the reduced weight is $(1,1,1)$.
\end{eg}

\section*{Acknowledgements}

The last two authors gratefully acknowledge the support of National Science Foundation. The first author was supported by the National Security Agency.  The second author was supported by the Simons Foundation. The first two authors also acknowledge support of the NSF at the beginning of this project many years ago. The first and third authors thank Faculty of Mathematics and Computer Science of Nicolaus Copernicus University in Torun, Poland and the University of Iowa for their hospitality in September 2013 and November 2014, where the results were announced. The third author thanks University of Syracuse, University of Barcelona and Centro de Investigaci\'on en Matem\`aticas (CIMAT) in Guanajuato, Mexico for the invitation to present the results of this paper and application on April 11, May 26 and June 26, 2015. The first, third and fourth authors are very grateful to the University of Connecticut for hosting (and to National Science Foundation for sponsoring) a very enjoyable and productive International Conference in Representation Theory and Commutative Algebra (ICRTCA) in honor of the fourth author on April 24-27, 2015. The first and third authors thank the Centre de Recerca Matem\'atica (CRM) at the University of Barcelona for their hospitality during May 2015 where first versions of the appendices of this paper were written. The second author gratefully acknowledges the support provided by the SFB 1085 ÔHigher InvariantsÕ at the University of Regensburg while on sabbatical in the fall of 2014, and funded by the Deutsche Forschungsgemeinschaft (DFG). The authors also had very useful conversations and communications with Hugh Thomas, Nathan Reading, Calin Chindris, Helmut Lenzing, Stephen Hermes, Thomas Br\"ustle, Alex Dugas and Ernst Dieterich.

%%%%%%%%%%%%%%%%%%%%%%%%%%%%%%%%%%%
\end{document}